\documentclass[a4paper,11pt,reqno]{amsart}

\usepackage{amsmath}
\usepackage{amssymb}
\usepackage{xspace,epsfig}

\usepackage[pagewise]{lineno}
\usepackage{enumitem}
\usepackage{hyperref}
\usepackage{amsmath,amsfonts,amssymb,mathrsfs,amsthm}

\numberwithin{equation}{section}

\newcommand{\R}{\mathbb R}
\newcommand{\C}{\mathbb C}
\newcommand{\N}{\mathbb N}

\newcommand{\be}{\begin{equation}}
\newcommand{\ee}{\end{equation}}
\newcommand{\ba}{\begin{eqnarray}}
\newcommand{\ea}{\end{eqnarray}}

\newcommand{\p}{\partial}

\newcommand{\Comp}{\mathcal{C}}
\newcommand{\yx}{Y^{x}}
\newcommand{\yt}{Y^{t}}
\newcommand{\yxt}[2]{Y^{x,t}_{#1,#2}}
\newcommand\E[2]{E_{#1,#2}}
\newcommand\Y[2]{Y_{#1,#2}}
\newcommand{\dxf}{\widetilde{D}_{x}}
\newcommand{\dtf}{\widetilde{D}_{t}}

\def\e{{\varepsilon}}

\newcommand\mub{\widetilde{\mu}}

\newcommand\bna{\begin{eqnarray*}}
\newcommand\ena{\end{eqnarray*}}

\newcommand\bnan{\begin{eqnarray}}
\newcommand\enan{\end{eqnarray}}

\newcommand\bnp{\begin{proof}}
\newcommand\enp{\end{proof}}

\newcommand\bneq{\begin{eqnarray*}\left\lbrace \begin{array}{rcl}}
\newcommand\eneq{\end{array} \right.\end{eqnarray*}}
\newcommand\bneqn{\begin{eqnarray}\left\lbrace \begin{array}{rcl}}
\newcommand\eneqn{\end{array} \right.\end{eqnarray}}

\newcommand\bni{\begin{itemize}}
\newcommand\eni{\end{itemize}}

\newcommand{\yot}{(y_0,y_1,...,y_{N-1})}

\newcommand\nor[2]{\left\|#1\right\|_{#2}}

\newtheorem{nota}{Notation}[section]

\newtheorem{theorem}{Theorem}[section]
\newtheorem{proposition}[theorem]{Proposition}
\newtheorem{lemma}[theorem]{Lemma}
\newtheorem{corollary}[theorem]{Corollary}
\newtheorem{remark}{Remark}
\theoremstyle{remark}

\newtheorem{definition}{Definition}

\numberwithin{equation}{section}

\begin{document}
\title[Exact controllability of anisotropic 1D PDE]
{Exact controllability of anisotropic 1D partial differential equations  in spaces of analytic functions}

\author[C. Laurent]{Camille Laurent}
\address{CNRS UMR 9008, Universit\'e Reims-Champagne-Ardennes, Laboratoire de Math\'ematiques de Reims (LMR), Moulin de la Housse-BP 1039, 51687 REIMS cedex 2, France}
\email{camille.laurent@univ-reims.fr}

\author[I. Rivas]{Ivonne Rivas}
\address{Universidad del Valle, Departamento de Matem\'aticas, Cali, Colombia}
\email{ivonne.rivas@correounivalle.edu.co}

\author[L. Rosier]{Lionel Rosier}
\address{Universit\'e du Littoral C\^ote d'Opale, Laboratoire de Math\'ematiques Pures et Appliqu\'ees J. Liouville,
 B.P. 699, F-62228 Calais, France}
\email{Lionel.Rosier@univ-littoral.fr}

\date{}

\begin{abstract}
     
In this article, we prove a local controllability result for a general class of 1D partial differential equations on the interval $(0,1)$. The PDEs we consider take the form $\partial_t^N y=\zeta_M \partial_{x}^{M}y+f(x , y , \partial_{x} y,...,  \partial _x^{M-1} y)$ where $1\le N < M$, $\zeta_M\in \C ^*$, and $f$ is some linear or nonlinear term of lower order. In this context, we prove a local controllability result between states that are analytic functions. If some boundary conditions are prescribed, a similar local controllability result holds between analytic functions satisfying some compatibility conditions that are natural for the existence of smooth solutions of the considered PDE. 
The proof is performed by studying a nonlinear Cauchy problem in the spatial variable with data in some spaces of Gevrey functions and by investigating the relationship between the jet of space derivatives and the jet of time derivatives. We give various examples of applications, including the (good and bad) Boussinesq equation, the Ginzburg-Landau equation, the Kuramoto-Sivashinsky equation and the Korteweg-de Vries equation.

\end{abstract}
\maketitle

\vspace{0.3cm}

\textbf{2010 Mathematics Subject Classification: 35K40, 93B05}  

\vspace{0.5cm}

\textbf{Keywords:}  anisotropic PDEs; exact controllability; ill-posed problems; Gevrey functions; 
reachable states. 

\tableofcontents
\section{Introduction}

For $M$, $N\in \N^*:=\N\setminus \{0\}$ fixed with $M>N$ and $y$  a function defined on $[0,1]\times [0,T]$, with value in $\R$, we consider the  abstract dynamical system 
\begin{eqnarray}
\partial _t^N y = P\, y + f(x , y , \partial_{x} y,...,  \partial _x^{M-1} y),&& x\in [0,1],\ t\in [0,T],  \label{W1a}\\ 
B\yx(0,t)=0,&& t\in [0,T],  \label{W1b}\\
\yt(x,0) = Y_0(x),&&  x\in [0,1] \label{W1c},
\end{eqnarray}

with
\bnan
Y^x(x,t) :=(y(x,t),\partial_x y(x,t),...,\partial_x^{M-1}y(x,t)),  \label{Yx}\\
Y^{t}(x,t) :=(y(x,t),\partial_ty(x,t),...,\partial_t^{N-1}y(x,t)),\label{Yt}
\enan
\bnan
\label{formP}
P:=\sum_{j=0}^{M}\zeta _{j}\partial_{x}^{j},
\enan
where  $\zeta _{j}\in\R$ for $0\le  j \le M$ and $ \zeta _{M}\neq 0$,  $Y_0\in C^\infty ([0,1])^N$, $B\in \R ^{v\times M}$ is a fixed real matrix of size $ v\times M$, and
 $v\in \N$ is the number of boundary conditions that we require to be equal to zero. (If $v=0$,  it indicates that there is no boundary condition at $x=0$.) 
 Finally, we assume $f\in C^{\infty}(\R ^{M+1}; \R)$ and $f$ is analytic with respect to all its arguments in a \emph{neighborhood} of $\vec{0}_{\R^{M+1}}$. More precisely, we assume that 
\be
f(x,0,\ldots,0)=0, \quad \forall x\in (-4,4), \label{AB1}
\ee 
and 
\be
f(x,\vec{y}) :=  \sum_{ (\vec{p},r)\in \N^{M+1} }a_{\vec{p},r}\vec{y}^{\vec{p}} x^r=\sum_{ (\vec{p},r)\in \N^{M+1} }a_{\vec{p},r} y_0^{p_0} y_1^{p_1}\dots y_{M-1}^{p_{M-1}} x^r   \label{AB2}
\ee
 with  $\vec{y}=(y_0,y_1,...,y_{M-1})$, $(x,\vec{y})\in (-4,4)^{M+1}$, and $\vec{p}=(p_0, \dots,p_{M-1})\in \N^M$ where the coefficients 
 $a_{\vec{p},r} $ are  such that 
\be
|a_{\vec{p},r} | \le \frac{C_a}{b^{|\vec{p}|} b_2^r},\qquad  \forall r\in \N, \ \forall \vec{p}\in \N^M\label{AB3},
\ee
for some constants
\be
C_a>0 , \quad b >4,\  \textrm{ and }  b_2>4. \label{AB4}
\ee
Note that $a_{\vec{0},r}=0$ for all $r \in \N$ by \eqref{AB1}. For $\vec{p}\in \N^M$, we define 
\[
A_{\vec{p}}(x):=\sum_{r\in \N} a_{\vec{p},r} x^r, \quad |x|<b_2. 
\]
We infer from \eqref{AB2} and \eqref{AB3} that
\begin{eqnarray}
\label{formfbis}
f(x,\vec{y}) &=& \sum_{\tiny\begin{array}{c}\vec{p}\in \N^M \\ |\vec{p}|>0\end{array}} A_{\vec{p}}(x)\vec{y}^{\vec{p}}= \sum_{\tiny\begin{array}{c}\vec{p}\in \N^M \\ |\vec{p}|>0\end{array}} A_{\vec{p}}(x)y_0^{p_0} y_1^{p_1}\dots y_{M-1}^{p_{M-1}},  
\end{eqnarray}  

\begin{gather}\label{estimAp}
|A_{\vec{p}}(x)| \le \frac{C_a}{b^{|\vec{p}|}} \, \frac{1}{1-\frac{|x|}{b_2}}, \quad |x|<b_2.
\end{gather}
Among the many physically relevant instances of \eqref{W1a} satisfying \eqref{AB1}-\eqref{AB4}, we can mention

\begin{enumerate}
\item the {\em Korteweg-de Vries (KdV)  equation}
$$\partial _t y=\partial _x^3 \, y +\partial_xy+ y\partial_xy ;$$ 
\item the  {\em ``good'' $(-)$ or ``bad''  $(+)$ Boussinesq equation}
$$\partial _t^2 y =\pm \partial _x ^4 y +\partial _x ^2 y- \partial_x^2(y^2) ;$$
\item the {\em Kuramoto-Sivashinsky (KS) equation}
$$\partial _t y + \partial _x ^4 y +\partial _x ^2 y+y\partial_xy=0.$$
\end{enumerate}
With a few modifications in the framework, we can also treat
\begin{enumerate}[resume]
\item the {\em complex Ginzburg-Landau (GL) equation}
$$ \partial _t y=e^{i\theta}\partial _x^2 \, y +e^{i\varphi} |y|^2y \quad \text{where} \quad \theta, \varphi\in \R. $$
\end{enumerate}
The exact controllability result has to be stated in a space of analytic functions (see \cite{MRRreachable} for the linear heat equation). For given $R>1$ and $C>0$, we denote by ${\mathcal N} _{R,C}$ and ${\mathcal R} _{R,C}$ the sets 
\begin{equation}
{\mathcal N}_{R,C} := \left\{ (\alpha _n)_{n\ge 0} \in \C^\N ; |\alpha _n| \le C \frac{n!}{R^n} , \, \ \forall n\ge 0\right\}\subset  \C^{\N}, 
\label{WW}
\end{equation}

\begin{equation}
{\mathcal R}_{R,C} := \left\{  z:[-1,1]\to \C: \ \exists (\alpha _n)_{n\ge 0} \in {\mathcal N}_{R,C} \textnormal{ with }  
 z(x) =\sum_{n=0}^\infty \alpha _n \frac{x^n}{n!}, \, \  \forall x\in [-1,1]\right\}.   
\label{WWW}
\end{equation}
 Let  us denote by $H(\Omega)$ the space of holomorphic functions in $\Omega$,  and let us introduce the Hardy space $H^\infty_R:=H (B(0,R))\cap L^\infty (B(0,R))$,  which is a Banach space for the 
norm $\Vert \cdot\Vert _{L^\infty ( B(0,R )) }$ (see \cite{rudin2}). Let 
\[
{\mathcal B}_{R,C} :=\{  z:[-1,1]\to \C ;  \ \exists f\in H^\infty_R,\ \Vert f\Vert_{L^\infty(B(0,R))} \le C, \ f_{\vert\,  [-1,1]}=z\}.
\]
Observe that $${\mathcal{B}}_{R,C} \subset {\mathcal R}_{R,C}\subset {\mathcal B}_{r,C  ( 1-\frac{r}{R} )^{-1} } \textrm{  for }  1<r<R \textrm{ and }  C>0.$$
For the proof, see below Lemma \ref{lemmasets}.\\


We say that a function $h\in C^\infty ([t_1,t_2])$ is {\em Gevrey of order $s\ge 0$ on $[t_1,t_2]$}, and we write  
$h\in G^s([t_1,t_2])$, if there exist
some positive constants $C,R$ such that 
\[
|\partial _t ^p h(t) | \le C \frac{(p!)^s}{R^p},\quad \forall t\in [t_1,t_2], \ \forall p\in \N. 
\]
Similarly, we say that a function $y\in C^\infty ([x_1,x_2]\times [t_1,t_2])$ is 
{\em Gevrey of order $s_1$ in $x$ and $s_2$ in $t$}, with $s_1,s_2\ge 0$, and we write $y\in G^{s_1,s_2}([x_1,x_2]\times [t_1,t_2])$, if there exist some positive constants $C,R_1,R_2$ such that 
\begin{align}
\label{defGev2D}
\vert \partial _x ^{p_1}\partial _t ^{p_2} y(x,t) \vert  \le C \frac{ (p_1!)^{s_1} (p_2!)^{s_2}}{R_1^{p_1} R_2^{p_2}},\quad 
\forall (x,t)\in [x_1,x_2]\times [t_1,t_2], \ \forall (p_1,p_2)\in \N ^2.
\end{align}

The suitable time Gevrey regularity in our situation is
\bna
\lambda:=\frac{M}{N}>1.
\ena
Before giving our results, we need to define a set of compatibility conditions. The initial data need to belong to a specific set to ensure the existence of smooth solutions issuing from these initial data. Indeed, the equation imposes some relations between the time derivatives of the solutions and the space derivatives of the initial data. Namely, we have the following property whose proof is constructive and mainly consists in taking derivatives in the PDE.
\begin{lemma}
\label{lmdefJlintro}
For any $l\in \N$, there exist a number $m=m(l)\in \N$ and a smooth application $J_{l}: [-1,1]\times (\R^{N})^{m(l)+1}\to \R^{M}$ such that for any 
solution $y\in C^\infty ([-1,1]\times [t_1,t_2])$ of
$\partial _t^N y = P\, y + f(x , y , \partial_{x} y,...,  \partial _x^{M-1} y)$, 
we have
\bnan
\label{propJlintro}
\partial_{t}^{l}\yx=J_{l}(x,\yt,\partial_{x}\yt,...,\partial_{x}^{m}\yt))\quad \textnormal{ on }[-1,1]\times [t_1,t_2].
\enan
\end{lemma}
\begin{definition}
Let $J_{l}$, $l\in \N$,  be the vector  functions defined in Lemma \ref{lmdefJlintro}. We define the following \textbf{compatibility set}
\begin{equation}\label{comp}
\Comp :=\left\{Y_{0}\in C^{\infty}([0,1])^{N}; \quad BJ_{l}( x,Y_{0},\partial_{x}Y_{0},...,\partial_{x}^{m(l)}Y_{0})\Big|_{ x=0}=0,\quad \forall l\in \N\right\}.
\end{equation}
\end{definition}
The compatibility set $\Comp$  plays an important role in the exact controllability of system \eqref{W1a}-\eqref{W1c}.
Since the PDE \eqref{W1a}  is time-invariant, we can check that the condition \eqref{comp} is the same at any time. In particular
\begin{itemize}
\item for any smooth solution $y$ of \eqref{W1a}-\eqref{W1b}, we have that $\yx(t)\in \Comp$ for any $t\in [0,T]$. (See below Lemma \ref{lmcompatibility}.)
\item if $y$ is a smooth solution to \eqref{W1a} such that $\yx(t)\in \Comp$ for any $t\in [0,T]$, then $y$ satisfies the boundary condition \eqref{W1b}.
(See below Lemma \ref{lmcompatibilityrecip}.)
\end{itemize}
If we want to consider the boundary controllability of the PDE \eqref{W1a}  subject to the boundary conditions \eqref{W1b}, it is therefore very natural to consider initial and final data in the space $\Comp$. We will derive a controllability result by considering small amplitude analytic functions in 
$\Comp$.

The main result in this paper is the following {\em local exact controllability result}. 
\begin{theorem}
\label{thm1}
Let $f=f(x,\vec{y})$ be as in \eqref{AB1}-\eqref{AB4} with $b, b_2>\hat R:=4M e^{(\lambda e)^{-1}}$ 
Let $R>\hat R$ and $T>0$. 
Then there exists some number $\hat C >0$ such that for all
 $Y_0,Y_1\in ({\mathcal R}_{R,\hat C})^{N}\cap \Comp$, there exists a smooth solution $y$ of 
\eqref{W1a}-\eqref{W1c} defined for all $(x,t)\in [0,1]\times [0,T]$ and satisfying $\yt(x,T)=Y_1(x)$ for all  $x\in [0,1]$. Furthermore, we have $y\in G^{1,\lambda}([0,1]\times [0,T])$.      
\end{theorem}
We stress that Theorem \ref{thm1} can be applied to {\bf any} PDE with less derivatives in time than in space, even if the 
corresponding initial boundary value problem is {\bf not well-posed}. For instance, the backward heat equation $\partial _t y=-\partial _x^2 y$ and the bad Boussinesq 
equation  $\partial _t^2 y =\partial _x ^4 y +\partial _x ^2 y- \partial_x^2(y^2)$ are concerned. 

It is difficult in general to describe explicitly $\Comp$  (see Section \ref{sectKdV} for the KdV equation). However, the set $\Comp$ can be precisely described in the following cases:
\begin{itemize}
\item If $B=0$ (i.e. no boundary conditions at $x=0$), then $\Comp=C^{\infty}([0,1])^{N}$ (i.e. all smooth initial data are allowed)
\item If $f=0$ (linear PDE with constant coefficients), then the compatibility set reads
\begin{multline*}
    \Comp=\Big\{Y_{0}=(y_0,y_1, \ldots,y_l,\ldots,y_{N-1}) \in C^{\infty}([0,1])^{N} \, \text{such that} \, BP^jY^{x,l}_0(0)=0, \\ \, \forall j\in \N, \forall l=0,\dots, N-1\Big\}
\end{multline*}
when  we denoted $Y^{x,l}_0(x):=(y_l(x), \dots, \partial_x^{M-1}y_l(x))$ as in \eqref{Yx}. We refer to Proposition \ref{propcompatlin} for a precise statement and for the proof.
\item if $M\in 2\N$ and $P=\sum_{j=0}^{M/2} \zeta_{2j}\partial_{x}^{2j}$ (i.e. $P$ contains only even derivatives), some parity arguments can be used under some symmetry assumptions about the non-linearity, as it is shown in the following proposition.  
\end{itemize} 

\begin{proposition}\label{propCompsym}
Assume that $M \in 2\N$ and $P=\sum_{j=0}^{M/2}\zeta _{2j}\partial_{x}^{2j}$.
 \begin{enumerate}
\item If the  boundary conditions 
$B\yx(0,t)=0$ reduce to $\partial^{2j}_{x}y(0,t)=0$ for $2j\leq  M-1$, and if for all $x\in [-1,1]$ and all $(y_0 , ... , y_M)\in (-4,4)^{M+1}$ 
we have 
\be
\label{WWimpair}
f(-x,-y_{0}, \dots,(-1)^{i+1} y_{i},\dots,y_{M-1})= - f(x , y_{0} ,\dots,  y_{M-1})
\ee
then 
\bna
\Comp=\left\{Y_{0}=\yot\in C^{\infty}([0,1])^{N};\quad \partial^{2j}_{x}y_{l}(0)=0\quad \forall j\in \N,\  \forall l=0, .... N-1\right\}
\ena
\item If the  boundary conditions 
$B\yx(0,t)=0$ reduce to $\partial^{2j+1}_{x}y(0,t)=0$ for $2j+1\leq  M-1$, and if for all $x\in [-1,1]$ and all  $(y_0, ... , y_M)\in (-4,4)^{M+1}$ we have 
\be
\label{WWpair}
f(-x,y_{0},\dots,(-1)^{i} y_{i},\dots, - y_{M-1})= f(x,y_{0},\dots,y_{M-1})
\ee
then 
\bna
\Comp=\left\{Y_{0}=\yot\in C^{\infty}([0,1])^{N};\quad \partial^{2j+1}_{x} y_{l}(0)=0,\quad \forall j\in \N, \forall l=0, ... , N-1\right\}.
\ena
\end{enumerate}
\end{proposition}

Note that in the last two cases, the intersection of $\Comp$ with the set of analytic functions is a set of functions that admit odd (respectively even) extensions. Note also that the ``good''  and ``bad'' Boussinesq equations satisfy only \eqref{WWpair}, while the Ginzburg-Landau equation satisfies both \eqref{WWimpair} and \eqref{WWpair}.\\

\begin{remark}
\begin{enumerate}
\item 
The constant $\hat R:= 4M e^{(\lambda e)^{-1}}$ is probably not optimal, but we aimed to provide an explicit (reasonable) constant. For the linear heat equation, it is known that the optimal constant is $\hat R:= 1$ with a diamond-shaped domain of analyticity (see \cite{DE,HKT,HO}).
\item
If $f$ is linear in the variables $\vec{y}=(y_0, \dots, y_{M-1})$, then the PDE \eqref{W1a}  is linear and the smallness assumption on the amplitude of the initial and final data can be removed, as long as  $Y^0,Y^1\in ({\mathcal R}_{R,\hat C})^{N}\cap \Comp$ for some 
$\hat C \in (0,+\infty )$. In particular, for $f(x,\vec{y})=V(x)y_0$, Theorem \ref{thm1} applies for any equation of the form $\partial _t^N y = P\, y + V(x) y$ where $V$ is analytic in a sufficiently large ball. The compatibility set $\Comp$ may depend on $V$. Note however that both conditions \eqref{WWimpair} and \eqref{WWpair} are satisfied if $V$ is even (i.e. $V(x)=V(-x)$ for all $x\in [-1,1]$). Theorem \ref{thm1} applies for instance for the linear heat equation $\partial _t y=\partial _x^2y + V(x)y$ without any smallness assumption about the potential $V(x)$, giving that the reachable space from zero contains functions that are analytic in some sufficiently large ball. See \cite{ELT22} for a more precise result about the reachable space, but under a {\em smallness assumption about the potential}.

Note also that the most relevant term of $P$ is actually the higher order term $\zeta _{M}\partial_{x}^{M}$, since linear lower order terms can be put either in $P$ or in $f(x , y , \partial_{x} y,...,  \partial _x^{M-1} y)$. Yet, we have chosen to keep this form because  $\Comp$ is easily determined for a linear PDE with constant coefficients. 
\item
The definition of $\Comp$ seems to depend on some choice of the functions  $J_{l}^{k}$.  However,  the proof of Lemma \ref{lmdefJlintro} is constructive and therefore it provides an algorithm to define these functions. 
Moreover, it is possible (see Lemma \ref{lmuniqJl} below) to prove that if two functions $J_{l}^{k}$ satisfy the property \eqref{propJlintro} for all solution $y$ of \eqref{W1a}, then they coincide in the product of $[-1,1]$ and some small ball $B(0,\e)$ of $(\R^{N})^{m(l)+1}$ which is the domain where we are going to use it. In any case, the previous property implies that the functions $J_{l}^{k}$ are unique in the class of analytic functions.
\item
The paper has been written for a quite general PDE.  However, it might certainly be possible to consider more general PDEs, containing for instance time derivatives in the lower order terms, 
or in the nonlinearity,  or some time-dependent coefficients. We did not consider these cases because 
it would render the proof more technical and more difficult to read.  
 \end{enumerate}
\end{remark}
The paper is organized as follows. In Section \ref{section1}, we apply our main results to the Korteweg de-Vries equation, the Boussinesq equation, the Ginzbourg-Landau equation, and the Kuramoto-Sivashinsky equation.  Section \ref{section2} is concerned with the existence and uniqueness results for the Cauchy problem in the $x$-variable  (Theorem \ref{thm2}). The relationship between the jet of space derivatives and the jet of time derivatives at some point  (jet analysis) for a solution of \eqref{W1a} is studied in Section \ref{section3}. In particular, 
we show that the nonlinear equation \eqref{W1a} can be (locally) solved forward and backward if the initial data $Y_0$ can be extended as an analytic function in some ball of $\C$ (Proposition  \ref{prop100}). Finally, the proofs of Theorem \ref{thm1} and Proposition \ref{propCompsym} are displayed in Section \ref{section4}.  

\section{Examples}\label{section1}
In this section, we list a few examples of equations coming from physical models for which our general result applies. The list is of course not exhaustive. Also, we limited ourselves to some models that contain a regularizing effect coming from a parabolic behavior or from smoothing boundary conditions. It is not that Theorem \ref{thm1} is limited to this kind of problems, but for conservative equations (like nonlinear Schr\"odinger equations, KdV with some specific boundary conditions as in \cite{rosier}, \cite{ERZ} among other works), it is quite likely (and very often it has already been proved) that the controllability can be obtained in much lower regularity. Notice that even in this context, our result can be interesting if we are looking for a very regular control since the control we build is in some Gevrey class.
\subsection{The Korteweg-de Vries equation}\label{sectKdV}

In this section, we are concerned with the controllability of the  Korteweg-de Vries (KdV) equation: 
\begin{eqnarray}
\partial_{t} y = \partial_{x}^{3}y+\partial_{x}y + y \partial_{x}y, && x \in [0,1], \quad t \in [0,T],\label{KdV1}\\
y(1,t)=h(t),&&  t \in [0,T], \label{KdV2}\\
y(0,t)=0,&&  t \in [0,T], \label{KdV3}\\
\partial_{x}y(0,t)=0,&&  t \in [0,T], \label{KdV4}\\
y(x,0)=y^{0}(x),&& x \in [0,1], \label{KdV5}
\end{eqnarray}
which adapts to our abstract setting \eqref{W1a}-\eqref{W1c}  with $N=1$, $M=3$ (hence $\lambda =3$), $P=\partial_{x}^{3}+\partial_{x}$ and 
$f(x,y,\partial_{x}y, \partial_x^2 y )= y \partial_{x}y$. Thus $Y^t=y$ and $Y^x=(y, \partial _x y, \partial _x^2y)$. 
Note that the change of variables $x\to 1-x$ transforms 
\eqref{KdV1} into the classical form of the KdV equation $\partial _t y + \partial _x^3 y + \partial _x y + y\partial _x y=0$, and \eqref{KdV2}-\eqref{KdV4} into the boundary conditions
$y(0,t)=h(t)$ and $y(1,t)=\partial _x y (1,t)=0$. 
 
It is well-known \cite{GG,rosier} that system \eqref{KdV1}-\eqref{KdV5} is null controllable, and also controllable to the trajectories. Due to the smoothing effect, an exact controllability cannot hold in $L^2(0,1)$. The reachable space for the linearized KdV  equation $\partial _t y=\partial _x^3 y +\partial _x y$ supplemented with the boundary conditions \eqref{KdV2}-\eqref{KdV4} was described in \cite{MRRRkdv}.\\     
By Theorem \ref{thm1}, for any $T>0$ and any $R>\hat R:= 12 e^{(3e)^{-1}}$, there is some number $\hat C>0$ such that for any $ y^0,
\widetilde{y}^0 \in {\mathcal R}_{R, \hat C} \cap \Comp$, 
there exists a solution $y\in G^{1,3}([0,1]\times [0,T])$ of \eqref{KdV1}-\eqref{KdV5} satisfying $y(x,T)=\widetilde{y}^0 (x)$ for all $x\in [0,1]$. Let us now describe more precisely the set $\Comp$
defined in \eqref{comp}. Denote $J_l=(J_{l,1}, J_{l,2}, J_{l,3})$. Recall that $\Comp $ is given by the conditions $B{J_l}(x,y_0,\partial _x y_0,  ...., \partial _x^{m(l)}y_0)\Big|_{x=0}=0$ for all $l\ge 0$, where $B=\left( \begin{array}{ccc} 1&0&0 \\ 0&1&0 \end{array} \right)$. The following Lemma provides a more precise version of Lemma \ref{lmdefJlintro} adapted to KdV.
\begin{lemma}
\label{lem100}
For any $l\in \N $, $m(l)=3l+2$ and there exists a smooth map $H_l:\R ^{3l-1}\to \R$ such that 
\begin{eqnarray}
J_{l,1}&=&y_{3l} + H_l(y_0,y_1,..., y_{3l-2}), \label{LL1} \\
J_{l,2}&=& y_{3l+1} + \sum_{i=0}^{3l-2} \frac{\partial H_l}{\partial y_i} (y_0,y_1, ... , y_{3l-2})y_{i+1},  \label{LL2}\\
J_{l,3}&=& y_{3l+2} + \sum_{i=0}^{3l-2} \frac{\partial H_l}{\partial y_i} (y_0,y_1, ... , y_{3l-2})y_{i+2}, 
+\sum_{i,j=0}^{3l-2}  \frac{\partial ^2 H_l}{\partial y_j \partial y_i} (y_0,y_1, ... , y_{3l-2})y_{j+1} y_{i+1}.\quad \label{LL3}
\end{eqnarray}
\end{lemma}
\begin{proof}
Clearly $J_{0,1}=y_0,\ J_{0,2}=y_1,\ J_{0,3}=y_2$, so that $m(0)=2$ and $H_0=0$. From \eqref{KdV1}, we infer that 
\begin{eqnarray*}
\partial_t\partial_x y         &=& \partial _x^4 y + \partial _x^2 y + y\partial _x^2 y +(\partial _x y)^2,\\
\partial _t \partial _x^2 y  &=& \partial _x^5 y +  \partial _x^3  y + y\partial _x^3 y +3\partial _x y \partial _x ^2 y. 
\end{eqnarray*}
Therefore $m(1)=5$ with 
\[
J_{1,1}=y_3+y_1+y_0y_1,\ \ J_{1,2}=y_4+y_2+y_0 y_2+y_1^2, \ \ J_{1,3}=y_5+y_3+y_0y_3+3y_1y_2. 
\]
Thus $H_1(y_0,y_1)=y_1+y_0y_1$. Assume now that $m(l) = 3l+2$ and that \eqref{LL1}-\eqref{LL3} hold. Then 
\begin{eqnarray*}
\partial _t^{l+1}y&=& \partial _t J_{l,1} (x,y, \partial _x y , ... , \partial _x ^{3l+2} y) \\
&=& \partial _t\big(  \partial _x^{3l} y  + H_l (y, \partial  _xy , ... , \partial _x ^{3l-2} y ) \big)\\
&=& \partial _t \partial _x^{3l} y +\sum_{i=0}^{3l-2} \frac{\partial H_l}{\partial y_i} (y, \partial  _xy , ... , \partial _x ^{3l-2} y )   \partial _t \partial _x^i y. 
\end{eqnarray*}
Since 
\[
\partial _x^i (y\partial _xy) = \partial _x^{i+1} (\frac{y^2}{2}) = \frac{1}{2} \sum_{k=0}^{i+1} \left(\begin{array}{c} i+1\\k\end{array}  \right) \partial _x^k y \partial _x^{i+1-k} y,
\]
we obtain 
\begin{eqnarray}
\partial_t^{l+1} y 
&=& \partial _x^{3l + 3} y + \left( \partial _x ^{3l+1} y + \frac{1}{2} \sum_{k=0}^{3l+1}  \left(\begin{array}{c} 3l+1\\k\end{array}  \right) \partial _x^k y \partial _x^{3l+1-k} y \right. \nonumber\\
&&\left.  + \sum_{i=0}^{3l-2} \frac{\partial H_l}{\partial y_i} (y, \partial _x y, ... , \partial _x ^{3l-2} y)  
\big( \partial _x ^{i+3} y + \partial _x ^{i+1} y + \frac{1}{2} \sum_{k=0}^{i+1} \left(  \begin{array}{c} i+1\\k\end{array}  \right) \partial _x^k y \partial _x^{i+1-k} y\big) 
\right) .    \qquad \label{UU1}
\end{eqnarray}
It follows that $J_{l+1,1}=y_{3l+3} + H_{l+1}(y_0, y_1, ... , y_{3l+1})$ with 
\begin{eqnarray*}
H_{l+1} 
&:=&  y_{3l+1} + \frac{1}{2} \sum_{k=0}^{3l+1}  \left(\begin{array}{c} 3l+1\\k\end{array}  \right) y_k y_{3l+1-k} \\
&&\qquad + \sum_{i=0}^{3l-2} \frac{\partial H_l}{\partial y_i} (y_0, y_1, ..., y_{3l-2})
\big( y_{i+3} + y_{i+1} + \frac{1}{2} \sum_{k=0}^{i+1} \left(  \begin{array}{c} i+1\\k\end{array}  \right) y_k y_{i+1-k} \big) .
\end{eqnarray*}
Thus \eqref{LL1} holds at the step $l+1$. Taking the derivative in $x$ in \eqref{UU1}  gives \eqref{LL2} and \eqref{LL3} at the rank  $l+1$. Finally $m(l+1)=3l+5$. 
\end{proof}
Thus $\Comp$ is the set of the functions $y_0\in C^\infty ([0,1])$ such that $J_{l,1}=J_{l,2}=0$ for all $l\ge 0$, i.e.
\begin{eqnarray*}
&&y(0)=\partial _x y(0)=0, \\
&& \partial _x^{3l} y(0) =- H_l (y,\partial _x y,..., \partial _x ^{3l-2} y)\Big|_{x=0}, \quad \forall l\in \N^*, \\ 
&&\partial_x^ {3l+1} y(0)= - \left(  \sum_{i=0}^{3l-2} \frac{\partial H_l}{\partial y_i} (y,\partial _x y, ... , \partial _x^{3l-2} y)\partial _x ^{i+1} y\right)\Big|_{x=0} \quad \forall l\in \N ^* .
\end{eqnarray*}
Writing $y_0(x)=\sum_{n=0}^\infty \alpha _n\frac{x^n}{n!}$, we obtain the following conditions for the coefficients $\alpha_n$: 
\begin{eqnarray}
&&\alpha_0 = \alpha _1=0,\label{BB1}\\
&&\alpha_{3l} = - H_l (\alpha_0,\alpha_1,..., \alpha_{3l-2}), \quad \forall l\in \N^*, \label{BB2} \\ 
&&\alpha_{3l+1}= -   \sum_{i=0}^{3l-2} \frac{\partial H_l}{\partial y_i} (\alpha_0,\alpha _1, ... , \alpha _{3l-2} )\alpha _{i+1}, \quad \forall l\in \N ^*.\label{BB3}
\end{eqnarray}
We conclude that 
\begin{eqnarray*}
{\mathcal R}_{R, \hat C} \cap \Comp &=& 
\left\{  z :[-1,1]\to \C: \ \exists (\alpha _n)_{n\ge 0} \in {\mathcal N}_{R,\hat C} \textnormal{ such that \eqref{BB1}-\eqref{BB3} hold and }\right. \\
&& \left. z(x) =\sum_{n=0}^\infty \alpha _n \frac{x^n}{n!}, \, \  \forall x\in [-1,1]\right\}. 
\end{eqnarray*}
\begin{remark} 
The condition 
\be
\label{BB10}
|\alpha _n|\le \hat C \frac{n!}{R^n}
\ee
has to be satisfied for all $n\in \N$. It is likely (but still to be proved) that if \eqref{BB10} is satisfied for the subsequence $(\alpha _{3l+2})_{l\ge 0}$, eventually for a small constant $\hat C$, it is also satisfied for the whole sequence 
$(\alpha _n)_{n\ge 0}$ (the two other subsequences $(\alpha _{3l})_{l\ge 0}$ and $(\alpha _{3l+1})_{l\ge 0}$ being defined due to \eqref{BB1}-\eqref{BB3}). 
If it  is indeed the case, then the coefficients $\alpha_{3l+2}$ ($l\in \N$)  can be chosen ``freely'' provided that they satisfy \eqref{BB10}, and hence
the set $ {\mathcal R}_{R, \hat C} \cap \Comp $ looks like a nonlinear submanifold.
  
\end{remark}
\begin{theorem}
\label{thmKdVN}
Let $R>\hat R  := 12 e^{(3e)^{-1}}$ and $T>0$. 
Then there exists some number $\hat C >0$ such that for all functions 
 $y^0$, $ \widetilde{y}^0 \in {\mathcal R}_{R, \hat C} \cap \Comp$, there exist functions 
 $y\in G^{1,3}([0,1]\times [0,T])$ and $h \in G^{3}([0,T])$ satisfying
\eqref{KdV1}-\eqref{KdV5}  together with  $y(x,T)= \widetilde{y}^0 (x)$  for all  $x\in [0,1]$.    
\end{theorem}
\subsection{Boussinesq equation}\label{sectBeq}
We consider the issue of the exact controllability of two systems involving the (good or bad) Boussinesq equation.  
 \subsubsection{Neumann boundary conditions} We first consider the system 
\begin{eqnarray}
\label{BoussiN}
\left\lbrace 
\begin{array}{rll}
\partial_{t}^{2} y &= \pm \partial_{x}^{4}y+\partial_{x}^{2}y -\partial_{x}^2(y^{2}), \quad & x \in [0,1], \quad t \in [0,T],\\
\partial_{x}y(0,t)&=0,   \quad &  t \in [0,T], \\
 \partial_{x}y(1,t)&=v(t), \quad &  t \in [0,T], \\
\partial_{x}^3y(0,t)&=0,  \quad &  t \in [0,T],\\
\quad \partial_{x}^3y(1,t)&=w(t), \quad &  t \in [0,T],\\
y(x,0)&=y^{0}(x), \quad & x \in [0,1], \\
y_{t}(x,0)&=y^{1}(x),  \quad & x \in [0,1]. 
\end{array}
\right.
\end{eqnarray}

If the sign in $\pm$ is $+$, the first equation in \eqref{BoussiN} is called the {\em bad Boussinesq equation} which is known to be severely ill-posed, even for the linear part. It would therefore be difficult to obtain any controllability result with the standard methods. We shall obtain the following exact controllability result. 
\begin{theorem}
\label{thmBoussiN}
Let $R>\hat R  := 16 e^{(2e)^{-1}}$ and $T>0$. 
Then there exists some number $\hat C >0$ such that for all pairs of functions 
 $(y^0,y^1)$, $ (\widetilde{y}^0,\widetilde{y}^1) \in ({\mathcal R}_{R,\hat C})^{2}$ which are even with respect to $0$, there exist functions 
 $y\in G^{1,2}([0,1]\times [0,T])$ and $v, w\in G^{2}([0,T])$ satisfying
\eqref{BoussiN} together with  $y(x,T)=\widetilde{y}^{0}(x)$ and  $y_{t}(x,T)=\widetilde{y}^{1}(x)$ for all  $x\in [0,1]$.      
\end{theorem}
\begin{proof}[Proof of Theorem \ref{thmBoussiN}]
We apply Theorem \ref{thm1} together with Proposition \ref{propCompsym}  with $\lambda=4/2=2$. Note that the control inputs  $v, w$ are just taken as traces of the constructed solution $y\in G^{1,2}([0,1]\times [0,T])$. The regularity of $v, w$ then follows from \eqref{defGev2D}.  

We need to check that the non-linearity satisfies the right assumption. Since 
$\partial_{x}^2(y^{2})=2\left(y\partial_{x}^2y+(\partial_x y)^2\right) $, the non-linearity reads $f(x,y_0,y_1,y_2,y_3)=-2(y_0 y_2+y_1^2)$.
As 
$$f(-x,y_{0}, -y_{1},y_{2},-y_3)=-2(y_0 y_2+(-y_1)^2)= f(x , y_{0} , y_{1},y_2,y_3),$$ 
we see that condition \eqref{WWpair} in Proposition \ref{propCompsym} is fulfilled. 
Finally, we notice that for any function $h\in {\mathcal R}_{R,\hat C}$, $h$ is even if and only if $\partial^{2j+1}_{x}h(0)=0$ for any  $j\in \N .$
\end{proof}

\subsubsection{Dirichlet boundary conditions} If we keep the non-linearity $f(x,y,\partial _x y, \partial _x^2 y, \partial _x ^3 y)=-\partial_{x}^2(y^2)$, then
 $$f(-x,-y_{0}, y_{1},-y_{2},y_3)=-2(y_0 y_2+y_1^2)= f(x , y_{0} , y_{1},y_2,y_3),$$ 
 so that condition \eqref{WWimpair} in Proposition \ref{propCompsym} is not fulfilled. Theorem \ref{thm1} may be applied, but the determination of the compatibility set $\Comp$ is not obvious.

We consider instead a different non-linearity, namely  $f(x,y,\partial _x y, \partial _x^2y, \partial _x ^3 y)=-\partial_x(y^2)$. More precisely, we consider
the system
\begin{eqnarray}
\label{BoussiD}
\left\lbrace 
\begin{array}{rll}
\partial_{t}^{2} y &= \pm \partial_{x}^{4}y + \partial_{x}^{2}y -\partial_{x}(y^{2}), \quad & x \in [0,1], \quad t \in [0,T],\\
y(0,t)&=0,   \quad &  t \in [0,T], \\
y(1,t)&=v(t), \quad &  t \in [0,T], \\
\partial_{x}^2y(0,t)&=0,  \quad &  t \in [0,T],\\
\quad \partial_{x}^2y(1,t)&=w(t), \quad &  t \in [0,T],\\
y(x,0)&=y^{0}(x), \quad & x \in [0,1], \\
y_{t}(x,0)&=y^{1}(x),  \quad & x \in [0,1]. 
\end{array}
\right.
\end{eqnarray}

\begin{theorem}
\label{thmBoussiDir}
Let $R>\hat R  := 16 e^{(2e)^{-1}}$ and $T>0$. 
Then there exists some number $\hat C >0$ such that for all pairs of functions 
 $(y^0,y^1)$, $(\widetilde{y}^0,\widetilde{y}^1)\in ({\mathcal R}_{R,\hat C})^{2}$ which are odd with respect to $0$, there exist functions $y\in G^{1,2}([0,1]\times [0,T])$ and $v, w\in G^{2}([0,T])$
 satisfying \eqref{BoussiD} together with  $y(x,T)=\widetilde{y}^{0}(x)$ and $y_{t}(x,T)=\widetilde{y}^{1}(x)$ for all  $x\in [0,1]$.    
\end{theorem}
\begin{proof}[Proof of Theorem \ref{thmBoussiDir}]
The proof is the same as for Theorem \ref{thmBoussiN}. Since $\partial_{x}(y^{2})=2 y\partial_{x}y$, the non-linearity reads $f(x,y_0,y_1,y_2,y_3)=-2 y_0 y_1$. 
From  $$f(-x,-y_{0}, y_{1},-y_{2},y_3)=2y_0 y_1= - f(x , y_{0} , y_{1},y_2,y_3),$$
we infer that condition \eqref{WWimpair} in Proposition \ref{propCompsym} is fulfilled. As 
a function $h\in {\mathcal R}_{R,\hat C}$ is odd if and only if $\partial^{2j}_{x}h(0)=0$ for any  $j\in \N$, the conclusion follows at once. 
\end{proof}

\subsection{The complex Ginzburg-Landau equation}
We are concerned with the controllability of the complex Ginzburg-Landau equation with parameters $\theta, \varphi\in \R$. 
We begin with the control problem with Dirichlet boundary conditions:
\begin{eqnarray}
\label{GLDir}
\left\lbrace 
\begin{array}{rll}
\partial _t y&=e^{i\theta}\partial _x^2 \, y +e^{i\varphi} |y|^2y, \quad & x \in [0,1], \quad t \in [0,T],\\
y(0,t)&=0, \quad &  t \in [0,T], \\
y(1,t)&=v(t), \quad &  t \in [0,T], \\
y(x,0)&=y^{0}(x)  \quad & x \in [0,1]. 
\end{array}
\right.
\end{eqnarray}
\begin{theorem}
\label{thmGLDir}
Let $R>\hat R := 8 e^{(2e)^{-1}}$ and $T>0$. 
Then there exists some number $\hat C >0$ such that for all functions 
 $y^0$, $\widetilde{y}^0 \in {\mathcal R}_{R,\hat C}$ which are odd with respect to $0$, there exist  $y\in G^{1,2}([0,1]\times [0,T])$ 
 and  $v\in G^{2}([0,T])$ satisfying  
\eqref{GLDir} together with $y(x,T)=\widetilde{y}^{0}(x)$ for all  $x\in [0,1]$.      
\end{theorem}
The control problem with Neumann boundary conditions reads
\begin{eqnarray}
\label{GLN}
\left\lbrace 
\begin{array}{rll}
\partial _t y&=e^{i\theta}\partial _x^2 \, y +e^{i\varphi} |y|^2y, \quad & x \in [0,1], \quad t \in [0,T],\\
\partial_x y(0,t)&=0, \quad &  t \in [0,T], \\
\partial_x y(1,t)&=v(t), \quad &  t \in [0,T], \\
y(x,0)&=y^{0}(x)  \quad & x \in [0,1]. 
\end{array}
\right.
\end{eqnarray}
\begin{theorem}
\label{thmGLN}
Let $R>\hat R := 8 e^{(2e)^{-1}}$ and $T>0$. 
Then there exists some number $\hat C >0$ such that for all functions
 $y^0$, $\widetilde{y}^0\in {\mathcal R}_{R,\hat C}$ which are even with respect to $0$, there exist $y\in G^{1,2}([0,1]\times [0,T])$ and $v\in G^{2}([0,T])$ satisfying
\eqref{GLN} together with $y(x,T)=\widetilde{y}^{0}(x)$ for all  $x\in [0,1]$.      
\end{theorem}
The proof follows the previous cases closely, except that they are complex-valued functions and the nonlinearity $|y|^2y=y^2\overline{y}$ cannot be written as a sum (finite or infinite) of powers of the variable $y$. We describe in Section \ref{s:GLproof} the modifications that must be performed to get the expected result.
\begin{remark}It might seem problematic to use the nonlinearity $f(z)=|z|^2z$ which is not holomorphic. The solution we construct satisfies $y(\cdot,t)\in {\mathcal R}_{R,\hat C}$, which means that it is real analytic on $[-1,1]$ for any $t\in [0,T]$, in the sense that it agrees with its Taylor expansion at $0$, which is enough for the proof we are doing. Indeed, as noticed in Lemma \ref{lemmasets}, it implies that it has a holomorphic extension $z\mapsto y(z,t)$ for $z\in B_{\C}(0,\widetilde{R})$ for some $\widetilde{R}>0$. The application $x\in [-1,1]\mapsto |y(x,t)|^2y(x,t)$ is also real analytic and also has a holomorphic extension. Yet, this extension does not coincide with $|y(z,t)|^2y(z,t)$. In particular, the solution can be extended to $B_{\C}(0,\widetilde{R})\times [0,T]$ but it is not clear what equation it satisfies on this set. We only know that the solution satisfies the Ginzburg-Landau equation on $[-1,1]\times [0,T]$.\end{remark}
\subsection{The Kuramoto-Sivashinsky equation}
We investigate the controllability of the Kuramoto-Sivashinsky (KS) equation with boundary conditions of Dirichlet type:
\begin{eqnarray}
\label{KSDir}
\left\lbrace 
\begin{array}{rll}
\partial _t y &= -\partial _x ^4 y - \partial _x ^2 y  -y\partial_xy, \quad & x \in [0,1], \quad t \in [0,T],\\
y(0,t)&=0, \quad &  t \in [0,T], \\
y(1,t)&=v(t), \quad &  t \in [0,T], \\
\partial_{x}^2y(0,t)&=0, \quad &  t \in [0,T],\\
\partial_{x}^2y(1,t)&=w(t), \quad &  t \in [0,T],\\
y(x,0)&=y^{0}(x),   \quad & x \in [0,1]. 
\end{array}
\right.
\end{eqnarray}

\begin{theorem}
\label{thmKSDir}
Let $R>\hat R=16 e^{(4e)^{-1}}$ and $T>0$. 
Then there exists some number $\hat C >0$ such that for all functions 
 $y^0$, $\widetilde{y}^0\in {\mathcal R}_{R,\hat C}$ which are odd with respect to $0$, there exist functions $y\in G^{1,4}([0,1]\times [0,T])$ and $v, w\in G^{4}([0,T])$ satisfying 
\eqref{KSDir} together with  $y(x,T)=\widetilde{y}^{0}(x)$ for all  $x\in [0,1]$.      
\end{theorem}
\begin{proof}[Proof of Theorem \ref{thmKSDir}] 
For $\lambda=4/1=4$, and the non-linearity reads as $f(x,y_0,y_1, y_2,y_3)= - y_0 y_1$. It satisfies 
$$f(-x,-y_{0}, y_{1} ,-y_2,y_3)=y_0 y_1= -f(x,y_0,y_1, y_2,y_3),$$
 which is  condition \eqref{WWpair} in  Proposition \ref{propCompsym}.
\end{proof}


The null controllability for the Kuramoto-Sivashinsky equation has been already studied in \cite{C,CMP,GMC,KM}, for different combinations of boundary data, and in the cases where boundary setting agrees with the setting of \eqref{W1b}, our results are consistent with the known results.  However, the critical set of parameters of diffusion appears only in cases when only one control is considered, which is not the case in this paper.

\subsection{The case of a linear PDE with constant coefficients}
\begin{proposition} \label{propcompatlin}
Assume $f=0$ (linear PDE with constant coefficients). Then 
\bna
\Comp=\left\{Y_{0}=\yot\in C^{\infty}([0,1])^{N};\quad (BP^kY^{x,l})(0)=0,\  \forall k\in \N, \ \forall l=0,\dots, N-1\right\}
\ena 
where we have denoted $Y^{x,l}:=(y_l, \dots, \partial_x^{M-1} y_l)$ as in \eqref{Yx} .
\end{proposition}
\begin{proof}
Using Euclidian division, we are led to compute the application $J_{Nk+l}$ defined in Lemma \eqref{lmdefJlintro} for any $k\in \N$ and $l=0,\dots, N-1$.

We infer from \eqref{W1a} that  $\partial_t^{Nk+l}\partial_x^i y=P^k \partial_{t}^l \partial_x^i y$ for any $k\in \N$, $l=0,...,N-1$ and $i\in \N$. In particular, $\partial_t^{Nk+l}Y^x =P^k \partial_{t}^l Y^x$ for any $k\in \N$, $l=0,...,N-1$,  
we can define a linear map $J_{Nk+l}: (\R^{N})^{(k+1)M}\to \R^{M}$ such that 
$$J_{Nk+l}(Y_0(0),\cdots, \partial_x^{(k+1)M-1} Y_0(0))=(P^k y_l, P^k\partial_x y_l, \dots,P^k\partial_x^{M-1} y_l  )(0),$$ 
for any $Y_{0}=(y_0,y_1,...,y_{N-1})\in C^{\infty}([0,1])^{N}$ (denoting 
$y_0=y$, $y_l =\partial _t ^l y_0$ for $1\le l\le N-1$). Moreover, for a solution of the equation with initial datum $Y_0$, we have 
\begin{equation}\label{property}
(P^k y_l, P^k\partial_x y_l, \dots,P^k\partial_x^{M-1} y_l  )(0)=P^k Y^{x,l}(0).\end{equation}
The previous computation gives  $\partial_t^{Nk+l}Y^x(0)=J_{Nk+l}(Y_0(0),\cdots, \partial_x^{(k+1)M-1} Y_0(0))$. Therefore, the application $J_{Nk+l}$   by \eqref{property} 
satisfies the property  \eqref{lmdefJlintro} for all solutions. Then, using the uniqueness of the operators $J_l$ (up to adding unnecessary variables) proved in Lemma \ref{lmuniqJl}, we conclude that it is the expected application. \\

In particular, $BJ_{Nk+l}(Y_0(0),\cdots, \partial_x^{(k+1)M-1} Y_0(0))=0$ is equivalent to $BP^k Y^{x,l}(0)=0$
 for any $k\in \N$ and $l=0,\dots, N-1$.
\end{proof}

\section{Cauchy problem in the space variable}
\label{section2}
\subsection{Statement of the global wellposedness result} 

Let $f=f(x,y_0,y_1,\cdots,y_{M-1})$ be as in \eqref{AB1}-\eqref{AB4}. 
We are concerned with the wellposedness of the Cauchy problem: 
\bneqn
&&\partial _t^N y = P  y + f(x , y ,...,  \partial _x^{M-1} y),\quad  x\in [-1,1],\,  t\in [t_1,t_2], \label{C1} \\
&&\partial_x^{i}y(0,t)= k_{i}(t), \hspace{2,9cm} 0\le i\leq  M-1, \quad t\in [t_1,t_2] 
\eneqn 
for some given functions $k_0,\dots, k_{M-1}\in G^{\lambda}([t_1,t_2])$. We denote $K_0=(k_0 , ... , k_{M-1})$.
 Note that the initial conditions of \eqref{C1} can be written as $\yx(0,t)=K_0(t)$.\\

The goal of this section is to prove the following result.
\begin{theorem}
\label{thm2}
Let $P$ be as in \eqref{formP} and $f=f(x,\vec{y})$ be as in \eqref{AB1}-\eqref{AB4}. 
Let $-\infty <t_1<t_2<+\infty$ and $R>(4N)^{\lambda}$, where $\lambda =M/N$. Then there exists some numbers $C>0$, $Q>0$, $R_1,R_2$ with $4Me^{-1/M}<R_1<R_2$ satisfying that 
for all 
$K_0=(k_0,\dots, k_{M-1})\in G^{\lambda}([t_1,t_2])^{M}$ with 
 \be
 \label{AA}
 \vert k^{ (n) } _i(t)\vert \le C  \left(\frac{|\zeta_M|^{1/N}}{R}\right)^{n}(n!)^{\lambda}, \quad i=0,1,\dots, M-1,\ n\ge 0, \ t\in[t_1,t_2],  
 \ee
there exists a solution $y\in G^{1,\lambda} ([-1,1]\times [t_1,t_2])$ of \eqref{C1} satisfying,
\be
\label{AAAAA}
\vert \partial _x ^{p_1}\partial _t ^{p_2} y(x,t) \vert  \le  Q \frac{ ( p_1 + \lambda p_2) !  }{R_1^{p_1} R_2^{\lambda p_2}}|\zeta_M|^{p_2/N},\quad 
 (x,t)\in [-1,1]\times [t_1,t_2], \  (p_1,p_2)\in \N ^2.
\ee
\end{theorem} 
The proof of Theorem \ref{thm2} will be given after that some preliminary results are established.
We use the notation $x!=\Gamma(x+1)$ even if $x$ is not an integer. 
\begin{remark}
\label{rkscaling}

 It is sufficient to prove Theorem \ref{thm2} for the unidimensional system \eqref{C1}, i.e  Considering  $ |\zeta _M|=1$ and $[t_1,t_2]=[0,t_2]$. Indeed, the equation 
$\partial _t^N y=Py+f(x,y, \ldots , \partial _x ^{M-1}y)$ is invariant by translation in time, so that we can assume that 
 $[t_1,t_2]=[0,t_2]$.  On the other hand, if  $| \zeta_ M |\in (0,+\infty )\setminus \{ 1\}$, we can use the following scaling argument. 
Set  $\widetilde \zeta_M :=\zeta _M / | \zeta _M|$, 
$\widetilde{P} := |\zeta_{M}|^{-1} P$ and $\widetilde{f} :=|\zeta_{M}|^{-1}f$. Note that $\widetilde{P}$ and $\widetilde{f}$ satisfy the expected assumptions with $|\widetilde{\zeta}_M|=1$. 
For $K_0$ satisfying \eqref{AA} on $[0,t_2]$, define $\widetilde{K}_0(t) :=K_0(|\zeta_M|^{-1/N}t)$. Then  $\widetilde{K}_0$ 
satisfies \eqref{AA} with $|\widetilde \zeta_M|=1$, that is $\vert \widetilde{k}^{ (n) } _i(t)\vert \le C\frac{ (n!)^{\lambda} }{R^{n}}$ on $ [0,|\zeta_M|^{1/N}t_2] $. This allows to define a solution $\widetilde{y}(x,t)$ of \eqref{C1} for $x\in [-1,1]$ and $t\in [0,|\zeta_M|^{-1/N}t_2] $   associated with $\widetilde P$, $\widetilde f$ and $\widetilde K_0$. Then the function
\[ y(x,t):=\widetilde y ( x,  |\zeta _M |^{\frac{1}{N}} t )  , \qquad x\in [-1,1], \ t\in [0,t_2] \]
is a solution of \eqref{C1} associated with $P$, $f$ and $K_0$. 
\end{remark}
\subsection{Abstract existence theorem}
We consider a family of Banach spaces $(X_s)_{s\in [0,1]}$ satisfying for $0\leq s'\leq s\leq 1,$
\bnan
\label{embed}
X_s\subset X_{s'},\\
\label{CC1}
\nor{f}{X_{s'}}\leq \nor{f}{X_s};
\enan
that is, the embedding $ X_s\subset X_{s'}$ for $s'\le s.$


We are concerned with an abstract Cauchy problem:  
\begin{gather*}
\begin{cases}
\partial_x U(x)=\mathrm{T} (x)U(x), &  -1\le x\le 1, \\
U(0)=U^0
\end{cases}
\end{gather*}
where $U^0\in X_1$ and $\big( \mathrm{T}(x)\big)_{x\in [-1,1]}$ is a family  of nonlinear operators
with possible {\em loss of derivatives}. 

The following result, taken from \cite[Theorem 2.2]{LR},  is a  
global wellposedness result. It extends the abstract result  in \cite{nirenberg, nishida} which gives only {\em local } solutions.
\begin{theorem}
\label{thmexistenceloc}
Let $\e \in (0,1/4)$, $D>0$ and a family $(\mathrm{T}(x))_{x\in [-1,1]} $ of nonlinear maps from $X_s$ to $X_{s'}$ for $0\le s'<s\le 1$ satisfying 
\bnan
\nor{\mathrm{T}(x)U}{X_{s'}}& \le &\frac{\e }{s-s'}\nor{U}{X_{s}},
\label{hypoGloc1}
\\
\label{hypoGloc2}
\nor{\mathrm{T}(x)U- \mathrm{T}(x)V}{X_{s'}}& \le &\frac{\e }{s-s'}\nor{U-V}{X_{s}}
\enan
for $0\le s'<s\le 1$, $x\in [-1,1]$ and $U$,$V \in X_s$ 
with $\nor{U}{X_{s}}\le D$, $\nor{V}{X_{s}}\leq D$. Then there exists a number $0<\eta\leq D$ such that for any $U^0\in X_1$ 
with $\nor{U^0}{X_1}\leq \eta$, there exists a 
solution $U\in C([-1,1],X_{s_0})$ for some $s_0\in (0,1)$ of the integral equation
\begin{equation}
\label{equationintegrale}
U(x)=U^0+\int_0^x \mathrm{T}(\tau)U(\tau)d\tau .
\end{equation}
Moreover, we have the estimate 
\[
\nor{U( x )}{X_s}\leq C_1 \left(1 -\frac{ \alpha |x|}{a_\infty (1-s)} \right) ^{-1} \nor{U^0}{X_1}, \quad 
\textrm{\rm for }   \ 0\le s<1, \ |x|<\frac{a_\infty}{\alpha} (1-s),
\]
where $\alpha \in (0,1)$, $a_\infty \in (\alpha ,1)$ and $C_1>0$ are some constants. In particular, we have 
\[
\nor{U( x )}{X_s}\leq C_1 \left(1 -\frac{ 2}{\frac{a_\infty}{\alpha} +1} \right) ^{-1} \nor{U^0}{X_1}, \quad 
\textrm{\rm for }   \ 0\le s\le s_0=\frac{1}{2}(1-\frac{\alpha}{a_\infty}), \ |x| \le 1.
\]
If, in addition,  we assume that 
\be
\label{hypoGloc0}
\textrm{ for all }  U_0\in X_s \textrm{ with } \Vert U_0\Vert _{X_s} \le D,   \textrm{ the map } \tau \in [-1,1]\to \mathrm{T}(\tau ) U_0\in X_{s'} \  \textrm{ is continuous}, 
\ee
then $U$ is the  classical solution  of  
\bneqn
\label{eqnabstraite}
\partial_x U(x)&=&\mathrm{T}(x)U(x), \quad -1\le x\le 1, \\
U(0)&=&U^0.
\eneqn
\end{theorem}
Note that we slightly changed the order of the quantifiers for $D$ to the original statement in \cite[Theorem 2.2]{LR}. The result is a direct consequence of \cite[Proposition 2.3.]{LR} where the quantifiers are written this way.

\subsection{Gevrey type functional spaces}

We define several  $\lambda$  Gevrey spaces for $\lambda>1$ (see \cite{KY,yamanaka}) and we follow closely the ideas developed in \cite{LR} for the heat equation. We shall take $\lambda=M/N$, but for the moment we stay in the generality.\\ 

We  introduce a variant of the Gamma function of Euler  with a parameter $a\in\R$ given by 
\begin{gather}\label{defgammaa}
\Gamma_{\lambda,a}(k)=\begin{cases}
2^{-5}(\Gamma(k+1 -a))^{\lambda}(1+k)^{-2}, &   k\in \N, \  k> |a|+1,  \\
\Gamma_{\lambda}(k),&  k\in \N, \  0\leq k\leq  |a|+1
\end{cases}
\end{gather}
with 
\begin{gather}
 \Gamma_{\lambda}(k)=2^{-5}(k!)^{\lambda}(1+k)^{-2},
 \end{gather}
and $\Gamma$ being the usual Gamma function of Euler which is increasing on $[2,+\infty )$. \\

Clearly, $\Gamma _{\lambda ,0}=\Gamma _\lambda$. 
Note that for $k> |a|+1 $, we have $k +1 -a\geq 2$ and $k+1\geq 2$, so we are in an interval where $\Gamma$ is increasing. Thus  we have for all $k\in \N$
\bnan
\label{Gammacroissance}
&&\Gamma_{\lambda,a}(k)\leq \Gamma_{\lambda}(k),  \quad\textnormal{ if } a\geq 0,\\
\label{Gammacroissance2}
&&\Gamma_{\lambda}(k)\leq \Gamma_{\lambda,a}(k),  \quad  \textnormal{ if } a\leq 0.
\enan
For any $L>0$, we consider the intermediate space of functions in $C^{\infty}(K)$ (where $K=[t_1,t_2]$ with $-\infty <t_1<t_2<\infty$) such that
\bna
\left|u\right|_{L,a}:=\sup_{t\in K, k\in\N} \frac{\left|u^{(k)}(t)\right|}{L^{|k-a|}\Gamma_{\lambda,a}(k)} <\infty .
\ena
Note that for $a=0$, we recover the spaces defined earlier in \cite{yamanaka}, and $\left|u\right|_{L,0}=\left|u\right|_{L}$.

\begin{definition}\label{defGevrey}
We consider the norm defined in  \cite{yamanaka} by Yamanaka 
\bna 
\nor{u}{L}:=\max \left\{2^6\nor{u}{L^{\infty}(K)},2^3L^{-1}\left|u'\right|_L\right\} ,
\ena
and similarly, we define for $a\in \R$
\bna
\nor{u}{L,a}:=\max \left\{2^6\nor{u}{L^{\infty}(K)},2^3L^{-1}\left|u'\right|_{L,a}\right\} .
\ena

For $L>0$ and $0\le a_1<a_2$, we have 
\begin{equation}
\Vert u\Vert _{L,a_1} \le C(L,a_1,a_2, \lambda ) \Vert u\Vert _{L,a_2} \quad \forall u\in G_{L,a_2}^\lambda . 
\label{QQ1}
\end{equation}
Indeed, for $k> a_2+1$, we have $k+1-a_1\geq k+1-a_2\geq 2$ where $\Gamma$ is increasing so that $\Gamma(k+1-a_2)\leq \Gamma(k+1-a_1)$, and therefore $L^{|k-a_2|}\Gamma_{\lambda,a_2}(k)\leq L^{a_1-a_2}L^{|k-a_1|}\Gamma_{\lambda,a_1}(k)$. We can obtain a similar inequality for $k\leq a_2+1$ with different constant which gives then \eqref{QQ1}.\\

We define the Banach spaces  $G_{L,a}^{\lambda}$ and $G_L^\lambda$ as
\begin{gather}
G_{L,a}^{\lambda} :=\{u\in C^{\infty}(K) \quad \text{such 
that} \quad \nor{u}{L,a} <\infty \}
\end{gather} 

and 
\begin{gather}
G_{L}^{\lambda} :=\{u\in C^{\infty}(K) \quad \text{such 
that} \quad \nor{u}{L} <\infty \}.
\end{gather} 

\end{definition}
The space  $G_{L,a}^{\lambda}$ can be seen as 
the space of functions Gevrey $\lambda$ with radius $L^{-1}$ with $a$ derivatives. Roughly, we could think that $u\in G_{L,a}^{\lambda}$ if $D^a u\in G^{\lambda}_L$, even if it is not completely true if $a\notin \N$.

Note that, as a direct consequence of 
\eqref{Gammacroissance}-\eqref{Gammacroissance2}, we have the embeddings $G_{L,a}^{\lambda}\subset G_L^{\lambda}$ if $a\geq 0$ and $G_L^{\lambda} \subset G_{L,a}^{\lambda}$ if $a\leq 0$, together with the inequalities
\bnan
\label{Gacroissance+}\nor{u}{L}\leq \max(L^{a},L^{-a})\nor{u}{L,a}, \textnormal{ if } a\geq 0,\\
\label{Gacroissance-}\nor{u}{L,a}\leq \max(L^{a},L^{-a}) \nor{u}{L}, \textnormal{ if } a\leq 0.
\enan

Furthermore, for any $a\in \R$ and $0<L<L'$, we have the embedding $G_{L,a}^{\lambda}\subset G_{L',a}^{\lambda}$ with
\bnan
\label{GacroissanceL}\nor{u}{L',a}\leq \nor{u}{L,a}.
\enan


The following result  \cite[Theorem 5.4]{yamanaka:CPDE} will be used several times in the sequel.
\begin{lemma} \label{algebre}
\textrm{(Algebra property)} For $L>0$
\be
\nor{uv}{L}\leq \nor{u}{L}\nor{v}{L}\quad \forall u,v\in G_L^\lambda. 
\ee
\end{lemma}
The following result \cite[Lemma 2.6]{LR}  is a variant of \cite[Proposition 2.3]{KY} with spaces containing non-integer ``derivatives". 
\begin{lemma}[Cost of derivatives for Gevrey spaces containing derivatives]\label{coutderivee}

Let $\lambda>0$ and $\delta >0$. Let $q\in \N$ and $a,b\in \R$ with $d=q-a+b>0$. Then there exists some number $C=C(\lambda,\delta,a,b,q)>0$  such that for all $L>0$,  $\alpha>1$ and  $u\in G_{L,a} ^\lambda$, we have 
\begin{equation}
\label{UUU1}
\left|u^{(q)}\right|_{\alpha L,b}\leq \left(C (L^{-d} + L^d) + (1+\delta)\alpha^b L^d \left(\frac{\lambda d}{e\ln \alpha}\right)^{\lambda d}\right) \left|u\right|_{L,a}
\end{equation}
and hence
\begin{equation}
\label{UUU2}
\nor{u^{(q)}}{\alpha L,b}\leq \left(C(L^{-d} +\left\langle L\right\rangle^C)+ (1+\delta)\alpha^b L^d \left(\frac{\lambda d}{e\ln \alpha}\right)^{\lambda d}\right)\nor{u}{L,a}.
\end{equation}
  where we denote $\langle x\rangle := \sqrt{1+x^2}$ for $x\in \R$.
\end{lemma}
\subsection{Application to the semi-linear PDE}
We write our system in the equivalent form
\begin{gather}
\begin{cases}
\partial _x ^Mu =\frac{1}{\zeta _M}\left( \partial _t^N u -\sum_{j=0}^{M-1}\zeta _j\partial_{x}^j u - f(x,u, \partial _x u,...,\partial_x^{M-1}u)\right) ,&  x\in [-1,1],\  t\in [t_1,t_2], \label{DDD10} \\
U^x(0,t)=K_0(t),& t\in [t_1,t_2], 
\end{cases}
\end{gather}

recalling $U^x(x,t)=\left(u(x,t),\partial_x  u(x,t),...,\partial_x^{M-1}u(x,t)\right)$ and $K_0:=\left(k_0(t),\dots, k_{M-1}(t)\right)$. $|\zeta _M|=1$ will be considered in this section, 
 for more detailed  see  Remark \ref{rkscaling}.\\        

We write \eqref{DDD10} as a first-order system
\bnan\label{eqnU}
\partial_x U &=& A U +F(x,U),\\
\label{cdini}U(0)&=&K_0
\enan
with $U=U^x=(u,\partial_x u,\cdots,\partial_x^{M-1}u)$,
\ba
A  =\left(\begin{array}{ccccc}0 &1&0&\dots&0\\ 
0 &0&1&\dots&0\\ 
\vdots &\vdots&\vdots&\dots&\vdots\\ 
0 &0&0&\dots&1\\ 
\zeta_M^{-1}(\partial_t^N-\zeta_{0} )&-\zeta_M^{-1}\zeta_{1}&\cdots&-\zeta_M^{-1}\zeta_{M-2}&-\zeta_M^{-1}\zeta_{M-1}\end{array}\right), 
\ea
and 
\[
F(x,\vec{u})=\left(\begin{array}{c}0\\ \vdots\\ -\zeta_M^{-1}f(x,\vec{u})\end{array}\right),
\]
where the current vector $\vec{u}:=(u_0,u_1,\ldots,u_{M-1})$ will contain the derivatives.
We decompose $A$ as
\begin{eqnarray*}
A  &=&A_0+A_R\\
&=&
\left(\begin{array}{ccccc}0 &1&0&\dots&0\\ 
0 &0&1&\dots&0\\ 
\vdots &\vdots&\vdots&\dots&\vdots\\ 
0 &0&0&\dots&1\\ 
\zeta_M^{-1}\partial_t^N&0&\cdots&0&0\end{array}\right)+\left(\begin{array}{ccccc}0 &0&0&\dots&0\\ 
0 &0&0&\dots&0\\ 
\vdots &\vdots&\vdots&\dots&\vdots\\ 
0 &0&0&\dots&0\\ 
-\zeta_M^{-1}\zeta_{0} &-\zeta_M^{-1}\zeta_{1}&\cdots&-\zeta_M^{-1}\zeta_{M-2}&-\zeta_M^{-1}\zeta_{M-1}\end{array}\right) .
\end{eqnarray*}

Let $L>0$, we define the space 
\begin{gather}\label{Xl}
{\mathcal X}_L:=\{ U=(u_0,u_1,\ldots,u_{M-1})
\in G_{L, \frac{M-1}{\lambda}}^\lambda\times \ldots \times G_{L, \frac{1}{\lambda}}^\lambda \times G_L^\lambda \}
\end{gather}
with the norm 
\bna
\nor{U}{{\mathcal X}_L}= \nor{ (u_0,u_1,\ldots,u_{M-1}) }{{\mathcal X}_L}=
 \nor{u_0}{L, \frac{M-1}{\lambda}}+ \ldots+\nor{u_{M-1}}{L}=\sum_{j=0}^{M-1}\nor{u_j}{L,\frac{M-j-1}{\lambda}},
\ena
where the norms are those defined in Definition {\ref{defGevrey} with $\lambda=M/N$. 
Note that $u_0$ is more regular than $u_1$ of "$1/\lambda $ derivative". In particular, using that $|\zeta _M|=1$, we have that
\bna
\nor{A_0U}{{\mathcal X}_{L}}=\sum_{j=1}^{M-1}\nor{u_j}{L,\frac{M-j}{\lambda}}+\nor{ \partial_t^N u_0}{L} .
\ena
In the following result,  $L_1$ stands for the inverse of the radius $R$ of the initial datum.
\begin{theorem}
\label{thm4}
Pick any $L_1$ with $0<L_1<\frac{1}{(4N)^\lambda}$. Then there exists a number $\eta>0$ such that for any $K_0\in {\mathcal X}_{L_1}$ 
with $\nor{K_0}{{\mathcal X}_{L_1}}\leq \eta$, there exists a 
solution to \eqref{DDD10}
 in 
$C([-1,1], {\mathcal X}_{L_0})$ for some $L_0>0$. 
\end{theorem}
\bnp
In order to apply Theorem \ref{thmexistenceloc}, we introduce a scale of Banach spaces $(X_s)_{s\in [0,1]}$ as follows, 
for $s\in [0,1]$, we set
\ba
\nor{U}{X_s}&=&e^{-\tau (1-s)}\nor{U}{{\mathcal X}_{L(s)}} \quad \textrm{ for } U\in X_s := {\mathcal X}_{L(s)}  \label{PP0}\\
L(s)&=&e^{r(1-s)}L_1,  \label{PP1}
\ea
where $$r=1/N$$ and $\tau>0$ will be chosen thereafter. Note that \eqref{CC1} is satisfied from \eqref{GacroissanceL} 
and the fact that $L(s')>L(s)$ for $s'<s$. Additionally, we have that
\bnan
\label{gainXs}
\nor{U}{X_{s'}}\leq e^{-\tau (s-s')}\nor{U}{X_{s}}.
\enan
The use of Lemmas \ref{lmAder}, \ref{lmArest} and  \ref{lemmeF}  will allow us to select the parameters such that  $\mathrm{T}=A+F$ satisfies the assumptions of Theorem \ref{thmexistenceloc}. Then, we only need to notice that $\nor{K_0}{X_1}= \nor{K_0}{{\mathcal X}_{L_1}}\leq D$ for $\eta=D$ small. 
\enp
\begin{remark}
It is interesting to notice that for Theorem \ref{thm4}, we use the analytic regularity of $f$ in the variables $u_0, \ldots, u_{M-1}$, but only the continuity of $f$ in $x$. The analyticity of $f$ in $x,u_0, \ldots, u_{M-1}$ will be used to prove the additional regularity of the solution in the variable $x$. Also, as noticed in Section \ref{s:GLproof} concerning the Ginzburg-Landau equation, the same result holds for a polynomial function of $u_0,\overline{u_0} \ldots, u_{M-1},\overline{u_{M-1}}$. The crucial part for the existence is the composition of Gevrey functions. 
\end{remark}
\begin{lemma}
\label{lmAder}
Let $L_1<\frac{1}{(4N)^\lambda}$. There exist $\tau_0>0$ (large enough) and $\e_0<1/4$ such that we have the estimates
 \bna
\nor{A_0U}{X_{s'}}&\leq &\frac{\e_0}{s-s'}\nor{U}{X_{s}}, \quad \forall U\in X_s,
\ena
for all $\tau \ge \tau _0$  (as in \eqref{PP0}) and all $s,s'$ with  $0\leq s'<s\leq 1$.
\end{lemma}
\bnp
By assumption,  $N L_1^{1/\lambda}<1/4$. Pick $\delta>0$ small enough such that 
\bnan
\label{choix14}
(1+\delta)N L_1^{1/\lambda}<1/4,
\enan

applying Lemma \ref{coutderivee} to the $M-1$-first terms of $A_0U$ (namely $u_1,...,u_{M-1}$)   for $\lambda=M/N$ and taking  
$q=0$, $b=\frac{M-j}{\lambda}$ and  $a=\frac{M-j-1}{\lambda}$, so   that $d=\frac{1}{\lambda}>0,$ 
we obtain the existence of some number $C=C_{\delta} >0$ such that for $j=1,...,M-1$
\bna
\nor{u_j}{\alpha L,\frac{M-j}{\lambda}}\leq \left(C(L^{-\frac{1}{\lambda}}+\left\langle L\right\rangle^C)+\frac{1+\delta}{e\ln \alpha}\alpha^{\frac{M-1}{\lambda}} L^{1/\lambda} \right)\nor{u_j}{L,\frac{M-j-1}{\lambda}}, \quad \text{ for } \alpha>1 \quad \text{and} \quad L>0.
\ena
 For the last term of $A_0U$ (namely $\zeta _M^{-1}\partial _t ^N u_0$)  with $\lambda=M/N$ and $\delta >0$, \eqref{choix14} is satisfied, and considering now 
$q=N$, $b=0$, $a=\frac{M-1}{\lambda}$, so  $d=\frac{1}{\lambda}>0 $, 
we obtain the existence of some number $C=C_{\delta} >0$ such that
\bna
\nor{ \partial_t^N u_0}{\alpha L}\leq  \left(C(L^{-\frac{1}{\lambda}}+\left\langle L\right\rangle^C)+\frac{1+\delta}{e\ln \alpha} L^{1/\lambda} \right)\nor{u_0}{L, \frac{M-1}{\lambda}}.
\ena
It gives after summation
\bnan
\label{estimderivA}
\nor{A_0U}{{\mathcal X}_{\alpha L}}&\leq&\left(C(L^{-\frac{1}{\lambda}}+\left\langle L\right\rangle^C)+\frac{1+\delta}{e\ln \alpha}\alpha^{\frac{M-1}{\lambda}} L^{1/\lambda} \right)\nor{U}{{\mathcal X}_{L}},
\enan
uniformly for $\alpha>1$ and $L>0$.\\

Therefore,  from equation \eqref{PP0}, \eqref{PP1}, \eqref{gainXs} and considering the estimate \eqref{estimderivA} with  $L=L(s)$, $\alpha=\frac{L(s')}{L(s)}=e^{r(s-s')} >1$  and $s'<s$.   Hence, 
for $0\le s'<s\le 1$, 
\begin{eqnarray}
\nor{A_0U}{X_{s'}}&\leq & e^{-\tau (s-s')} \left(C(L_1^{-\frac{1}{\lambda}}+\left\langle e^{r}L_1\right\rangle^C)+(1+\delta)\frac{e^{\frac{M-1}{\lambda}r(s-s')} e^{r\frac{1-s}{\lambda}}L_1^{1/\lambda}}{e r(s-s')}\right)\nor{U}{X_{s}} \nonumber\\
&\leq & \left(Ce^{-\tau (s-s')}(L_1^{-\frac{1}{\lambda}}+e^{rC})+(1+\delta)e^{rN}\frac{L_1^{1/\lambda}}{er(s-s')}\right)\nor{U}{X_{s}} \nonumber\\
&\leq &\left(\frac{e^{-1}}{\tau (s-s')}C(L_1^{-\frac{1}{\lambda}}+e^{rC})+(1+\delta)\frac{e^{rN}L_1^{1/\lambda}}{er(s-s')}\right)\nor{U}{X_{s}}
\label{SSS1}
\end{eqnarray}
where we have used $0<s-s'\leq 1$, $0<L_1<1/4$ and
\bnan
\label{estimexp}
e^{-\tau (s-s')} =\frac{\tau (s-s')e^{-\tau (s-s')}}{\tau (s-s')}\leq \frac{e^{-1}}{\tau (s-s')},
\enan
since $te^{-t}\le e^{-1}$ for $t\ge 0$. 
 Minimizing the constant in the second term of the right hand side of \eqref{SSS1} 
  leads to the choice $r=1/N$. 
(Note that the initial space $X_{1}={\mathcal X}_{L_1}$ is independent of the choice of $r$.)
We arrive at the estimate
\bna
\nor{A_0U}{X_{s'}}&\leq &\left(\frac{Ce^{-1}(L_1^{-\frac{1}{\lambda}}+e^{ C/N})}{\tau}  +(1+\delta)N L_1^{1/\lambda}\right)\frac{1}{s-s'}\nor{U}{X_{s}} \cdot
\ena
By \eqref{choix14}, selecting $\tau_0$ large enough so that $\e _0:=\frac{Ce^{-1}(L_1^{-\frac{1}{\lambda}}+e^{ C/N})}{\tau_0}  +(1+\delta)N L_1^{1/\lambda}<1/4$. 
 This  completes the proof of  Lemma \ref{lmAder}. 
\enp 

\begin{lemma}
\label{lmArest}
Let $\e>0$, $r=\frac{1}{N}$ and $L_1>0$. There exists $\tau_0>0$ such that we have the estimates
 \bna
\nor{A_RU}{X_{s'}}&\leq &\frac{\e}{s-s'}\nor{U}{X_{s}} \quad \forall U\in X_s,
\ena
for all $\tau \ge \tau _0$ and all $s,s'$ with  $0\leq s'<s\leq 1$.
\end{lemma}
\bnp
Using \eqref{Gacroissance+}, we first get that there exists $C>0$ (depending on all the previous constants $L_1$, $M$,...) 
such that for $L\in [L_1,e^rL_1]$, ($|\zeta_M|=1$),
\bnan
\label{estimderivAr}
\nor{A_RU}{{\mathcal X}_{L}}= |\zeta_M|^{-1}\nor{\sum_{j=0}^{M-1} \zeta_ju_j}{L}\leq  \sum_{j=0}^{M-1} \nor{ \zeta_j u_j}{L} \leq C \sum_{j=0}^{M-1} \nor{ u_j}{L,\frac{M-j-1}{\lambda}}=  C\nor{U}{{\mathcal X}_{L}}.
\enan
Applying the previous estimate to $L=L(s')$ and using \eqref{PP0} and \eqref{estimexp}, we obtain
\begin{multline*}
\nor{A_RU}{X_{s'}}
= e^{-\tau (1-s')}\nor{A_RU}{\mathcal{X}_{L(s')}}
\leq \\
C e^{-\tau (1-s')}\nor{U}{{\mathcal X}_{L(s')}}= 
C e^{-\tau (s-s')}\nor{U}{X_{s'}}\leq  C\frac{e^{-1}}{\tau (s-s')} \nor{U}{X_{s}}.
\end{multline*}
It gives the result for $\tau_0$ large enough.
\enp
\begin{lemma}
\label{lemmeF}
Let $f$ be as in \eqref{AB1}-\eqref{AB4}, and let 
$F(x,U)=\left(\begin{array}{c}0\\ -f(x,u_0,u_1,\dots, u_{M-1})\end{array}\right)$ for $x\in [-1,1]$ and $U=(u_0,u_1,\dots, u_{M-1})\in L^\infty (K)^M$ with 
$\sup_{i=0,\dots,M-1} (\Vert u_i\Vert _{L^\infty (K) } )<4$.
Let $r=1/N$, $L_1>0$, and $\e>0$.
Then there exists $\tau _0>0$ (large enough) such that for any $\tau \ge \tau _0$, 
there exists $D>0$ (small enough) such that we have the estimates
\ba
\nor{F(x,U)}{X_{s'}}&\leq &\frac{\e}{s-s'}\nor{U}{X_s},  \label{R1}\\
\nor{F(x,U)-F(x,V)}{X_{s'}}&\leq& \frac{\e}{s-s'} \nor{U-V}{X_s} \label{R2}
\ea
for $0\leq s'<s\leq 1$, and $U=(u_0,u_1,\dots, u_{M-1})\in X_s,V=(v_0,v_1,\dots, v_{M-1})\in X_s$ with
\be
\nor{U}{X_{s}}\le D, \ \nor{V}{X_{s}}\leq D. \label{R3}
\ee
Furthermore, for $0\le s\le 1$ and $U\in X_s$ with 
$\Vert U\Vert _{X_s} \le D$, the map $x\in [-1,1]\to F(x,U)\in X_s$ is continuous. 
\end{lemma}
\bnp
The assumption \eqref{AB1} gives $F(x,0)=0$ and therefore \eqref{R1} follows from \eqref{R2}. Thus it is sufficient to prove \eqref{R2}. Pick 
$0\le s'<s\le 1$, $D>0$ and $U,V\in X_s$ satisfying \eqref{R3}. Then, the definition \eqref{formfbis} of $f$ gives 
\begin{eqnarray*}
\Vert F(x,U)-F(x,V) \Vert _{X_{s'}}
&=& 
\left\| 
-\left( 
\begin{array}{c}
0\\
f(x,U)-f(x,V)
\end{array}
\right) 
\right\| _{X_{s'}}\\
&= &  e^{-\tau (1-s') } \Vert f(x,U)-f(x,V)\Vert _{L(s')} \\
&\le& e^{-\tau (1-s')} \sum_{|\vec{p}|>0} \Vert  A_{\vec{p}}(x)  [\prod_{j=0}^{M-1}u_j^{p_j}  -\prod_{j=0}^{M-1}v_j^{p_j}] \Vert  _{L(s')}\\
&\le& e^{-\tau (1-s')}\sum_{|\vec{p}|>0} \vert A_{\vec{p}}(x)\vert \sum_{j=0}^{M-1}  \Vert u_j^{p_j}-v_j^{p_j}\Vert _{L(s')} \prod_{i\neq j}\left(\Vert u_i \Vert _{L(s')}^{p_i}+\Vert v_i \Vert _{L(s')}^{p_i}\right) \end{eqnarray*} 
where we used the triangle inequality, Lemma \ref{algebre} and an iteration argument. 
Note that, by \eqref{Gacroissance+}, we have for a constant ${\widehat C}={\widehat C}(L_1, M)\ge 1$ and any $0\le s'<1$
\be
\label{TTT1}
\sum_{i=0}^{M-1}\nor{u_i}{L(s')}\leq {\widehat C }\sum_{i=0}^{M-1} \nor{u_i}{L(s'),\frac{M-i-1}{\lambda}}
\leq {\widehat C} e^{  \tau(1-s')}\nor{U}{X_{s'}}\leq  {\widehat C}D e^{\tau}, 
\ee
and similarly $\sum_{i=0}^{M-1}\nor{v_i}{L(s')}\leq {\widehat C}D e^{\tau}$. Using again Lemma \ref{algebre}, for $j= 0,\dots,M-1$,
we obtain
\begin{eqnarray*}
\Vert u_j^{p_j} - v_j^{p_j} \Vert _{L(s')} 
&=& \Vert (u_j -v_j )(u_j^{p_j-1}  +u_j^{p_j-2}v_j +\cdots + v_j^{p_j-1} ) \Vert _{L(s')}  \\
&\le& \Vert u_j -v_j \Vert _{L(s')} 
\left( 
\Vert u_j\Vert _{L(s')}^{p_j-1} + \Vert u_j\Vert _{L(s')} ^{p_j-2} \Vert v_j \Vert _{L(s')}  + \cdots + 
\Vert v_j \Vert _{L(s')} ^{p_j-1} 
\right) \\
&\le& p_j( {\widehat C} De^{\tau})^{p_j-1} \Vert u_j-v_j\Vert _{L(s')}.
\end{eqnarray*}
It follows that 
\begin{eqnarray}
\Vert F(x,U) - F(x,V) \Vert _{X_{s'}} &\le&2^{M-1}e^{-\tau (1-s')}\sum_{|\vec{p}|>0} \vert A_{\vec{p}}(x)\vert \sum_{j=0}^{M-1} p_j \Vert u_j-v_j\Vert _{L(s')}( {\widehat C}De^{\tau})^{|\vec{p}|-1}\nonumber\\
&\le& C(L_1,N,M) \Vert U-V\Vert _{X_{s'}}\sum_{|\vec{p}|>0} \vert A_{\vec{p}}(x)\vert \sum_{j=0}^{M-1} p_j({\widehat C}De^{\tau})^{|\vec{p}|-1} \nonumber  \\
\label{ABC4}&=:&  C(L_1,N,M)  \Vert U-V\Vert _{X_{s'}} S.
\end{eqnarray}
where we have used \eqref{Gacroissance+}. Let us estimate the term
$$S:=\sum_{|\vec{p}|>0} \vert A_{\vec{p}}(x)\vert \sum_{j=0}^{M-1} p_j({\widehat C}De^{\tau})^{|\vec{p}|-1},$$  set $C_a':=C_a/(1-b_2^{-1})$. Estimate \eqref{estimAp} becomes
\[
|A_{\vec{p}}(x)| \le \frac{C_a}{b^{|\vec{p}|}} \, \frac{1}{1-\frac{|x|}{b_2}} \le \frac{C_a'}{b^{|\vec{p}|}}, 
\quad \textrm{ for } |x |\le 1, 
\]
so that we have 
\begin{equation*}
S\le \sum_{|\vec{p}|>0}\frac{C_a'}{b^{|\vec{p}|}}|\vec{p}| (CDe^{\tau})^{|\vec{p}|-1}\le \sum_{R=1}^{+\infty}\sum_{|\vec{p}|=R}\frac{C_a'}{b^{R}}R (CDe^{\tau})^{R-1}.
\end{equation*}
Using the fact that $\sum_{|\vec{p}|=R} 1 \le C(R+1)^{M-1}$ and that for $0<\rho<1$,
\[
\sum_{R=0}^\infty (R+1)\cdots (R+M) \rho ^R =\frac{d^M}{d\rho ^M }\sum_{R=0}^\infty \rho ^{R+M} =\frac{d^M}{d\rho ^M}
\frac{\rho ^M}{1-\rho } =P(\rho ) (1-\rho )^{-M-1},\]   for some $ P\in \R [X]$. We obtain 
\begin{eqnarray*}
S&\le& C(C_a',b,M)\sum_{R=1}^{+\infty}(R+1)^{M-1} R \left(\frac{\widehat{C}De^{\tau}}{b}\right)^{R-1} \\
&\le& C(C_a',b,M)\sum_{R=0}^{+\infty} (R+1)\cdots (R+M) \left(\frac{\widehat{C}De^{\tau}}{b}\right)^{R}\\
&\le& C(C_a',b,M) \left(1-\frac{\widehat{C}De^\tau}{b}\right)^{-M-1} \le C(C_a',b,M),
\end{eqnarray*} 
provided that 
\be
\label{AB20}
D\le \frac{ b e^{-\tau} }{2\widehat{C}(L_1,M)}\cdot
\ee
Therefore, using \eqref{gainXs}, \eqref{estimexp} and \eqref{ABC4}, we infer that \eqref{AB20} implies
\begin{eqnarray*}
\Vert F(x,U)-F(x,V)\Vert _{X_{s'}} 
&\le&  C(C_a',b,M,N,L_1) \Vert U-V\Vert _{X_{s'}}  \\
 &\le& C(C_a',b,M,N,L_1) e^{-\tau (s-s')} \Vert U-V\Vert _{X_s} \\
&\le& \frac{C(C_a',b,M,N,L_1)}{e} \, \frac{1}{\tau (s-s')} \Vert U-V\Vert _{X_s}   \cdot
\end{eqnarray*}  
To complete the proof of \eqref{R2}, it is sufficient to pick $\tau \ge \tau _0$ with $\tau _0$ such that  
$ \frac{C(C_a',b,M,N,L_1)}{e \tau_0}  \le \epsilon $,
 and $D$ as in \eqref{AB20}.\\
  
For given $0\le s\le 1$ and $U=(u_0,u_1,..., u_{M-1})\in X_s$ with 
$\Vert U\Vert _{X_s} \le D$, let us prove that the map 
$x\in [-1,1]\to F(x,U)\in X_s$ is continuous.   Pick any $x,x'\in [-1,1]$. 
From the mean value theorem, we have for $r\in \N$ such that 
$|x^R-x'^R|\le R|x-x'|$ with $R\in \N$,
\[
| A_{\vec{p}}(x)-A_{\vec{p}}(x') |\le |x-x'|\sum_{R\in \N} \frac{RC_a}{b^{|\vec{p}|}b_2^R}
=\frac{C_a}{b^{|\vec{p}|} b_2 }\left(1-\frac{1}{b_2}\right)^{-2} |x-x'|.
\]
We infer that 
\begin{eqnarray*}
\Vert F(x,U)-F(x',U)\Vert _{X_{s}} 
&=& e^{-\tau (1-s)} \Vert f(x,u_0,\cdots,u_{M-1})-f(x',u_0,\cdots,u_{M-1})\Vert _{L(s)} \\
&\le& \sum_{|\vec{p}|>0} |A_{\vec{p}}(x)-A_{\vec{p}}(x')| 
\Vert u_0^{p_0}u_1^{p_1},\cdots u_{M-1}^{p_{M-1}}\Vert _{L(s)} \\
&\le& 
\frac{C_a}{b_2}   \left(1-\frac{1}{b_2}\right)^{-2} |x-x'|
\sum_{|\vec{p}|>0} \frac{\left(\widehat{C}(L_1,M)De^{\tau}\right)^{|\vec{p}|}}{b^{|\vec{p}|}},
\end{eqnarray*}
due to Lemma \ref{algebre} and \eqref{TTT1}, the last series being convergent when \eqref{AB20} is fulfilled. This proves the continuity of the map $x\in [-1,1]\to F(x,U)\in X_s$.
\enp

We are in a position to prove Theorem \ref{thm2}.
\bnp[Proof of Theorem \ref{thm2}]
By Remark \ref{rkscaling}, we can assume $|\zeta_M|=1$. Let $f=f(x,\vec{y})$ be as in \eqref{AB1}-\eqref{AB4}, 
$ - \infty < t_1 < t_2< + \infty$ and $R>(4N)^{\lambda}$. Pick $k_0,k_1,\cdots,k_{M-1}\in G^{\lambda}( [t_1,t_2] )$ such that 
\eqref{AA} holds. We will show that Theorem \ref{thm4} can be applied provided that $C$ is small enough.   
Pick $L_1 \in (1/R, 1/(4N)^\lambda)$. Let $\eta =\eta (L_1)>0$ be as in Theorem \ref{thm4}. Let
$K_0=(k_0,k_1,\cdots,k_{M-1})$. We have to show that 
\[
\Vert K_0\Vert _{{\mathcal X}_{L_1}}=\sum_{j=0}^{M-1}\nor{k_i}{L_1,\frac{M-j-1}{\lambda}} \le \eta,
\]
for $C$ small enough. Thanks to \eqref{QQ1} and up to a change of $\eta(L_1)$ by a smaller constant, it is sufficient to have for any $i=0,\cdots,M-1$,
\ba
\nor{k_i}{L_1,\frac{M-1}{\lambda}}        &\le& \frac{\eta}{2}. \label{YY1}
\ea
Recall that 
\ba
\Vert f\Vert _{L_1,\frac{M-1}{\lambda}} &=& \max \left( 2^6 \Vert f\Vert _{L^\infty ( [t_1,t_2] )}, 2^3L_1^{-1} 
\sup_{t\in [t_1,t_2],n\in \N} \frac{\vert f^{(n+1)} (t)\vert }{ L_1^{ |n-\frac{M-1}{\lambda}| } \Gamma _{\lambda,\frac{M-1}{\lambda}} (n)} \right) , \label{YY3}
\ea
where 
\[
\Gamma _{\lambda, \frac{M-1}{\lambda}} (n) = 
\left\{ 
\begin{array}{ll}
2^{-5} \big( \Gamma (n+1-\frac{M-1}{\lambda}) \big) ^\lambda (1+n)^{-2}, &\textrm{ if } n> \frac{M-1}{\lambda}+1, \\
2^{-5} (n!)^{\lambda} (1+n)^{-2}, &\textrm{ if } 0\le n\le \frac{M-1}{\lambda}+1.
\end{array}
\right. 
\]


Then, if follows that \eqref{YY1} is satisfied provided that 
\begin{eqnarray}
 \Vert k_i\Vert _{L^\infty ( [t_1,t_2] )} &\le& 2^{-7}\eta, \label{YY5}\\
 \Vert k_i ^{(n+1)} \Vert _{L^\infty ( [t_1,t_2] )} &\le& 2^{-4}\eta   L_1^{1+|n-\frac{M-1}{\lambda} |} 
\Gamma _{\lambda, \frac{M-1}{\lambda}} (n), 
  \quad \forall n\in \N.  \label{YY6}
\end{eqnarray}
Since $\Gamma (n+1-\frac{M-1}{\lambda})\sim \Gamma (n+1)/ n^\frac{M-1}{\lambda} \sim n! /n^\frac{M-1}{\lambda}$ as $n\to +\infty$, we have that  $\big( \Gamma (n+1-\frac{M-1}{\lambda}) \big) ^{\lambda}\sim (n!)^\lambda/n^{M-1}$. Thus, the r.h.s. of \eqref{YY6}
is equivalent to $2^{-9} \eta L_1 ^{n+1-\frac{M-1}{\lambda}}  (n!)^\lambda n^{-(M+1)}$  as $n\to +\infty$. Using 
\eqref{AA} and   $L_1>1/R$, we have that \eqref{YY6} holds if $C$ is small enough. 
The same is true for \eqref{YY5}.\\ 

We infer from Theorem \ref{thm4} the existence of a solution $U=(y,\partial _x y,\cdots,\partial _x^{M-1} y) \in C([-1,1], X_{s_0})$ (for some 
$s_0\in (0,1)$) of \eqref{DDD10}. Let us check that $y\in C^\infty ([-1,1]\times [t_1,t_2])$.
To this end, we prove by induction on $n\in \N$ the following statement 
 \be
 \label{PPP1}
 U\in C^n([-1,1], C^k([t_1,t_2])^M),  \quad \forall k\in \N . 
 \ee
The assertion \eqref{PPP1} is true for $n=0$,  since $X_{s_0}\subset C^k([t_1,t_2])^M$ for all $k\in \N$. 
 Assume that \eqref{PPP1} is true for some $n\in \N$.
Since $A$ is a continuous linear map 
from $C^{k+N} ([t_1,t_2])^M$ into $C^k ([t_1,t_2])^M$ for all $k\in \N$, we have that 
\[
AU\in C^n([-1,1],C^k([t_1,t_2])^M), \quad   \ \forall k\in \N .
\]
On the other hand, as $f$ is analytic and hence of class $C^\infty$, we infer from \eqref{PPP1} that 
$F(x,U)\in C^n([-1,1],$ $C^k([t_1,t_2])^M )$ for all $k\in \N$.   Since $\partial _x U=AU+F(x,U)$, we obtain that 
\eqref{PPP1} is true with $n$ replaced by $n+1$. Therefore, 
$y\in C^\infty ([-1,1]\times [t_1,t_2])$.
Finally, the proof that   $y\in  G^{1,\lambda}([-1,1]\times [t_1,t_2])$,    
is given in Appendix \ref{s:A1}, which uses some estimates of the next section, with eventually a stronger smallness assumption on the initial data.\enp
\section{Correspondence between the space derivatives and the time derivatives}
\label{section3}

We would like to know the relationship between the time derivatives and the space derivatives of any solution  of a general nonlinear equation given by  
\be
\label{A1}
\partial ^{N}_t y = P y + f(x,\yx)
\ee
where $f=f(x,\yx)$ is of class $C^\infty$ on $\R ^{M+1}$.\\ 

When $f=0$ and $Py=\partial _x^My$, then it is easy to see that
\begin{equation}
\label{J1}
\partial _t ^{nN+j}Y^x=\partial _x^{nM} \partial _t ^jY^x, \quad \forall  j\in \{0, ..., N-1\}, \ \forall n\in \N.  
\end{equation}

It follows that for any $(x_0,t_0)$  the determination of the jet $(\partial _t ^n Y^x (x_0,t_0) )_{n\ge 0} $ is equivalent to the determination of the jet 
$(\partial _x ^nY^t (x_0,t_0) )_{n\ge 0}$. 
In the general case ($f=f(x,Y^x)$ and $Py=\sum_{j=0}^M\zeta _j \partial _x ^j y$), the relation \eqref{J1} may not be true. Nevertheless, there is still a 
one-to-one correspondence between 
the jet $(\partial _t ^n Y^x (x_0,t_0) )_{n\ge 0} $ and the jet $(\partial _x ^nY^t (x_0,t_0) )_{n\ge 0}$.

Introduce some notations. For given $-\infty <  t_1 \le \tau \le t_2<+\infty$, we set 
\begin{eqnarray}
{\mathcal S}&:=& \{ y\in C^\infty ([-1,1]\times [t_1,t_2]): \ 
y \textrm{ satisfies }  \eqref{A1} \textrm{ on } [-1,1]\times [t_1,t_2] \} , \\
{\mathcal J}^t &:=& \{ (\partial _t ^n \yx(0,\tau ))_{n\ge 0}: \
\yx =(y,\partial_xy, ..., \partial ^{M-1}_x y), \ 
\ y\in {\mathcal S} \} \subset (\R^{M})^{\N}, \\
{\mathcal J}^x &:=& \{ (\partial _x ^n \yt(0,\tau ))_{n\ge 0}: \
\yt =(y,\partial_t y, ..., \partial ^{N-1}_t y), \ 
\ y\in {\mathcal S} \} \subset (\R^{N})^{\N}. 
\end{eqnarray}
The set ${\mathcal J}^t$ (resp. ${\mathcal J}^x$), which stands  for the set of sequences of
{\em time derivatives} (resp. {\em space derivatives}) at $(0,\tau)$ of $Y^x$ (resp. $Y^t$) for smooth solutions $y$  of \eqref{A1}, is a subset
of $(\R ^M) ^\N$  (resp. $(\R ^N) ^\N$) that we will not determine explicitly. 


\begin{proposition}
\label{prop1}
Let $-\infty < t_1 \le \tau \le t_2 < +\infty$ and 
assume that $f\in C^\infty (\R ^{M+1})$. Then there exists a map
$\Lambda :  (\R^{N})^{\N}\to  (\R^{M})^{\N}$ whose restriction (still denoted by $\Lambda$) $\Lambda: \ {\mathcal J}^x\to {\mathcal J}^t$ is a bijection such that for 
any $y\in C^\infty ([-1,1]\times [t_1,t_2])$ satisfying \eqref{A1} on $[-1,1]\times [t_1,t_2]$, we have $(\partial _t ^n \yx (0,\tau ))_{n\ge 0} = 
\Lambda \left(  (\partial _x ^n \yt(0,\tau )) _{n\ge 0} \right)$, where 
$\yx =(y,\partial_xy, ..., \partial ^{M-1}_x y)$ and 
$\yt =(y,\partial_t y, ..., \partial ^{N-1}_t y)$. 
\end{proposition}

\begin{proof}
Proposition \ref{prop1} is a consequence of  Lemma \ref{lem1} (see below) which, roughly speaking, consists in taking sufficiently many derivatives in \eqref{A1}. 
\end{proof}
\begin{nota}
\label{defYxt}
 The space $(\R^{q+1})^{p+1}$ will be denoted $\E{p}{q}$. The current vector in $\E{p}{q}$ will be denoted $\Y{p}{q}\in \E{p}{q}$ when a confusion may occur, but very often merely $Y$ to make notations easier.
 
 For $y\in C^\infty ([-1,1]\times [t_1,t_2])$ and $p,q\in \N$, we denote the vector $\yxt{p}{q}(y):= (\yt_{q},\partial_{x}\yt_{q},\dots, \partial^{p}_x \yt_{q})\in \E{p}{q}$ with $Y^t_{q}(x,t)=(y(x,t),\partial_t y(x,t),...,\partial_t^{q}y(x,t))$ as it was defined in \eqref{Yt}. Most of the time, when only one function $y$ is concerned, we will write $\yxt{p}{q}$.
\end{nota}

\begin{lemma}
\label{lem1}
Let $f\in C^\infty (\R ^{M+1})$ and $l,k\in \N $ with $l=Nn+j$ for some $0\le j <N$ and $n\in \N$.  Then there exists a  smooth function  
 $H_{l}^{k}:\R\times \E{Mn+k-1}{N-1}\to \R$ 
 such that any solution $y\in C^\infty ([0,1]\times [t_1,t_2])$ of \eqref{A1} satisfies 
\ba
\partial_{t}^{l}\partial^{k}_{x}y=P^n \partial_{t}^{j}\partial_{x}^{k}y + H_{l}^{k}(x,\yxt{Mn+k-1}{N-1})  	
\label{B2} 
\ea
where we have used the Notations \ref{defYxt}.
\end{lemma}


We introduce first some definitions, notations, and lemmas that will be needed for the proof of Lemma \ref{lem1}. To apply Leibniz formula for $x$ in a formal way, we have to see how the derivations $\partial _x$ and $\partial _t$ operate in 
$E_{p,q}$.  This leads us to define the following operators.
\begin{nota}
\label{notaDtDx}
There is a linear operator $\dtf$ from $\E{p}{q+1}$ to $\E{p}{q}$ such that we can write $\partial_{t}\yxt{p}{q}=\dtf(\yxt{p}{q+1})$ for any smooth function.\\

Similarly, we define the operator $\dxf$ from $\E{p+1}{q}$ to $\E{p}{q}$ by the shift $\dxf (Y_{0},Y_{1},\ldots, Y_{p+1})= (Y_{1},\ldots, Y_{p+1})$ so that for any $y\in C^{\infty}([-1,1]\times [t_1,t_2])$, $\yxt{p}{q}$ being as in Notation \ref{defYxt}, we have 
\bnan
\label{propDtildex}
\dxf \yxt{k+1}{N-1}=\partial_{x}\yxt{k}{N-1}.
\enan
\end{nota}
Note that $\dtf$ can also be seen as a shift, but after a proper identification between $\E{p}{q}$ and $\E{q}{p}$.
The operator $\dxf$ depends of course on $p$ and $q$ but, since the definition is similar for each $p$, $q$, it should not lead to any confusion.

\begin{nota}
For $Y=(Y^{0}, \ldots  , Y^{i} , \ldots, Y^{k})\in  \E{k}{N-1}$, we denote
\[ I(Y) :=(Y^{0}, \ldots  ,(-1)^{i}Y^{i} , \ldots, (-1)^{k}Y^{k}) .\]
Strictly speaking, the operator $I$ depends on $k$, but since it takes the same form on each space, we will keep the same notation. The interest of this operator is that for $y\in C^{\infty}(  [-1,1] 
\times [t_1,t_2] )$ and $Y=\yxt{k}{N-1}(y)$ as in Notation \ref{defYxt},  we have 
\bnan
\label{propI}
I(Y(y))=Y(y_{-})(-x),
\enan
 where $y_{-}$ is the reflected function $y_{-}(t,x):=y(t,-x)$. 
\end{nota}

We notice that 
\begin{eqnarray}
\dtf I(Y)&=& I (\dtf Y),  \label{AS1}\\
\dxf I(Y)&=&-I(\dxf Y).\label{DxfvsI}
\end{eqnarray}
\begin{lemma}
\label{lmdxdtformel}
Let $p, q\in \N$ and let  $M: \R\times \E{p}{q}\to \R$ be a smooth function. Then there exist two smooth functions $M_{t}: [-1,1]\times \E{p}{q+1}\to \R$ and $M_{x}: [-1,1]\times \E{p+1}{q}\to \R$ such that for any $y\in C^{\infty}([ -1,1 ]\times [t_1,t_2])$ (not necessarily solution of \eqref{A1}), $\yxt{p}{q}$ being as in Notation \ref{defYxt}, we have
\bnan
\label{Mcompat}
\partial_{t}M(x,\yxt{p}{q})=M_{t}(x,\yxt{p}{q+1}),   \label{Mcompat1}\\
\partial_{x}M(x,\yxt{p}{q})=M_{x}(x,\yxt{p+1}{q}). \label{Mcompat2}
\enan
Moreover, if we assume that for some $\varpi, \sigma \in \{-1,1\}$, $M(-x,\varpi I(Y))=  \sigma M(x, Y)$, then we have 
\begin{eqnarray}
M_{t}(-x,\varpi I(Y))=\sigma M_{t}(x,Y),   \label{VVV1}\\
M_{x}(-x,\varpi I(Y))=-\sigma M_{x}(x,Y).\label{VVV2}
\end{eqnarray}
\end{lemma}
\bnp  By the chain rule, we have
\begin{gather} 
\partial_{t}M(x,\yxt{p}{q})= \nabla M(x,\yxt{p}{q})\cdot \left(\begin{array}{cc}  0 \\ \partial_{t} Y_{p,q}^{x,t}  \end{array} \right) .
\end{gather}

Using the operator $\dtf$ introduced in Notation \ref{notaDtDx}, we can 
define $M_{t}$ as 
\begin{gather} 
\label{formMt}
M_{t}(x,Y_{p,q+1}) := \nabla M(x,Y_{p,q})\cdot \left(\begin{array}{cc}  0 \\ \dtf(Y_{p,q+1}) \end{array} \right), \quad \forall x\in [-1,1], \ \forall  \,Y_{p,q+1}\in E_{p,q+1}.
\end{gather}
For $Y_{p,q+1} = (Y_0, Y_1, ..., Y_p)\in E_{p,q+1}$, we have denoted $Y_{p,q}$ the vector in $E_{p,q}$ obtained by selecting the $q+1$ first components of each vector $Y_i$ for $0\le i\le p$. 
 With this definition, \eqref{Mcompat1} is true for any smooth function $y$. \\

Similarly, we define the function $M_{x}$ by
\begin{gather} 
\label{formMx}
M_{x}(x,Y_{p+1,q}) := \nabla M(x,Y_{p,q})\cdot \left(\begin{array}{cc}  1 \\ \dxf(Y_{p+1,q}) \end{array} \right), \quad \forall x\in [-1,1], \ \forall  \,Y_{p+1,q}\in E_{p+1,q},
\end{gather}
and it can be seen that \eqref{Mcompat2} is true for any smooth function $y$.

To prove \eqref{VVV1}, we take the derivative w.r.t. $Y$ in the relation  $M(-x,\varpi I(Y))=  \sigma M(x, Y)$ to obtain for any $Z\in \E{p}{q}$,
\begin{eqnarray*}
\nabla M(-x, \varpi  I(Y))\cdot \left(\begin{array}{cc}  0 \\ \varpi I(Z) \end{array}\right) &=&\sigma   \nabla M(x,Y)\cdot \left(\begin{array}{cc}  0 \\ Z  \end{array}\right).
\end{eqnarray*}
Let  $Y_{p,q+1} \in \E{p}{q+1}$. Taking $Y=Y_{p,q}$ and $Z=\dtf(Y_{p,q+1})$ and noticing that $I(\dtf(Y_{p,q+1}))=\dtf(I(Y_{p,q+1}))$ by \eqref{AS1}, we obtain
\begin{eqnarray*}
\nabla M(-x, \varpi  I(Y_{p,q}))\cdot \left(\begin{array}{cc}  0 \\ \varpi \dtf(I(Y_{p,q+1})) \end{array}\right) &=& 
\sigma  \nabla M(x, Y_{p,q})\cdot \left(\begin{array}{cc}  0 \\ \dtf(Y_{p,q+1}) \end{array}\right),
\end{eqnarray*}
which is exactly \eqref{VVV1}. The proof of \eqref{VVV2} is similar and is omitted.
\enp

\begin{proof}[Proof of Lemma \ref{lem1}]We will actually prove the slightly stronger result that for $k\in \N$ and 
$l=Nn+j$ for some $0\le j < N$ and $n\in \N$, each $H_{l}^{k}$ is actually a function of $x$ and $Y\in \E{Mn+k-1}{j}$ so that \eqref{B2} is satisfied with $H_{l}^{k}(x,\yxt{Mn+k-1}{N-1})$ replaced by $H_{l}^{k}(x,\yxt{Mn+k-1}{j})$.\\

The case $n=0$, $0\le j <N$ is trivial since we can take $H_{l}^{k}=0$. \\

For some technical reasons, we will also need to deal with the case $n=1$.
Letting $f_{0}(x,\yx_{M-1}):= f(x,\yx_{M-1})$,  we apply the operator  $\partial _t^j$  in  \eqref{A1} for $ 0\le j<N$ to get
\begin{gather}\label{eqderj} \partial_{t}^{N+j}y= P\partial_{t}^{j}y + \partial_{t}^{j} f(x,\yxt{M-1}{0}) \end{gather}
We want to define  functions $f_{j}$ so that for any $y\in C^\infty ([-1,1]\times [t_1,t_2])$, we have
\begin{gather} 
\label{formfjsol}
f_{j}(x,\yxt{M-1}{j})= \partial^{j}_{t}f_{0}(x,\yxt{M-1}{0})\textnormal{ for }0<j<N.
\end{gather}
Using Notation \ref{notaDtDx}  we can define $f_{j}$ iteratively by 
\begin{gather} 
\label{formfjbis}
f_{j}(x,Y_{M-1,j}) := \nabla f_{j-1}(x,Y_{M-1,j-1})\cdot \left(\begin{array}{cc}  0 \\ \dtf(Y_{M-1,j}) \end{array} \right),
\end{gather}
so that  by Lemma \ref{lmdxdtformel}, \eqref{formfjsol}  is true for any $y\in C^\infty ([-1,1]\times [t_1,t_2])$. 
Now that the $f_{j}$ are defined, we see that any solution $y$ of \eqref{A1} satisfies \eqref{eqderj} and also
\begin{gather} \label{eqnfj}\partial_{t}^{N+j}y= P\partial_{t}^{j}y + f_{j}(x,\yxt{M-1}{j}) . \end{gather}
In particular, defining $H_{N+j}^{0} :=f_{j}$, we see that the case $k=0$, $l=N+j$ with $0\leq j<N$ is treated.\\


Applying $\partial_{x}^{k}$ in \eqref{eqnfj}   and using Lemma \ref{lmdxdtformel}, we can find some smooth functions $H_{N+j}^{k}$ such that
\begin{gather} \label{caseN1}\partial_{t}^{N+j}\partial_{x}^{k}y= P\partial_{t}^{j}\partial_{x}^{k}y + H_{N+j}^{k}(x,\yxt{M-1+k}{j}). \end{gather}
The $H_{N+j}^{k}$ are defined by the iteration formula
\begin{gather} 
\label{formfHk1n1}
H_{N+j}^{k}(x,Y_{M-1+k,j})   := \nabla H_{N+j}^{k-1}(x,Y_{M-1+k-1,j})\cdot \left(\begin{array}{cc}  1 \\ \dxf(Y_{M-1+k,j}) \end{array} \right),
\end{gather}

this is the case $n=1$ of the Lemma.

\bigskip
Now, we construct the functions $H_{Nn+j}^{k}$ by induction on $n$. Assume that the \eqref{B2} is satisfied for some $n\in \N ^*$, for
all $l=Nn+j$ with  $0\le j <N$ and all $k\in \N$. Applying  the operator $\partial _t^N$ in \eqref{B2} yields
\ba
\partial_{t}^{l+N}\partial^{k}_{x}y=P^n \partial_{t}^{j}\partial_{x}^{k}\partial_{t}^{N}y + \partial_{t}^{N} H_{l}^{k}(x,\yxt{Mn+k-1}{j})  .
\ea
Using equation \eqref{A1}, we obtain
\ba
\label{expren1}
\partial_{t}^{l+N}\partial^{k}_{x}y&=&P^{n+1} \partial_{t}^{j}\partial_{x}^{k} y+P^{n} \partial_{x}^{k}\partial_{t}^{j}f(x,Y_{M-1}^{x}) +  \partial_{t}^{N} H_{l}^{k}(x,\yxt{Mn+k-1}{j}) .
\ea
So, we are led to prove that the last two terms $P^{n} \partial_{x}^{k}\partial_{t}^{j}f(x,Y_{M-1}^{x}) +  \partial_{t}^{N} H_{l}^{k}(x,\yxt{Mn+k-1}{j})$ can be written as $H_{l+N}^{k}(x,\yxt{M(n+1)+k-1}{j})$. Concerning the first one, due to \eqref{formfjsol} we can write
\bnan
\label{iterterm1}
P^{n} \partial_{x}^{k}\partial_{t}^{j}f(x,Y_{M-1}^{x}) =P^{n} \partial_{x}^{k}f_{j}(x,\yxt{M-1}{j}).
\enan
Since $P^{n} \partial_{x}^{k}$ is a differential operator of order $Mn+k$ in $x$, we see by successive applications of Lemma \ref{lmdxdtformel} that the previous term can be written as a smooth function of $x$ and $\yxt{M(n+1)+k-1}{j}$.\\

By iterative applications of Lemma   \ref{lmdxdtformel},  the second term $\partial_{t}^{N} H_{l}^{k}(x,\yxt{Mn+k-1}{j})$ can be written as $F(x,\yxt{Mn+k-1}{j+N})$ for some smooth function $F$. But thanks to the case $n=1$, namely \eqref{caseN1}, for each $0\leq p\leq Mn+k-1$, $\partial_{t}^{N+j}\partial^{p}_{x}y$ can be written as $J_{N+j}^{p}(x,\yxt{M+p}{j})$ for some smooth function $J_{N+j}^{p}$. In particular, $\yxt{Mn+k-1}{N+j}$ can be written as a smooth function of $x$ and  $\yxt{M+Mn+k-1}{j}$. It follows that $\partial_{t}^{N} H_{l}^{k}(x,\yxt{Mn+k-1}{j})=F(x,\yxt{Mn+k-1}{j+N})$  can be written as a smooth function of $x$ and $\yxt{M(n+1)+k-1}{j}$. Going back to \eqref{expren1} and summing up the expression of the last two terms as functions of  $x$ and $\yxt{M(n+1)+k-1}{j}$, we can write
\ba
\partial_{t}^{N(n+1)+j}\partial^{k}_{x}y&=&P^{n+1} \partial_{t}^{j}\partial_{x}^{k} y+H^{k}_{N(n+1)+j}( x, \yxt{M(n+1)+k-1}{j})
\ea
for some smooth function $H^{k}_{N(n+1) +j}$. This is the expected result at step $n+1$.
\end{proof}

We present a few consequences of  Lemma \ref{lem1}.

\begin{nota}
\label{defGl}
Let $k\in \N$ and $l=Nn+j$ for some $0\le j <N$ and $n\in \N$. Noticing that  $P^n \partial _t^j \partial _x^k y$ can be expressed as a linear combination of variables in 
$Y_{Mn+k, N-1}^{x,t}$,  
we can define a smooth function $J_{l}^{k}: [-1,1]\times \E{Mn+k}{N-1}\to \R$  such that
\begin{equation}
J_{l}^{k}(x, \yxt{Mn+k}{N-1})=\partial_{t}^{l}\partial^{k}_{x}y=P^n \partial_{t}^{j}\partial_{x}^{k}y + H_{l}^{k}(x,\yxt{Mn+k-1}{N-1}) \label{DDD1}
\end{equation}
for any solution $y$ of \eqref{A1}.	
We define also the vector-valued functions $$J_{l}: [-1,1]\times \E{Mn+M-1}{N-1}\to \R^{M},$$ with $J_{l}=(J_{l}^{0},J_{l}^{1},\ldots,J_{l}^{M-1})$.
\end{nota}
These definitions will mainly be used at $x=0$ and $t=0$. Since the knowledge of the initial datum $Y_{0}$ and all its $x$-derivatives are sufficient to know $\yxt{Mn+k}{N-1}$ for $t=0$, $J_{l}$ has to be thought as the function that, from a sufficient amount of $x$-derivatives of the initial datum, provides $(\partial_t^{l}\yx)(0,0)$, that is the $l$ time derivative of the boundary data. More precisely, if $y$ is a solution of \eqref{A1}, we have for any $t, x$
\bnan
\label{propJl}
\partial_{t}^{l}\yx=J_{l}(x,\yxt{Mn+M-1}{N-1}).
\enan
In particular, this definition of $J_{l}$ provides a proof for Lemma \ref{lmdefJlintro} with the appropriate choice of $m(l)=Mn+{ M-1}$ if $l=Nn+j$ for some $0\le j <N$ and 
$n\in \N$.\\

The two following Lemmas are almost tautological with the definitions, but they are important to justify the relevance of the set $\Comp$.
\begin{lemma}
\label{lmcompatibility}
Assume that $y$ is a smooth solution of \eqref{W1a}-\eqref{W1c}. Then $\yt(.,t) \in \Comp$ for all $t\in[0,T] $.  
\end{lemma}
\bnp
From \eqref{W1b}, we have $B\yx(0,t)=0$ for all $t\in [0,T]$. Applying the operator $\partial _t^l$ in that equation yields $B\partial_{t}^{l}\yx(0,t) =0$ for all $t\in [0,T]$. Writing $l=Nn+j$, with $0\le j <N$, $n\in \N$, and using the fundamental property \eqref{propJl} of the function $J_{l}$, we obtain $BJ_{l}(x,\yt,\partial_{x}\yt,...,\partial_{x}^{Mn+ M-1 }\yt )_{x=0}$. It means that
$\yt(.,t)\in\Comp$ for all $t\in [0,T]$.
\enp

\begin{lemma}
\label{lmcompatibilityrecip}
Let $y$ be a smooth solution to $\partial _t^N y = P\, y + f(x , y , \partial_{x} y,...,  \partial _x^{M-1} y)$ such that $Y^{x}(0,t)\in \Comp$ for some $t\in [0,T]$. Then $\yx$  satisfies the boundary condition $B\yx(0,t)=0$. 
\end{lemma}
\bnp
We have $Y^{x}(0,t)\in \Comp$, which implies, with the choice $l=0$ \[BJ_{0}(x,\yt,\partial_{x}\yt,...,\partial_{x}^{Mn+ M-1}\yt )_{x=0}=0.\]
Using property \eqref{propJl} at time $t$ and with $x=0$ and $l=n=0$, one obtains $B\yx (0,t)=0$.
\enp
The following Lemma is needed to prove Proposition \ref{propCompsym}.

 %
\begin{lemma}
\label{lmparHn}
Assume that $M$ is even and that 
\begin{equation}
P=\sum_{j=0}^{M/2}\zeta _{2j}\partial_{x}^{2j}. \label{DDD2}
\end{equation}
\begin{enumerate}
 \item If \eqref{WWimpair} holds, then for all $l,k\in \N$ we have 
\bnan
\label{Hsym}H_{l}^{k}(-x,-I(Y))= (-1)^{k+1} H_{l}^{k}(x, Y),\quad \forall x\in [-1,1], \ \forall Y  \in \E{Mn+k-1}{N-1},\\
\label{Jsym}J_{l}^{k}(-x,-I(Y))= (-1)^{k+1} J_{l}^{k}(x, Y),\quad \forall x\in [-1,1], \ \forall Y  \in \E{Mn+k}{N-1}. 
\enan
\item If \eqref{WWpair} holds, then for all $l,k\in \N$ we have 
\bnan
\label{HsymNeum}H_{l}^{k}(-x,I(Y))= (-1)^{k} H_{l}^{k}(x, Y),\quad \forall x\in [-1,1], \ \forall Y  \in \E{Mn+k-1}{N-1},\\
\label{JsymNeum}J_{l}^{k}(-x,I(Y))= (-1)^{k} J_{l}^{k}(x, Y),\quad \forall x\in [-1,1], \ \forall Y  \in \E{Mn+k}{N-1}. 
\enan
\end{enumerate}
\end{lemma}
\bnp
To treat both cases simultaneously, we define $\varpi$ as $\varpi=-1$  (resp  $\varpi=1$) if \eqref{WWimpair} holds (resp. \eqref{WWpair} holds). Therefore, we want to prove 
\begin{eqnarray}
H_{l}^{k}(-x,\varpi I(Y)) &=& \varpi (-1)^{k} H_{l}^{k}(x, Y) \quad \forall Y\in E_{Mn+k-1,N-1},  \label{OPA1}\\
J_{l}^{k}(-x,\varpi I(Y)) &=&  \varpi (-1)^{k} J_{l}^{k}(x, Y) \quad \forall Y\in E_{Mn+k,N-1}.\label{OPA2}
\end{eqnarray}

We still denote $l=Nn+j$, where $n\in \N$ and $0\le j<N$. We first prove \eqref{OPA1} by induction on $n$. 
If $n=0$, then \eqref{OPA1} is obvious since $H_l^k=0$. 

Assume that $n=1$ so that $l=N+j$. Assume first that $k=0$. We claim that 
\ba
\label{fjsym}
f_{j}(-x,\varpi I(Y_{M-1,j}))=\varpi f_{j}(x,Y_{M-1,j}), \quad \forall Y_{M-1,j}\in E_{M-1,j}\cdot
\ea
We proceed by induction. 
For $j=0$, $f_{0}=f$, so it follows from assumption \eqref{WWimpair} or \eqref{WWpair} thanks to the choice of $\varpi$. If \eqref{fjsym} is true for $j-1$, taking derivatives with respect to $Y$, we also have, for any $Z\in E_{M-1,j-1}$,
\ba
\label{derifjsym}
\varpi  \nabla f_{j-1}(-x,\varpi I(Y_{M-1,j -1 }))\cdot \left( \begin{array}{c} 0 \\ I(Z) \end{array} \right) =\varpi \nabla f_{j-1}(x,Y_{M-1,j -1 }) \cdot  \left( \begin{array}{c} 0 \\ Z\end{array} \right) . 
\ea
By \eqref{AS1}, \eqref{formfjbis} and \eqref{derifjsym}, we have 
\begin{eqnarray*}
f_j(-x, \varpi I(Y_{M-1,j}))
&=& \nabla f_{j-1} (-x, \varpi I(Y_{M-1,j-1}))\cdot \left( \begin{array}{c}0\\ \dtf (\varpi I (Y_{M-1,j} )  )  \end{array}  \right) \\
&=& \varpi  \nabla f_{j-1} (-x, \varpi I(Y_{M-1,j-1}))\cdot \left( \begin{array}{c}0\\    I( \dtf (  Y_{M-1,j} )  )  \end{array}  \right) \\
&=& \varpi\nabla f_{j-1} (x, Y_{M-1,j-1})\cdot \left( \begin{array}{c}0\\  \dtf (Y_{M-1,j})  \end{array}  \right) \\
&=& \varpi f_{j}(x,Y_{M-1,j} ).
\end{eqnarray*}
Assume now that \eqref{OPA1} is true for $0\le j< N$ and at step $k-1$, i.e. 
\[ 
H_{N+j}^{k-1}(-x,\varpi I(Y_{M-1+k-1,j} )) = \varpi (-1)^{k-1} H_{N+j}^{k-1}(x, Y_{M-1+k-1,j}).
 \]
Taking derivatives with respect to $x$ and $Y$, it gives  for any $Z\in E_{M-1+k-1,j}$,
\[ \nabla H_{N+j}^{k-1}(-x,\varpi I(Y_{M-1+k-1,j}))\cdot\left( \begin{array}{c} -1\\  \varpi I(Z) \end{array} \right) = \varpi (-1)^{k-1} \nabla H_{N+j}^{k-1}(x, Y_{M-1+k-1,j})
\cdot\left( \begin{array}{c} 1\\ Z \end{array} \right)  .\]
Combined with \eqref{formfHk1n1} and \eqref{DxfvsI}, this gives 
 \bna
  H_{N+j}^{k}( -x,\varpi I(Y_{M-1+k,j}))
 &=& 
  \nabla H_{N+j}^{k-1}(-x,\varpi I(Y_{M-1+k-1,j}))\cdot \left(\begin{array}{cc}  1 \\ \dxf( \varpi  I(Y_{M-1+k,j})) \end{array} \right)\\
 &=& 
  \nabla H_{N+j}^{k-1}(-x,\varpi I(Y_{M-1+k-1,j}))\cdot \left(\begin{array}{cc}  1 \\ -\varpi I(\dxf(Y_{M-1+k,j})) \end{array} \right)\\
&=& 
\varpi (-1)^{k} \nabla H_{N+j}^{k-1}(x,Y_{M-1+k-1,j})\cdot \left(\begin{array}{cc}  1 \\ \dxf(Y_{M-1+k,j} ) \end{array} \right)\\
&=&
\varpi (-1)^{k}   H_{N+j}^{k}( x,Y_{M-1+k,j}).
 \ena
Thus \eqref{OPA1} is proved for $n=1$. 
Assume that \eqref{OPA1} is true for $l=Nn+j$, with $0\le j<N$ and $n\in \N^*$, and for $k\in \N$. Let us prove that \eqref{OPA1} is also true for $l+N$ and $k$. From  
\eqref{expren1}-\eqref{iterterm1}, we have that 
\begin{equation}
H_{l+N}^k (x, Y_{M(n+1)+k-1,j}^{x,t} )= P^n \partial _x ^k f_j (x, Y_{M-1,j}^{x,t} ) + \partial _t^N H_l^{k} (x, Y_{M(n+1) +k-1,j}^{x,t} ). 
\label{OPA3}
\end{equation}

By Lemma \ref{lmdxdtformel}, \eqref{DDD2} and \eqref{fjsym}, we infer that the first term 
$P^{n} \partial_{x}^{k}f_{j}(x,\yxt{M-1}{j})$ can be written as $G(x,\yxt{M(n+1)+k-1}{j})$ where $G$ satisfies $G(-x,\varpi I(Y))=(-1)^{k}\varpi G(x,Y)$.

The second term $\partial_{t}^{N} H_{l}^{k}(x,\yxt{Mn+k-1}{j})$, by an application of  Lemma \ref{lmdxdtformel}, can be written as $F(\yxt{Mn+k-1}{N+j})$ for a smooth function $F$ that satisfies the same parity property as $H_{l}^{k}$, that is $F(-x,\varpi I(Y))= \varpi (-1)^{k} F(x, Y)$.\\

 But the case $n=1$ (see \eqref{caseN1}) gives that, for each $0\leq p\leq Mn+k-1$, $\partial_{t}^{N+j}\partial^{p}_{x}y$ can be written as $J_{N+j}^{p}(\yxt{M+p}{j})$ for some smooth function $J_{N+j}^{p}$ that satisfies $J_{N+j}^{p}(-x,\varpi I(Y))= \varpi (-1)^{p} J_{N+j}^{p}(x, Y)$. In particular, $\yxt{Mn+k-1}{N+j}$ can be written  as $K(x,\yxt{M+Mn+k-1}{N-1})$ (the components of $K$ are the $J_{N+j}^{p}(x,Y)$). Therefore, the symmetry properties of $J_{N+j}^{p}$ imply $K(-x,\varpi I(Y))=\varpi I(K(x,Y))$. In particular, we can write 
 \[\partial_{t}^{N} H_{l}^{k}(x,\yxt{Mn+k-1}{j})=F(x,\yxt{Mn+k-1}{j+N})=F(x,K(x,\yxt{M+Mn+k-1}{N-1})).\]
  Summarizing the symmetry properties of $K$ and $F$, we obtain
 \bna
 F(-x,K(-x,\varpi I(Y)))=F(-x,\varpi I(K(x,Y))=\varpi (-1)^{k} F(x, K(x,Y)).
 \ena
This is the expected result, and it completes the proof of \eqref{OPA1}.


To prove \eqref{OPA2}, we use \eqref{DDD1} and \eqref{OPA1}. Thus it remains to establish the symmetry property for the term $P^n\partial _ t^j\partial _x^k y$ for any 
smooth function $y$. This follows at once from \eqref{propI} and \eqref{DDD2}.
\enp

Next, we relate the behaviors as $n\to +\infty$ of the jets  $(\partial _x ^n \yt(0,\tau))_{n\ge 0}$ and
$(\partial _t ^n \yx(0,\tau ))_{n\ge 0}$.

 To do that, we assume that in \eqref{A1} the nonlinear term reads 
\be
f(x,y_0,y_1,...,y_{M-1})= \sum_{ (\vec{p},r)\in \N^{M+1} }a_{\vec{p},r} y_0^{p_0} y_1^{p_1}\dots y_{M-1}^{p_{M-1}} x^r  \quad \forall (x,y_0, ..., y_{M-1}) \in (-4,4)^{M +1 }, 
\ee

where the coefficients $a_{\vec{p},r}$, $(\vec{p},r)\in \N^{M+1}$, satisfy \eqref{AB3}-\eqref{AB4}. 

For $x\in (-1,+\infty)$, we denote $x!=\Gamma (x+1)$, where $\Gamma (x)=\int_0^\infty t^{x-1}e^{-t}dt$ is the Gamma function. Then $(x+1)!=(x+1) (x! )$ for $x>-1$. We also set $\left( \begin{array}{c} y \\ x \end{array} \right) 
=\frac{y!}{x! (y-x)!}$ for $y\ge x\ge 0$.

\begin{proposition}
\label{prop10}
Let $-\infty < t_1 \le \tau \le t_2 < +\infty$ and $f=f(x,y_0,y_1,...,y_{M-1})$ 
be as in  \eqref{AB1}-\eqref{AB2} with the 
coefficients  $a_{\vec{p},r}$, $(\vec{p},r)\in \N ^{M+1}$, satisfying \eqref{AB3}-\eqref{AB4}.
 Assume that $|\zeta _M |=1$. Let $\widetilde R>4$,  $R,R' \in \R$ with $4<R'<R<\min (\widetilde R,b_2)$, and  $\mu > M+1$. 
Then there exists some number $\widetilde C>0$ such that for any 
$C\in (0,\widetilde C ]$, one can find a number $C' =C'(C,R,R' , \mu )>0$ with $\lim _{C\to 0^+} C'(C,R,R' ,\mu) =0$ such that 
\begin{enumerate}
\item for any function 
 $y\in C^\infty([-1,1]\times [t_1,t_2] )$ satisfying
 \eqref{A1} on $[-1,1]\times [t_1,t_2]$ and 
\be
\yt(x,\tau )=Y_0(x)=\sum_{k=0}^\infty A_k \frac{x^k}{k!} , \quad \forall x\in [-1,1] \label{D1}\\
\ee
for some $ Y_0\in ({\mathcal R}_{\widetilde R,C})^N$,
we have
\be
\vert \partial _x^k \partial _t ^n y(0, \tau ) \vert \le C'  \frac{ (\lambda n+k)!}{R^kR'^{\lambda n} (\lambda n+k+1)^\mu} \ \forall k, n \in \N ;
\label{D2} 
\ee
\item there exists an application 
\[
\Lambda^{\infty} : (A_k)_{k\ge 0} \in( {\mathcal N}_{\tilde R,C})^N  \to   (d_n^k)_{ (n,k)\in \N ^2}  \in \R^{\N^{2}}
\]
such that if there exists a solution $y$ of  \eqref{A1} on $[-1,1]\times [t_1,t_2]$ with $Y^t(x,\tau )=\sum _{k\ge 0}  A_k \frac{x^{k}}{k!}$, then 
$\partial_{x}^{k} \partial_{t}^{n}y  (0, \tau ) =d_{n}^k$ for all $(n,k)\in \N ^2$ (without knowing {\em a priori} the existence of such solution). Moreover, we have \be
\vert d_{n}^k \vert \le C'  \frac{ (\lambda n+k)!}{R^kR'^{\lambda n} (\lambda n+k+1)^\mu} \qquad  \forall k, n \in \N .
\label{Gd} 
\ee
\item The application $\Lambda^{\infty}$ satisfies the following property: Assume $y$ is a smooth solution of \eqref{A1} such that there exists $(d_n^k)_{ (n,k)\in \N ^2} =\Lambda^{\infty}((A_k)_{k\ge 0})$ for some $(A_k)_{k\ge 0} \in( {\mathcal N}_{\tilde R,C})^N$ so that $\partial_{x}^{k} \partial_{t}^{n}y  (0, \tau ) =d_{n}^k$ for $k=0,\dots,M-1$ and $n\in \N$. Then $\partial_{x}^{k} \partial_{t}^{n}y  (0, \tau ) =d_{n}^k$ for all  $(n,k)\in \N ^2$.
\end{enumerate}
\end{proposition}

 We shall need several lemmas and give the proof of Proposition \ref{prop10} later. 
\begin{lemma}
\label{lem2}
(see \cite[Lemma A.1]{KiNi}) For all $k,q\in \N$ and $a\in \{ 0 , ... , k+q\}$, we have 
\[
\sum_{\tiny
\begin{array}{c} 
j+p=a\\
0\le j \le k\\
0\le p\le q
\end{array}}
\left( \begin{array}{c} k\\j \end{array} \right) \, 
\left( \begin{array}{c} q\\p \end{array} \right)  =
\left( \begin{array}{c} k + q \\a \end{array} \right) . 
\]
\end{lemma}
\begin{lemma}
\label{lemma36}
For all $\lambda  \in [1,+\infty )$ and all $k,j,n,i\in \N$ with $k\ge j$ and $n\ge i$, we have 
\be
\left( \begin{array}{c} k \\ j  \end{array} \right) 
\left( \begin{array}{c} n \\ i  \end{array} \right) \le 
\lambda
\left( \begin{array}{c} k +  \lambda  n  \\  j +  \lambda  i   \end{array} \right) \cdot 
\ee
\end{lemma} 
\noindent 
\bnp[Proof of Lemma \ref{lemma36}]
Recall the relationship (see e.g. \cite{rudin}) between the Gamma function $\Gamma$  and  the Beta function $B$
defined by
$B(x,y)=\int_0^1t^{x-1}(1-t)^{y-1}dt$  for $\textrm{Re }x>0$ and $\textrm{Re } y>0$:  
\be
B(x,y)=\frac{\Gamma (x) \Gamma (y)}{\Gamma (x+y)} \cdot
\ee 
In particular, we have for $x,y\in [0,+\infty )$
\begin{multline*}
\left( \begin{array}{c} x+y \\ x  \end{array} \right)
=\frac{\Gamma (x+y+1)}{ \Gamma (x+1) \Gamma (y+1) }\\
=\frac{\Gamma (x+y+1)}{ \Gamma (x+y+2) }   B(x+1,y+1)^{-1} =    \left(  (x+y+1) \int_0^1 t^x(1-t)^y dt \right) ^{-1}.
\end{multline*}
Taking $x=j+ \lambda  i$, $y=k-j+\lambda (n-i)$, this yields 
\be
( k + \lambda n  +1 ) \left( \begin{array}{c} k+\lambda n \\ j+ \lambda i    \end{array}  \right)
=\left(  \int_0^1 t^{j+\lambda  i }  (1-t)^{    k-j + \lambda (n-i) }  dt \right)^{-1}.   
\label{LLL1}
\ee
As the right-hand side of \eqref{LLL1} is a non-decreasing function of $\lambda$, we infer that for $\lambda \ge 1$
\[
( k +  n  +1 ) \left( \begin{array}{c} k+ n \\ j+  i    \end{array}  \right)
\le  ( k + \lambda n  +1 ) \left( \begin{array}{c} k+\lambda n \\ j+ \lambda i    \end{array}  \right). 
\]
Therefore, using Lemma  \ref{lem2}, 
\[
\left( \begin{array}{c} k\\ j  \end{array} \right) 
 \left( \begin{array}{c} n\\ i \end{array} \right)
\le   \left( \begin{array}{c} k + n  \\  j +i  \end{array} \right)
\le \frac{k+ \lambda n + 1}{k+n+1 }  \left( \begin{array}{c}  k + \lambda n  \\  j + \lambda i   \end{array} \right)
\le \lambda  \left( \begin{array}{c} k + \lambda n  \\ j + \lambda i   \end{array} \right).
\] 
\enp
The following result gives  the algebra property for the mixed Gevrey spaces $G^{1,\lambda }([-1,1]\times [t_1,t_2])$.
\begin{lemma}
\label{lem3}
Let $-\infty < t_1\le t_2 <  \infty$, $(x_0,t_0)\in [-1,1]\times [t_1,t_2]$, $R,R' \in (0,+\infty ) $, $q\in \N$, $\lambda \in [1,+\infty)$,  $\mu  \in (q+2,+\infty )$,  
 $k_0,n_0\in \N$, $C_1,C_2\in (0,+\infty )$, and $y_1,y_2\in C^\infty ([-1,1]\times [t_1,t_2] )$ be such that 
\be
\vert \partial _x ^k \partial _t ^n y_i (x_0,t_0)\vert \le C_i \frac{ (\lambda n+k+q)!}{R^k R'^{\lambda n} (\lambda n+k+1)^\mu  } \quad \forall i=1,2, \quad  \forall k\in \{ 0, ..., k_0\} ,\ \forall n\in \{ 0, ..., n_0 \} .  
\label{A49}
\ee
Then we have 
\be
\vert \partial _x^k \partial _t ^n (y_1y_2) (x_0,t_0) \vert \le  K_{q,\mu}  C_1C_2 \frac{ (\lambda n+k+q )!}{R^k R'^{\lambda n} (\lambda n+k+1)^\mu } \quad  \forall k\in \{ 0, ..., k_0\} ,\ \forall n\in \{ 0, ..., n_0 \} ,
 \label{A50}
\ee
where
\[
K_{q,\mu} :=\lambda 2^{\mu -q +1} (1+q)^{2q} \sum_{j\ge 0} \sum_{i\ge 0} \frac{1}{ (\lambda  i +j+1)^{\mu -q} } <\infty  .
\]
\end{lemma}
\noindent 
{\em Proof of Lemma \ref{lem3}:} 
Using $(\lambda n+k+q)^q\leq (1+q)^q\left(1+\lambda n+k\right)^q$, we obtain 
\[
(\lambda n+k+q)!=(\lambda n+k)!\prod_{j=1}^q(\lambda n+k+j)\leq (\lambda n+k)! (\lambda n+k+q )^q\leq (1+q)^q(\lambda n+k)!\left(1+\lambda n+k\right)^q.
\]
So, denoting $\mub :=\mu-q>2$ and $\widetilde{C}_i:=(1+q)^qC_i$, we have
\be
\vert \partial _x ^k \partial _t ^n y_i (x_0,t_0)\vert \le \widetilde{C}_i \frac{ (\lambda n+k)!}{R^k R'^{\lambda n} (\lambda n+k+1)^{\mub}  }, \quad \forall i=1,2, \quad  \forall k\in \{ 0, ..., k_0\} ,\ \forall n\in \{ 0, ..., n_0 \} .  
\ee
We infer from the Leibniz rule that 
\begin{eqnarray*}
&& \vert \partial _x ^k \partial _t ^n (y_1y_2) (x_0,t_0) \vert \\
&&\qquad = 
\left\vert \sum _{0\le j\le k}  \,  \sum_{0\le i\le n}  \left( \begin{array}{c} k\\j \end{array} \right) \, 
\left( \begin{array}{c} n\\i \end{array} \right)  
(\partial _x^j\partial_t^i y_1) (x_0,t_0)  (\partial _x^{k-j} \partial _t ^{n-i} y_2) (x_0,t_0)  \right\vert \\
&&\qquad \le 
\sum _{0\le j\le k}  \,  \sum_{0\le i\le n}  \left( \begin{array}{c} k\\j \end{array} \right) \, 
\left( \begin{array}{c} n\\i \end{array} \right)  
 \frac{\widetilde{C}_1(\lambda i +j)!}{R^jR'^{\lambda i} (\lambda i+j+1)^{\mub} } \frac{\widetilde{C}_2(\lambda (n-i) +k-j )!}{R^{k-j}R'^{\lambda (n-i)}
(\lambda (n-i)+k-j+1)^{\mub}} \\
&&\qquad = \frac{\widetilde{C}_1\widetilde{C}_2}{R^kR'^{\lambda n}} (\lambda n+k) !   
\underbrace{\sum _{0\le j\le k}  \,  \sum_{0\le i\le n} 
\frac{ \left( \begin{array}{c} k\\j \end{array} \right) \, \left( \begin{array}{c} n\\i \end{array} \right) \, 
 \left( \begin{array}{c} \lambda n+k\\ \lambda i+j   \end{array} \right) ^{-1} } 
 {(\lambda i+j+1)^{\mub}(\lambda (n-i)+k-j+1)^{\mub}}}_{I} \cdot
\end{eqnarray*}
We infer from Lemma \ref{lemma36} that 
\[
\left( \begin{array}{c} k\\j \end{array} \right) \, \left( \begin{array}{c} n\\i \end{array} \right) 
\left( \begin{array}{c} \lambda n+ k\\ \lambda i+ j \end{array} \right) ^{-1}\le \lambda. 
\]
Finally, by the convexity of $x\to x^{\mub}$  on $[0,+\infty)$, we have that 
 \begin{multline*}
\sum _{0\le j\le k}  \sum_{0\le i\le n}  \frac{ (\lambda n+k+ 2)^{\mub}}{(\lambda  i+j+1)^{\mub}(\lambda (n-i)+k-j+1)^{\mub}} 
= \\ 
\sum _{0\le j\le k}\sum_{0\le i\le n}  \left( \frac{1}{ \lambda i+j+1}  + \frac{1}{\lambda (n-i)+k-j+1}  \right)^{\mub} \le\\
 2^{\mub -1} \sum _{0\le j\le k}\sum_{0\le i\le n}  \left( \frac{1}{ (\lambda i+j+1)^{\mub} }  + \frac{1}{ (\lambda (n-i)+k-j+1)^{\mub} }  \right) \le\\ 
 2^{\mub} \sum_{j\ge 0}\sum_{i\ge 0} \frac{1}{ (\lambda i+j+1)^{\mub} } <\infty,
\end{multline*}
where we used the fact that ${\mub} =\mu  - q >2$. 

It follows that 
\begin{eqnarray*}
I &\le& 2^{\mub} \lambda 
\left( \sum_{j\ge 0} \sum_{i\ge 0} \frac{1}{ (\lambda i+j+1)^{\mub}} \right) \frac{1}{ (\lambda n+k+2)^{\mub}}\\
&=&2^{\mu -q} \lambda  \left( \sum_{j\ge 0}\sum_{i\ge 0} \frac{1}{ (\lambda i+j+1)^{\mu -q}} \right) 
\frac{  (\lambda n+k+2)^q }{ (\lambda n+k+2)^\mu },
\end{eqnarray*} 
and hence the proof of Lemma \ref{lem3} is complete once we have
noticed that $ (\lambda n+k) !(\lambda n+k+2)^{q}\leq  2 (\lambda n+k+q) !$. (We used the fact that $(x+2)^q\le 2\prod_{j=1}^q (x+j)$ for all $x\ge 0$, $q\in \N^*$.) \qed 
  
{\begin{remark}
\label{rk:Leibniz}
Lemma \ref{lem3} can also be written as the existence of an application $\pi:\R^{(k_{0}+1)\times (n_{0}+1)}\times \R^{(k_{0}+1)\times (n_{0}+1)} \mapsto \R^{(k_{0}+1)\times (n_{0}+1)} $ such that, if for some $d_{1}, d_{2}\in \R^{(k_{0}+1)\times (n_{0}+1)}$ and two smooth functions $y_{1}$, $y_{2}$ satisfying $\partial _x^k \partial _t ^n y_{i} (x_0,t_0)=d_{n,i}^{k}$ , $i=1,2$ for all $k\in \{ 0, ..., k_0\}$, for all $n\in \{ 0, ..., n_0 \}$, then $\partial _x^k \partial _t ^n (y_{1}y_{2}) (x_0,t_0)=\left(\pi(d_{1},d_{2})\right)_{n}^{k}$. The definition of $\pi(d_{1},d_{2})$ is given inside of the proof by the Leibniz formula. The Lemma gives then that the estimates
\be
\left|d_{n,i}^{k}\right| \le C_{i} \frac{ (\lambda n+k+q)!}{R^k R'^{\lambda n} (\lambda n+k+1)^\mu  } \quad \forall i=1,2, \quad  \forall k\in \{ 0, ..., k_0\} ,\ \forall n\in \{ 0, ..., n_0 \} .  
\ee
imply
\be
\left| \left(\pi(d,\widetilde{d})\right)_{n}^{k}\right| \le  K_{q,\mu}  C_1C_2 \frac{ (\lambda n+k+q )!}{R^k R'^{\lambda n} (\lambda n+k+1)^\mu } \quad  \forall k\in \{ 0, ..., k_0\} ,\ \forall n\in \{ 0, ..., n_0 \}.
\ee
This equivalent way of writing the same result is consistent with the second part of Proposition \ref{prop10}.
\end{remark}}
We are now ready to complete the proof of Proposition \ref{prop10}.
\begin{proof}[Proof of Proposition \ref{prop10}]
We will prove the first part of the proposition. The construction of the application $\Lambda$  in the second part of the proposition will appear along the proof.

 Pick any number  $\mu >M+1$. We shall prove by induction on 
$n\in \N$ that 
\be
\vert \partial _x^k\partial _t ^n y(0,\tau)\vert \le C_n \frac{(\lambda n+k)!}{R^k R'^{\lambda n} (\lambda n+k+1)^\mu }, \quad  \forall k\in \N ,\label{D20} 
\ee
where $0<C_n\leq C_{n+1} \le C'< +\infty$. The value of the constant $C'$ will appear along the proof. Assume first that $n=0$. Recall that 
$Y^t (x, \tau ) =Y_0(x)=\sum_{k=0}^\infty A_k \frac{x^k}{k!}$ with 
$ \Vert  A_k \Vert _\infty  \le C\frac{k!}{\widetilde R ^k}$. Denote  
$A_k=(A_k^0, ..., A_k^{N-1})$.
Using the fact that $R<\widetilde R$, we have that for $0\le n \le N-1$,
\[
\vert \partial _x^k\partial _t ^n  y(0,\tau)\vert  
= \vert  A_k^n \vert  
\le  C\frac{ k ! }{ \widetilde R ^k} 
\le  
CD\frac{(\lambda n + k ) ! }{ R^k {R'}^{\lambda n}(\lambda n+ k+1)^\mu } 
\]
where 
\[
D : = \left(\sup_{k\in \N, 0\le n\le N-1} (\frac{R}{\widetilde R})^k
R'^{\lambda n}  (\lambda n + k +1)^\mu \frac{k!}{(\lambda n+ k)!} \right) <\infty. 
\]

Il follows that 
\eqref{D20} holds for $0\le n\le N-1$ for some constants $C_0, ..., C_{N-1}\le CD$. 

Assume now that \eqref{D20} is true up to  the rank $n-1$ for some $n\ge N$. 
Let us show that \eqref{D20} is also true at the rank $n$ for some constant $C_n>0$.  
Then, by \eqref{W1a} and  \eqref{formfbis}, we have that 
{\begin{eqnarray}
\partial _x^k \partial _t ^n y(0,\tau )
&=& \partial _x ^k \partial _t ^{n-N} \sum_{j=0}^M 
\zeta _ j \partial _x^j y  
(0, \tau) 
+ \sum_{\vec p\ne 0}
\partial _x ^k \partial _t ^{n-N} \left(  
A_{\vec{p}} (x)  y^{p_0}(\partial _x y)^{p_1} \cdots (\partial _x ^{M-1} y)^{p_{M-1}}
 \right) (0,\tau)\nonumber \\
&=:& I_1+I_2. \label{E1}
\end{eqnarray}
}
Let us estimate $I_1$ first. For $0\le j\le M$, we have that 
\begin{eqnarray*}
|\zeta _j \partial _x^{k+j}\partial _t ^{n-N} y(0, \tau )| 
&\le & |\zeta _j | C_{n-N} \frac{(\lambda (n-N)+k+j)!}{R^{k+j}R'^{\lambda (n-N)} 
(\lambda (n-N)+k+j+1 )^\mu } \\
&\le & |\zeta _j | C_{n-N} \frac{(\lambda n + k + j - M )!}{R^{k+j}R'^{\lambda n-M} 
(\lambda n + k + j - M +1 )^\mu }, 
\end{eqnarray*}
where we have used $\lambda N=M$. 
It follows that 
\ba
&&|I_1| \le \left[  C_{n-N} \left( \frac{R'}{R} \right)^M
+\sum _{j=0}^{M-1} \frac{ |\zeta _j | C_{n-N}}{(\lambda n + k+ j -M+1) \cdots (\lambda n +k) } \frac{R'^M}{R^j} 
\left( \frac{\lambda n + k +1}{\lambda n + k + j - M + 1 }\right) ^\mu \right] \nonumber
 \\
&& \qquad \times \frac{(\lambda n+ k)!}{R^kR'^{\lambda n } (\lambda n + k + 1)^\mu }  \nonumber\\
&&\le \left[  C_{n-N} \left( \frac{R'}{R} \right)^M
+\sum _{j=0}^{M-1} \frac{ |\zeta _j | C_{n-N}}{(\lambda n + j -M+1) \cdots \lambda n } \frac{R'^M}{R^j} 
\left(M+1\right) ^\mu \right] \nonumber
 \\
&& \qquad \times \frac{(\lambda n+ k)!}{R^kR'^{\lambda n } (\lambda n + k + 1)^\mu }\cdot 
\label{E2}
\ea
where we have used $k, j\geq 0$ and $\lambda n+k\geq M$ so that $ \frac{\lambda n + k +1}{\lambda n + k + j - M + 1 }\leq M+1$.
Let us estimate $I_2$. 
Since $A_{\vec p}$ does not depend on $t$, we have that 
$\partial _x^k\partial _t^m A_{\vec p}=0$ for $m\ge 1$ and $k\ge 0$. Next, for $k\ge 0$, we have that 
\[
\vert \partial _x ^k A_{\vec p}(0)\vert =k!\, \vert a_{\vec p, k}\vert 
\le \frac{C_a\ k!}{b^{|\vec p |} b_2^k}
\le \frac{\overline{C}}{ b^{| \vec p | }}\frac{ k!}{(k+1)^\mu R^k},
\]
for some constant $\overline{C}>0$ depending on $R$, $b_2$, $\mu$, since $R<b_2$.  

Note that, still by the iteration assumption \eqref{D20} at step $n-1$, for $0\le j\le M-1$ the function $\partial _x^j y$ satisfies the estimate
\begin{eqnarray*}
\vert \partial _x^k\partial _t ^m (\partial _x^j  y)(0,\tau)\vert 
&\le& C_m \frac{(\lambda m+k +j)!}{R^{k+j} R'^{\lambda m} (\lambda m+k+j +1)^\mu } \\
&\le&  \frac{C_m}{R^j} \,  \frac{(\lambda m+k + M-1 )!}{R^{k} R'^{\lambda m} (\lambda m+k+1)^\mu }  \quad  \forall k\in \N, 
\ \forall m\in \{ 0 , ... , n-1\} . 
\end{eqnarray*}

 Let $q=M-1$. Since $\mu >M+1=q+2$,  it  follows from iterated applications of Lemma \ref{lem3} that 
\ba
&&\left\vert 
\partial_x ^k\partial_t^{n-N} \big( A_{\vec{p}}y^{p_0}
(\partial _x y)^{p_1} \cdots (\partial _x ^{M-1}y)^{p_{M-1}}
\big) (0, \tau ) 
\right\vert \nonumber \\ 
&&\qquad \le   \overline{C} \left( \frac{K C_{n-N}}{b}\right) ^{|\vec p|} 
\frac{ (\lambda (n-N)+ k + M-1 )!}{R^k R'^{\lambda (n-N)} (\lambda (n-N)+k+1)^\mu} 
\prod_{j=0}^{M-1} \frac{1}{(R^j)^{p_j}} 
\label{E3}
\ea
where $K=K_{q,\mu}>0$. If, for some number $\delta \in (0,1)$, we have   
\be
C_{n-N} \le \delta \min_{0\le j \le M-1} \frac{bR^j }{K}, 
\label{E150}
\ee
then 
\[
\sum_{\vec p \ne 0 } 
\left( \frac{K C_{n-N}}{b}\right) ^{|\vec p|} \prod_{j=0}^{M-1} \frac{1}{(R^j)^{p_j}}
\le \frac{KC_{n-N}}{b} 
\left( ( 1-\delta  ) ^{-1} + \cdots +  (1-\delta )^{-M} \right) \le \frac{MK}{b(1-\delta )^M} C_{n-N} .
\]
(We considered the subcases (1) $p_0\ge 1$ and $p_1=\ldots =p_{M-1}=0$; (2) $p_0\ge 0$,  $p_1\ge 1$ and $p_2=\ldots=p_{M-1}=0$; 
(3) $p_0\ge 0$, $p_1\ge 0$, $p_2\ge 1$ and $p_3=\ldots=p_{M-1}=0$  etc.).
Gathering the previous estimates and noticing that $\lambda N=M$, it follows that 
\ba
| I_2 | &\le&  
\overline{C} \frac{ (\lambda n + k -1 )!}{R^k R'^{\lambda n-M} (\lambda n-M+k+1)^\mu} 
\frac{M K}{b(1-\delta )^{M}}C_{n-N}  \nonumber \\
&\le& 
\left[ \frac{\overline{C}MK  {R'}^M }{b(1-\delta )^M(\lambda n + k)}
\left(  \frac{\lambda n + k  + 1}{\lambda n -M + k + 1 }\right) ^\mu     \right] 
\times \frac{(\lambda n+ k)!}{R^kR'^{\lambda n } (\lambda n + k + 1)^\mu }
C_{n-N}\\
&\le& 
\left[ \frac{\overline{C}MK  {R'}^M }{b(1-\delta )^M\lambda n}
\left(  M+1\right) ^\mu     \right] 
\times \frac{(\lambda n+ k)!}{R^kR'^{\lambda n } (\lambda n + k + 1)^\mu }
C_{n-N}\cdot
\ea
where we have used again $ \frac{\lambda n + k +1}{\lambda n + k - M + 1 }\leq M+1$.
We set $C_n:=\max(\lambda_n,1) C_{n-N}$, where 
\begin{multline*}
\lambda _n :=   |\zeta _M|  \left( \frac{R'}{R} \right)^M
+\sum _{j=0}^{M-1} \frac{ |\zeta _j | }{ (\lambda n +  j -M+1)  \ldots \lambda n    }  \frac{R'^M}{R^j} 
\left(M+1\right) ^\mu  + \frac{\overline{C}MK {R'}^M}{b(1-\delta )^M\lambda n }\left(M+1\right) ^\mu .    \end{multline*}
Then \eqref{D20} holds. 
Since $ \left( \frac{R'}{R} \right)^M<1$ and $\vert \zeta_M\vert =1$, 
it is clear that $|\lambda _n| \le 1$ for $n\gg 1$, say for $n\ge n_0\ge N$.   
This yields $C_{n}\le C_{n-N}$ for all $n\ge n_0\ge N$, provided that \eqref{E150} holds for $n\le n_0$. To ensure \eqref{E150} for $n \le n_0$, it is sufficient to choose $C$ small enough (or, equivalently, $\widetilde C$ small enough). The proof by induction of \eqref{D20} is achieved.
The proof of the first part of Proposition \ref{prop10} is complete. For the second part of Proposition \ref{prop10}, we follow the proof of the first part and define the coefficients $d_{n}^{k}$
by induction on $n$.\\

For $n=0, \dots, N-1$ and $k\in \N$, if we denote $A_k=(A_k^0, A_k^1, ... , A_k^{N-1})$, then we have $\partial_x^k \partial_t^n y(0,\tau ) =A_k^n$ for any solution satisfying \eqref{D1}. So we are led to define $d_n^k:=A_k^n$.
 
 For $n\ge N$, following the proof of the previous estimates, we obtain using the notations introduced in  \eqref{E1} and Leibniz' rule
\begin{eqnarray}
I_{1}&=&\sum_{j=0}^M \zeta _ j \partial _x ^{k+j} \partial _t ^{n-N} y 
(0, \tau)  \label{KL1}\\
I_2&=& \sum_{\vec{p} \ne 0}  \ \  \sum_{k_1+\cdots + k_{M+1} =k} \ \ \sum_{n_1+\cdots + n_{M+1} =n-N}
\frac{k!}{k_1! \cdots k_{M+1}!} \frac{(n-N)!}{n_1!\cdots n_{M+1}!}   \nonumber \\
&&\left( \partial _x^{k_1}\partial _t ^{n_1} ({ A_{\vec{p}} }) \, \partial _x^{k_2} \partial _t ^{n_2} (y^{p_0}) \,  
\partial _x^{k_3}\partial _t ^{n_3} (\partial _x y)^{p_1}\cdots \partial _x^{k_{M+1}}
\partial _t ^{n_{M+1}} (\partial _x^{M-1} y)^{p_{M-1}} \right) (0,\tau) \label{KL2}
\end{eqnarray}
with, for $0\le i\le M-1$, 
\begin{eqnarray}
&&\partial _x^{ k_{i+2} } \partial _t ^{n_{i+2}} (\partial _x^i y)^{p_i} (0, \tau ) 
= \sum_{l_1+\cdots + l_{p_i} =k_{i+2}} \, \sum_{m_1+\cdots + m_{p_i} =n_{i+2}} \nonumber \\
&&\qquad\qquad \frac{k_{i+2}!}{l_1! \cdots l_{p_i}!} \frac{n_{i+2}!}{m_1!\cdots m_{p_i}!}  
\partial _x^{l_1+i}\partial _t ^{m_1} y (0, \tau ) \cdots \partial _x^{l_{p_i} +i } \partial _t ^{m_{p_i}} y (0,\tau)  . \label{KL3}
\end{eqnarray}
We define some $\widetilde{I}_1$ and $\widetilde{I}_2$ by replacing in \eqref{KL1} $\partial _x ^{k+j} \partial _t ^{n-N} y  (0, \tau)$ by $d_{n-N}^{k+j}$, and in \eqref{KL3} 
 $\partial _x^{l_j+i}\partial _t ^{m_j} y(0,\tau)$ by $d_{m_j}^{l_j+i}$, where $m_j\le n_{i+2}\le n-N$. For instance, $\widetilde{I}_1$ writes 
 \begin{align*}
    \widetilde{I}_1=\sum_{j=0}^M \zeta _ j  d_{n-N}^{k+j}
 \end{align*}
 and $\widetilde{I}_2$ is defined similarly. We see that 
$$d_n^k :=\widetilde{I}_1+\widetilde{I}_2$$ is uniquely defined in terms of the $d_m^l$'s for $m\le n-N$, $l\in \N$. 
Thus the sequence $(d_n^k)_{(n,k)\in \N^{2}}$ can be defined by induction on $n$ and the same estimates as before allow us to obtain \eqref{Gd}, see also Remark \ref{rk:Leibniz}.

{For the third part of Proposition \ref{prop10}, we prove by iteration on $k$ that $\partial_{x}^{k} \partial_{t}^{n}y  (0, \tau ) =d_{n}^k$ for all $n\in \N$. By assumption, the result is true for all $k=0,\dots,M-1$. We assume that the result is true until the rank $k+M-1$ and we prove it at rank $k+M$. Let $n\geq N$.
We know that we have
 \begin{align}
 \label{e:yI2}
     \partial _x^k \partial _t ^n y(0,\tau )=\sum_{j=0}^M \zeta _ j \partial _x ^{k+j} \partial _t ^{n-N} y 
(0, \tau) +I_2,
 \end{align}
where $I_2$ is defined by \eqref{KL2} and \eqref{KL3}. 
The $d_n^k$ have been defined by iteration on $n$ by the formula 
\begin{align}
\label{e:dnkItilde}
d_n^k=\sum_{j=0}^M \zeta_j d_{n-N}^{k+j}+\widetilde{I}_{2}
\end{align}
where $\widetilde{I}_{2}$ has been obtained by replacing $\partial _x^{l_j+i}\partial _t ^{m_j} y(0,\tau)$ by $d_{m_j}^{l_j+i}$ in the formula of $I_2$. Since it only involves some terms with $0\le i\le M-1$ and $0\leq l_j\leq k$, we have $l_j+i\leq k+M-1$ for these terms and the iteration property gives $\partial _x^{l_j+i}\partial _t ^{m_j} y(0,\tau)=d_{m_j}^{l_j+i}$. In particular, $\widetilde{I}_{2}=I_2$. So, \eqref{e:dnkItilde} can be written as
\begin{align*}
\zeta_{M}d_{n-N}^{k+M}=d_n^k-\sum_{j=0}^{M-1} \zeta _ j d_{n-N}^{k+j}-I_{2}.
\end{align*}
Again, for $0\leq j\leq M-1$, the iteration assumption gives $d_{n-N}^{k+j}=\partial _x ^{k+j} \partial _t ^{n-N} y(0,\tau)$ and $d_n^k=\partial _x ^{k} \partial _t ^{n}y(0,\tau)$. So, we obtain
\begin{align*}
\zeta_{M}d_{n-N}^{k+M}=\partial _x ^{k} \partial _t ^{n}y(0,\tau)-\sum_{j=0}^{M-1} \zeta _ j \partial _x ^{k+j} \partial _t ^{n-N} y(0,\tau)-I_{2}.
\end{align*}
After comparison with \eqref{e:yI2} and since $\zeta_{M}\neq 0$, we obtain $d_{n-N}^{k+M}=\partial _x ^{k+M} \partial _t ^{n-N} y(0,\tau)$. Since $n\geq N$ is arbitrary, it gives the result at step $k+M$.
}
\end{proof}
\begin{remark}
We note that even if we do not know {\em a priori} whether $Y_{0}$ will give rise to a solution, the algorithm is still well-defined. Our proof will show {\em a posteriori} that any initial data $Y_{0}$ which is analytic (with an appropriate radius) and small enough will produce a solution making this detail not so relevant. But this fact is not obvious at this moment of the proof.
\end{remark}
Note that at that moment, both Proposition \ref{prop1} and Proposition \ref{prop10} seem to give two relations between the space derivatives of $Y_0$ and the time derivatives of an eventual solution. If there exists a solution $y$ starting from $Y_0$ at time $t=0$, that relation should be unique (but this claim is not proved yet).\\

 The following result will show the existence of a solution. It will allow us clarifying the relation between the $d_{n,k}$ and the functions $J_n^k$ in Corollary \ref{cordnkGnk}. 
 There is likely a direct way to prove this relation, but it might be quite computational. The difference between Lemma \ref{lem1} and Proposition \ref{prop10} is only the order in which we apply time and space derivatives to the equation.


\begin{proposition}[Existence of solution without boundary condition]
\label{prop100}
Let $-\infty < t_1\le \tau \le t_2<+\infty$ and $f=f(x,y_0,y_1, ... ,y_{M-1} )$ be as in \eqref{AB1}-\eqref{AB2} with the coefficients $a_{\vec{p},r}$, $(\vec{p},r)\in \N ^{M+1}$,
satisfying \eqref{AB3}-\eqref{AB4}.  Assume in addition that $b_2> \hat R:=4N\lambda e^{(\lambda e)^{-1}}$.
Let $\widetilde{R}>\hat R$. 
Then there exists some number $\widetilde C>0$ such that 
for any $C\in (0,\widetilde C ]$ and any numbers $R_{L}$ with $\hat R<R_{L}<\min (\widetilde{R}, b_2)$  there exists a number $C''=C''(C,\widetilde{R},R_L)>0$ with 
$\lim_{C\to 0^+} C''(C,\widetilde{R},R_L)=0$ such that for any $Y_0\in ({\mathcal R} _{\widetilde{R},C})^N$, we can pick a function $y\in G^{1,\lambda}( [-1,1]\times [t_1,t_2] )$ 
satisfying  \eqref{A1} for $(x,t)\in [-1,1]\times [t_1,t_2]$ and 
\be
\label{F1}
\yt(x,\tau ) =Y_0(x)=\sum_{k=0}^\infty A_k \frac{x^k}{k!}, \quad \forall x\in [-1,1], 
\ee
and such that for all $t\in [t_1,t_2]$ 
\ba
\label{F2}
\Vert \partial _t^n \yx(0, t) \Vert _\infty \le C''   (n!)^{\lambda}\left(\frac{D|\zeta_M|^{1/M}}{ R_L}\right)^{n\lambda},
\ea 
with $D:=\lambda e^{(\lambda e)^{-1}}$.
\end{proposition}
\begin{proof}
We assume first that $|\zeta_M|=1$, dealing with the general case at the end of the proof. Note that the scaling in time affects only \eqref{F2}.\\

Let $\hat R:=4N\lambda e^{(\lambda e)^{-1}}$, we will need some intermediate radii $R$, $R'$, $R''$ with $\hat R<R_L<R''<R'<R <\min (\widetilde{R}, b_2)$. 
Pick $\widetilde C,C$ as in Proposition \ref{prop10}, and pick any $Y_{0}\in ({\mathcal R} _{\widetilde{R},C})^N$.
If a function $y$ as in Proposition \ref{prop100} does exists, then both sequences of numbers 
\begin{eqnarray*}
d_n^k &:=& \partial _t^n \partial _x^k y(0,\tau ), \quad n\in \N, \quad k\in \N 
\end{eqnarray*}
can be computed inductively in terms of the coefficients $A_k=\partial _x ^k Y_0 (0)$, $k\in \N$, according to Proposition \ref{prop10}, that is $(d_{n}^{k})_{(n,k)\in\N^{2}}=\Lambda^{\infty}(A_k)_{k\in \N}$. Note that the sequence $(d_n^k)_{(n,k)\in \N^{2}}$ can be defined in terms of the coefficients $A_k$'s, even if the existence of the 
solution $y$ is not yet established, according to  Proposition \ref{prop10} (2). 
Furthermore, it follows from Proposition \ref{prop10} that we have for some constant $C'=C'(C,R,R')>0$,
\begin{eqnarray*}
|d_n^k|&\le& C'\frac{(\lambda n+k)!}{ R^k R'^{\lambda n}}, \quad \forall n\in \N, \quad \forall k\in \{ 0,\ldots, M-1\}.
\end{eqnarray*}

Since $R''\in (\hat R, R')$, there exists some constant $P=P(R,R',R'') >0$ such that we have also
\[
|d_n^k|  \le  C'P\frac{(\lambda n)!}{  (R'')^{\lambda n}}, \quad \forall n\in \N, \quad  \forall k\in\{ 0,\ldots , M-1\} .
\]
The following lemma is a consequence of \cite[Proposition 3.6]{MRRreachable}. The proof that \cite[Proposition 3.6]{MRRreachable} implies Lemma \ref{lem4} will be done later.
\begin{lemma}
\label{lem4}Let $\lambda>1$. 
Let $(d_q)_{q\ge 0}$ be a sequence of real numbers such that  
\[
|d_q|\le C H^q (\lambda q)!\quad \forall q\ge 0  
\]
for some $H>0$ and $C>0$. Then for all $\tilde H>e^{e^{-1}}H$ there exists a function $f\in C^\infty(\R )$ such that 
\ba
f^{ (q) } (0) &=& d_q\quad \forall q\ge 0, \\
| f^{ (q) } (t) |&\le&  C \tilde H^q (\lambda q)! \quad \forall q\ge 0, \ \forall t\in \R .  
\ea 
\end{lemma}
Pick $H :=1/(R'')^{\lambda}$ and $H_L:=e^{e^{-1}}/(R_L)^{\lambda}$. Since $\hat R<R_L<R''$, we have $e^{e^{-1}}H<H_L<1/(4N\lambda)^{\lambda}$.
Then by Lemma \ref{lem4}, there exist $M$ functions $h_0,h_1,\ldots,h_{M-1} \in G^{\lambda}( [t_1,t_2] )$
such that for $k=0,\ldots ,M-1$,
\ba
h_k^{(n)} (\tau) &=& d_n^k, \quad n\ge 0\label{K1}\\
|h_k^{(n)} (t) | &\le&  C' P H_L^n (\lambda n)!, \quad n\ge 0,\ t\in [t_1,t_2]. \label{K3}
\ea 
It follows at once from Stirling's formula that $(\lambda n)!\le C_s \lambda^{\lambda n} (n!)^{\lambda}$ for some universal constant $C_s>0$, so that for  $k=0,\ldots,M-1$,
\ba
|h_k^{(n)} (t) | &\le&  C' P C_s(\lambda^{\lambda}H_L)^n (n!)^{\lambda}, \quad n\ge 0,\ t\in [t_1,t_2], \label{K33}
\ea 
Note that $\lambda^\lambda H_L <1/(4N)^{\lambda}$. So, if $C$ is sufficiently small, then $C'$ is as small as desired, and 
it follows then from Theorem \ref{thm2}  that we can pick a function $y\in G^{1,\lambda}([-1,1]\times [t_1,t_2])$ satisfying 
\eqref{C1} with $k_i:=h_i$ for $0\le i\le M-1$. In particular, for all $n\in \N$ and $k=0,\ldots ,M-1$, we have $\partial_t^n\partial _x^k y (0,\tau )=h_k^{(n)} (\tau) = d_n^k$.  Using the third Item of Proposition \ref{prop10}, we infer that $\partial_t^n\partial _x^k y (0,\tau )= d_n^k$ for $n\in \N$ and $k\in \N$. Moreover, we can check in the proof of Proposition \ref{prop10} (case $0\leq n\leq N-1$) that if $(d_{n}^{k})_{(n,k)\in\N^{2}}=\Lambda^{\infty}(A_k)_{k\in \N}$, then $d_n^k=A_k^n$ for $k\in \N$ and $0\le n\le N-1$. In particular, 
$\partial_t^n\partial _x^k y (0,\tau )= A_k^n= \partial _x^k y_0^n(0)$ for $k\in \N$ and $0\le n\le N-1$, where $Y_0=(y_0^0, y_0^1, .... , y_0^{N-1})$, and hence \eqref{F1} holds. 
Since $Y^x(0,t)=(h_0(t),\dots, h_{M-1} (t))$ by construction, the estimate \eqref{F2} follow from \eqref{K33}
with  $C'' := C'PC_s$ and $\lambda^{\lambda}H_L=\left(\lambda e^{(\lambda e)^{-1}}/R_L\right)^{\lambda}$. The proof of  Proposition \ref{prop100}  is complete for the case $|\zeta_M|=1$.\\

In the general case, assuming $\tau=0$ without loss of generality, we proceed as in Remark \ref{rkscaling} and define  $\widetilde{P}=|\zeta_{M}|^{-1}P$ 
and $\widetilde{f}=|\zeta_{M}|^{-1}f$ for which the result is proved for any interval in time. We have therefore a solution $\widetilde{y}$
 of $\partial_t^N \widetilde{y}=|\zeta_{M}|^{-1}P\widetilde{y}+|\zeta_{M}|^{-1}f(x,\widetilde{Y^x})$ on $[-1,1]\times [|\zeta_{M}|^{-1}t_1,|\zeta_{M}|^{-1}t_2]$
 with $\widetilde{Y}^t(x,0)=Y_0(x)$.
  Moreover, $\widetilde{y}$ satisfies \eqref{F2} with $|\zeta_M|=1$. 
 Now, we define $y(x,t):=\widetilde{y}(x,|\zeta_M|^{1/N}t)$ which is a solution of $\partial_t^N y=Py+f(x,Y^x)$ on $[-1,1]\times [t_1,t_2]$  with $Y^t(x,0)=Y_0(x)$. By scaling, $Y^x$ satisfies
\ba
\left\vert  \partial _t^n \yx(0, t) \right\vert \leq |\zeta_{M}|^{n/N}\left\vert  (\partial _t^n \widetilde{\yx})(0,|\zeta_{M}|^N t)
\right\vert\le |\zeta_M|^{n\lambda /M} C''   (n!)^{\lambda}\left(\frac{D}{ R_L}\right)^{n\lambda}.
\ea 
\end{proof}
\begin{proof}[Proof of Lemma \ref{lem4}]We want to apply \cite[Proposition 3.6]{MRRreachable} (stated below in Proposition \ref{propMRR})
 with the choice $a_0=1$ and  $k$ becoming $k-1$ so that $M_{q}=(\lambda q)!$ $a_k=\frac{(\lambda (k-1))!}{(\lambda k))!}=\Gamma(\lambda)^{-1}B(\lambda (k-1)+1, \lambda)=\Gamma(\lambda)^{-1}\int_0^1t^{\lambda (k-1)}(1-t)^{\lambda-1}dt$ for 
$k\ge 1$. All the terms being positive, we obtain for $p\ge 1$
\begin{eqnarray*}
\sum_{k>p}a_k&=& 
\Gamma(\lambda)^{-1}\int_0^1(1-t)^{\lambda-1}(\sum_{k>p} t^{\lambda (k-1)} )dt \\ 
&=&\Gamma(\lambda)^{-1}\int_0^1(1-t)^{\lambda-1}\frac{ t^{\lambda p}}{1-t^{\lambda}}dt \leq \Gamma(\lambda)^{-1}\int_0^1(1-t)^{\lambda-2}t^{\lambda p}dt \\
&=&\frac{\lambda p}{\Gamma(\lambda)(\lambda-1)}\int_0^1(1-t)^{\lambda-1}t^{\lambda p- 1}dt  \leq \frac{\lambda }{(\lambda-1)} p a_{p}<+\infty,
\end{eqnarray*}
where we have used twice $t^{\lambda}\leq t$ for $t\in [0,1]$ and $\lambda>1$, and performed  an integration by parts. In particular, the three conditions of Proposition \ref{propMRR} are fulfilled with $A:=\frac{\lambda }{\lambda-1}+1$ and $M_{q}=(\lambda q)!$. This completes the proof of Lemma \ref{lem4}. 
\end{proof}
For the convenience of the reader, we state the following proposition that we used before to construct the suitable Gevrey functions. 
\begin{proposition}[Proposition 3.6 of \cite{MRRreachable}]
    \label{propMRR}
   Pick any sequence $(a_q)_{q\in \N} $ satisfying
   \begin{itemize}
\item $1= a_{0}\geq a_{1}\geq a_{2}\geq \cdots >0$
\item $\sum_{k\geq 1}a_{k}<+\infty$
\item $pa_{p}+\sum_{k>p}a_{k}\leq Apa_{p}, \quad \forall p\geq 1$,
\end{itemize}
for some constant $A\in (0,+\infty )$. Let $M_q :=(a_0\cdots a_q)^{-1}$ for $q \geq 0$. Then for any sequence of real numbers $(d_q)_{q\geq 0}$ such that
$$|d_q|\le CH^q M_q,\quad  \forall q \ge 0 $$
for some $H >0$ and $C >0$, and for any $\widetilde{H}>e^{e^{-1}}H$, there exists a function $f \in  C^{\infty}(\R)$ such that
\begin{align*}
f^{(q)}(0)&= d_q \quad  \forall q \ge 0,\\
|f^{(q)}(x)|&\le  C \widetilde{H}^q M_q \quad \forall q \ge 0, \forall x \in \R.
\end{align*}
\end{proposition}
\begin{corollary}
\label{cordnkGnk}
Let $Y_0$ satisfying the assumptions of Proposition \ref{prop10} and let $(d_n^k)_{(n,k)\in \N^{2}}$ be the sequence introduced in  Proposition \ref{prop10} (2). Then we have the relationship
\bnan
\label{equaldG}
d_{n}^k=J_n^k(0,A_0,A_1 , ... ,A_{M [\frac{n}{N}] +k})
\enan
where the $J_n^k$ are the functions defined in \eqref{DDD1}.

Moreover, if $Y_{0}\in \Comp$ and if we set $D_{n}:=(d_{n}^0,d_{n}^1,\ldots,d_{n}^{M-1})$, then $BD_{n}=0$ for all $n\in \N$. 
\end{corollary}
\bnp
Let $y$ be a solution given by Proposition \ref{prop10}. Then \eqref{equaldG} holds since both sides of the equality agree with  $\partial_t^n\partial_x^k y (0,\tau ) $, according to \eqref{DDD1} for the right-hand side and to Proposition \ref{prop10} (2)  for the left-hand side. Moreover if $Y_0\in \Comp$, then   $BD_{n}=BJ_n=0$ by \eqref{comp}.
\enp
\begin{proposition}[Existence of solution with boundary condition]
\label{propexistcompat}
Consider the same assumptions and constants as in Proposition \ref{prop100}. Then, for any $Y^0\in \left({\mathcal R} _{\tilde R,C}\right)^N\cap \Comp$, we can find a solution of \eqref{W1a}-\eqref{W1b} for $(x,t)\in [-1,1]\times [t_1,t_2]$ satisfying \eqref{F1} and \eqref{F2}. 
\end{proposition}
\bnp
The proof is similar to those of Proposition \ref{prop100}. The modifications to ensure the boundary conditions are the following.

The sequence $(D_{n})_{n\ge 0}$ defined in Corollary \ref{cordnkGnk} satisfies $BD_{n}=0$ for all $n\in \N$. We can then proceed as in the proof of Proposition \ref{prop100}, replacing Lemma \ref{lem4} by Lemma \ref{lem4bis} (see below). The advantage of using Lemma \ref{lem4bis} is that the condition $BH_0(t)=0$ is satisfied by the 
function $H_0=(h_0, ..., h_{M-1})\in G^\lambda ([t_1,t_2])^M$ it provides. Then, using Theorem \ref{thm2} again with that boundary condition $H_0$, the equation \eqref{C1} gives $\yx(0,t)=H_0(t)$, so that the boundary condition $B\yx(0,t)=0$ is satisfied, as expected. This gives a solution of the system \eqref{W1a}-\eqref{W1b}. The conditions \eqref{F1} and \eqref{F2} are fulfilled for the same reasons as in Proposition \ref{prop100}.
\enp
\begin{lemma}
\label{lem4bis}
Let $(D_q)_{q\ge 0}$ be a sequence in $\C^M$ such that  
\bna
\Vert D_q\Vert _\infty &\le& C H^q (\lambda q)!\quad \forall q\ge 0,  \\
BD_q &=& 0\quad \forall q\ge 0
\ena
for some $H>0$ and $C>0$. Then for all $\tilde H>e^{e^{-1}}H$, there exists a function $F\in C^\infty(\R )^M$ such that 
\ba
F^{ (q) } (0) &=& D_q\quad \forall q\ge 0, \\
\Vert  F^{ (q) } (t) \Vert _\infty&\le&  C \tilde H^q (\lambda q)! \quad \forall q\ge 0, \ \forall t\in \R ,\\
BF^{ (q) } (t) &=&0\quad \forall q\ge 0, \ \forall t\in \R.  
\ea 
\end{lemma}
\bnp
Let $e_i\in\C^M$, $i=1,\cdots, dim(\ker(B))$, be the vectors of a basis of $\ker(B)$. In particular, we can write $D_q=\sum_i D_{q,i} e_i$. By assumption, the real sequence $(D_{q,i})_{q\in\N}$ satisfies the assumptions of Lemma \ref{lem4}, so that there are  some functions $f_i\in C^\infty(\R )$ satisfying $f_i^{ (q) } (0) = D_{q,i}$ and 
$| f_i ^{ (q) } (t) |\le  C \tilde H^q (\lambda q)!$ for all $q\ge 0$, $t\in \R$.  The function $F=\sum_i f_ie_i$ satisfies the requested properties. 
\enp
We also infer from the existence of solutions given by Proposition \ref{prop100} the following uniqueness result for the functions $J_{l}^{k}$.
\begin{lemma}
 \label{lmuniqJl}Let $l\in \N$. Then there exists some number $\e>0$ such that if two applications $J_{l},\widetilde{J}_{l}:[-1,1]\times (\R^{N})^{m(l)+1}\to \R^{M}$ satisfy \eqref{propJlintro} for any  smooth solution $y$ of \eqref{A1}, then they coincide on $[-1,1]\times B(0,\e)$. 
 In particular, if both functions are analytic, then they are equal.
 \end{lemma}
 \bnp
 Since \eqref{propJlintro} is assumed to be satisfied, it is sufficient to prove that for any \linebreak $(x_{0}, Y_{0}, Y_{1}, \cdots, Y_{m(l)})\in [-1,1]\times B(0,\e)$, there exists one solution of $y\in C^\infty ([-1,1]\times [t_1,t_2])$ solution of \eqref{A1} with $(Y_{0}, Y_{1}, \cdots, Y_{m(l)})=(\yt,\partial_{x}\yt(x_{0},\tau ),...,\partial_{x}^{m(l)}\yt(x_{0},\tau )))$. Thanks to Proposition \ref{prop100}, it suffices to find $Y^0\in ({\mathcal R} _{\tilde R,C})^N$ so that $(Y^{0}(x_{0}), \cdots, \partial_{x}^{m(l)}Y^{0}(x_{0}))=(Y_{0}, Y_{1}, \cdots, Y_{m(l)})$. This is simple analytic interpolation if $\e$ is chosen small enough with respect to $\tilde R,C$. 
 \enp
\section{Proofs of Theorem \ref{thm1} and Proposition \ref{propCompsym}}
\label{section4}
\begin{proof}[Proof of Theorem \ref{thm1}.]
Let $R>\hat R:=4N\lambda e^{(\lambda e)^{-1}}$ and let $\widetilde C$ be the constant given by Proposition \ref{prop100}. Let  $Y^0,Y^1\in ({\mathcal R}_{R,\widetilde C})^{N}\cap \Comp$. 
We infer from Proposition \ref{propexistcompat} applied with $[t_1,t_2] =[0,T]$ and $\tau =0$ (resp. $\tau =T$) the existence of 
two functions $\hat y,\tilde y \in G^{1,\lambda}([-1,1]\times [0,T])$ satisfying \eqref{W1a}-\eqref{W1b} 
and such that 
\[ \hat \yt(x, { 0}) =Y^0(x) \ \textrm{ and } \  \tilde \yt(x, { T} ) =Y^1(x), \quad \forall x\in [-1,1].\]
Let $\rho\in C^\infty (\R )$ be such that
\[
\rho (t) =\left\{
\begin{array}{ll}
1 & \textrm{ if } t \le \frac{T}{4} , \\[3mm]
0 & \textrm{ if } t \ge \frac{3T}{4} ,  
\end{array}
\right.
\]
and $\rho _{\vert [0,T]}  \in G^{\frac{\lambda+1}{2}}([0,T])$. (Note that $(\lambda + 1) / 2 >1$.) 
Let 
\begin{eqnarray*}
K_0(t)&=& \rho (t) \hat Y^{  x}  (0,t) + (1-\rho (t)) \tilde Y^{ x}  (0,t), \quad t\in [0,T].
\end{eqnarray*}
Then  $K_0\in G^\lambda([0,T])^{M}$ by \cite[Lemma 3.7]{MRRreachable}, and 
assuming $Y^0,Y^1\in ({\mathcal R}_{R,\widehat C})^{N}\cap \Comp$ with $0 <\widehat{C} < \widetilde{C}$, $\widehat{C}$ small enough, 
 we can assume 
that \eqref{AA} is satisfied.  It 
 follows then from Theorem \ref{thm2} that there exists a solution $y\in G^{1,\lambda}([-1,1]\times [0,T])$
of \eqref{C1}. Then $y$ satisfies \eqref{W1a}-\eqref{W1b} together with $\yt(x, T ) =Y^1(x)$ for $x\in [-1,1]$.  

Indeed, since $\rho (t)=0$ for $t>3T/4$, we have 
\begin{eqnarray*}
\partial _t ^n Y^{ x}(0, T)=K_0^{ (n) } (T ) &=&\partial _t^n \tilde Y^{ x}(0,T),\quad \forall n\in \N, 
\end{eqnarray*}
It follows then from Proposition \ref{prop1} that  $\partial _x ^nY^{t}(0,T)=\partial _x ^n \tilde Y^{t} (0,T)=\partial _x^n Y^1(0)$ 
for all $n\in \N$, and hence $Y^{t}(.,T)=Y^1$.  We can prove in the same way that $Y^t (.,0) = Y^0$. 
The proof of Theorem \ref{thm1} is achieved.\end{proof}

Let us now proceed to the proof of Proposition \ref{propCompsym} describing the compatibility set in cases where parity arguments can be used.
\begin{proof}[Proof of Proposition \ref{propCompsym}] 
We first consider the Dirichlet case. We will give the modifications of the proof for the Neumann case after.

Consider first the Dirichlet case when $B\yx(0,t)=0$ reduces to $\partial^{2j}_{x}y(0,t)=0$ for $2j\leq  M-1$. It means that, following the definition \eqref{comp} and denoting $J_{l}^{i}$ the 
$i$th component of the vector $J_{l}\in \R^{M}$, we have 
\begin{equation*}
\Comp=\left\{Y_{0}\in C^{\infty}([0,1])^{N};\quad J_{l}^{2j}(0,Y_{0},\partial_{x}Y_{0},...,\partial_{x}^{m(l)}Y_{0})_{x=0}=0,\quad  \forall 0\leq 2j\leq  M-1, \forall l\in \N\right\}
\end{equation*}
So, we need to show $\Comp=\widetilde{\Comp}$, where
\bna
\widetilde{\Comp}&:=&\left\{Y_{0}=(y_{0},y_{1},...,y_{N-1})\in C^{\infty}([0,1])^{N};\quad \partial^{2j}_{x}y_{l}(0)=0, \quad \forall j\in \N, \ \forall   l =0 , ... ,  N-1\right\}\\
&=&\left\{Y_{0}\in C^{\infty}([0,1])^{N};\quad \partial^{2j}_{x}Y_{0}(0)=0, \forall j\in \N\right\} .
\ena
We first prove that $\widetilde{\Comp}\subset \Comp$. 

The set $\widetilde{\Comp}$ is the set of smooth functions that admit a smooth odd extension to $[-1,1]$. We still denote $Y_{0}\in C^{\infty}([-1,1])^{N}$ this extension. We use the notation $(Y_{0})^{x}_{k}$ for the vector $(Y_{0})^{x}_{k}=(Y_{0},\partial_{x}Y_{0},...,\partial_{x}^{k}Y_{0})\in \E{k}{N-1}$. A vectorial variant of property \eqref{propI} is then
\bnan
\label{propIbis}
I((Y_{0})^{x}_{k})(x)=(Y_{0,-})^{x}_{k} (-x)
\enan
 where $Y_{0,-}$ is the reflected application $Y_{0,-}(x)=Y_{0}(-x)$. 
 
The derivatives at zero are not modified, so we need to prove that $J_{l}^{2j}(0,(Y_{0})^{x}_{m(l)})_{x=0}=0$ for this extension.
Using Lemma \ref{lmparHn} and property \eqref{propIbis}
\be
\label{JksymDir}
J_{l}^{k}(-x,-(Y_{0,-})^{x}_{m(l)}(-x))=J_{l}^{k}(-x,-I((Y_{0})^{x}_{m(l)}) (x))= (-1)^{k+1} J_{l}^{k}(x, (Y_{0})^{x}_{m(l)}(x)) .
\ee
But since $Y_{0}$ is odd, $Y_{0,-}=-Y_{0}$ and $(Y_{0,-})^{x}_{m(l)}=-(Y_{0})^{x}_{m(l)}$, which gives  
\[ J_{l}^{k}(-x,-(Y_{0,-})^{x}_{m(l)})(-x)))= J_{l}^{k}(-x, (Y_{0})^{x}_{m(l)}(-x)).\]
In particular, thanks to \eqref{JksymDir}, the function $x\mapsto J_{l}^{2j}(x, (Y_{0})^{x}_{m(l)}(x)) $ is odd and $J_{l}^{2j}(0, (Y_{0})^{x}_{m(l)}(0))=0$.\\

Next we prove that $\Comp\subset \widetilde{\Comp}$. Let $Y_0\in \Comp$. We prove by induction on $k$ the following equivalent fact: $I((Y_{0})^{x}_{k})=-(Y_{0})^{x}_{k}$ at $x=0$.\\

For $k\leq M-1$, we notice from the proof of Lemma \ref{lem1} that for $0\le l <N$, we have $H_{l}^{k}=0$ so that for $Y_{0}=(y_{0},y_{1},...,y_{N-1})$, we have 
$J_{l}^{k}(x,Y_{0},\ldots  ,\partial_{x}^{k}Y_{0})=\partial_{x }^{k}y_{l}$. So the assumption $J_{l}^{k}(x,Y_{0},\ldots , \partial_{x}^{k}Y_{0})_{x=0}=0$ for $k$ even,
 $k\leq M-1$ implies $\partial_{x }^{k}y_{l}=0$ for $k$ even, $k\leq M-1$. \\

Now, assume that $I((Y_{0})^{x}_{2k-1})=-(Y_{0})^{x}_{2k-1}$ at $x=0$ for some $k\in \N$ with $2k-1\ge M-1$. Write $2k=Mn+i$ with $0\leq i< M$ (necessarily even) and
 pick any  $l=Nn+j$, where $j$ is arbitrary with $0\leq j< N$. \\

By \eqref{Hsym}, since $i$ is even, we have $H_{l}^{i}(0,-I(Y))=- H_{l}^{i}(0, Y)$ for all $Y$. We have by the inductive hypothesis $I((Y_{0})^{x}_{Mn+i-1})=-(Y_{0})^{x}_{Mn+i-1}$ at $x=0$, so that $H_{l}^{i}(0,(Y_{0})^{x}_{Mn+i-1}(0))= -H_{l}^{i}(0, (Y_{0})^{x}_{Mn+i-1}(0))$, and hence $H_{l}^{i}(0, (Y_{0})^{x}_{Mn+i-1}(0))=0$. Now, using the definition \eqref{DDD1} of $J_{l}^{i}$ and the assumption $Y_0\in \Comp$ which gives $J_{l}^{i}(x, (Y_{0})^{x}_{Mn+i})_{x=0}=0$ (since $i$ is even), we obtain 
$P^n\partial_{x}^{i}y_{j}=0$ if we denote $Y_{0}=(y_{0},\ldots, y_{N-1})$. By the structure of $P$, this gives the result at step $2k=Mn+i$ since $0\leq j<N$ is arbitrary. This implies that the result is also true at step $2k+1$.

\medskip

For the Neumann case, we modify the proof as follows.

This time, we are in the case when $B\yx(0,t)=0$ reduces to $\partial^{2j+1}_{x}y(0,t)=0$ for $2j+1\leq  M-1$, and using \eqref{comp}, we have
\begin{equation*}
\Comp=\left\{Y_{0}\in C^{\infty}([0,1])^{N};\quad J_{l}^{2j+1}(0,Y_{0},\partial_{x}Y_{0},...,\partial_{x}^{m(l)}Y_{0})_{x=0}=0,\quad  \forall 0\leq 2j+1\leq  M-1, \forall l\in \N\right\} .
\end{equation*}
So, we have to show that $\Comp=\widetilde{\Comp}$ with
\bna
\widetilde{\Comp}
&:=&\left\{Y_{0}=(y_{0},y_{1},...,y_{N-1})\in C^{\infty}([0,1])^{N};\quad \partial^{2j+1}_{x}y_{l}(0)=0, \quad \forall j\in \N, \ \forall l=0 , ... , N-1\right\}\\
&=&\left\{Y_{0}\in C^{\infty}([0,1])^{N};\quad \partial^{2j+1}_{x}Y_{0}(0)=0, \forall j\in \N\right\} .
\ena
We first prove that $\widetilde{\Comp}\subset \Comp$. In this case, the set $\widetilde{\Comp}$ is the set of smooth functions that admit a smooth even extension to $[-1,1]$. So we need to prove that $J_{l}^{2j+1}(0,Y_{0},\partial_{x}Y_{0},...,\partial_{x}^{m(l)}Y_{0})_{x=0}=0$ for this extension.
Using the second part of Lemma \ref{lmparHn} and property \eqref{propIbis}
\be
\label{JksymNeu}
J_{l}^{k}(-x,(Y_{0,-})^{x}_{m(l)}(-x))=J_{l}^{k}(-x,I((Y_{0})^{x}_{m(l)})  (x))= (-1)^{k} J_{l}^{k}(x, (Y_{0})^{x}_{m(l)}(x)).
\ee
But since $Y_{0}$ is even, $Y_{0,-}=Y_{0}$ and $(Y_{0,-})^{x}_{m(l)}=(Y_{0})^{x}_{m(l)}$, which gives this time
 \[ J_{l}^{k}(-x,(Y_{0,-})^{x}_{m(l)}(-x))= J_{l}^{k}(-x, (Y_{0})^{x}_{m(l)}(-x)). \]
 In particular, thanks to \eqref{JksymNeu}, the function $x\to J_{l}^{2j+1}(x, (Y_{0})^{x}_{m(l)}(x))$ is odd and  \[J_{l}^{2j+1}(0,  (Y_{0})^{x}_{m(l)}(0))=0.\]

In order to prove that $\Comp\subset \widetilde{\Comp}$, we prove by induction on $k$ that for all $k\in \N$,  $I( (Y_0)_k^x ) = (Y_0)_k^x$ at $x=0$.

For $k\leq M-1$, we still have $H_{l}^{k}=0$ and the same arguments as in the Dirichlet case gives $\partial_{x }^{k}y_{l}=0$ for $k$ odd in the range we consider.\\ 

Assume that $I((Y_{0})^{x}_{2k})=(Y_{0})^{x}_{2k}$ at $x=0$ for some $k\in \N$ with $2k\ge
 M-1$. Write $2k+1=Mn+i$ with $0\leq i< M$ (necessarily odd), and pick $l=Nn+j$ where $j$ is arbitrary with $0\leq j< N$.\\ 

By \eqref{HsymNeum},
since $i$ is odd, we have $H_{l}^{i}(0,I(Y))= -H_{l}^{i}(0, Y)$ for all $Y$. But we have from the inductive hypothesis $I((Y_{0})^{x}_{Mn+i-1})=(Y_{0})^{x}_{Mn+i-1}$ at $x=0$, so that $H_{l}^{i}(0,(Y_{0})^{x}_{Mn+i-1})= -H_{l}^{i}(0, (Y_{0})^{x}_{Mn+i-1})$, and hence $H_{l}^{i}(0,(Y_{0})^{x}_{Mn+i-1})= 0$. Now, using the definition \eqref{DDD1} of $J_{l}^{i}$ and the assumption 
$J_{l}^{i}(x,(Y_{0})^{x}_{Mn+i})_{x=0}=0$ (since $i$ is odd), we obtain 
$P^n\partial_{x}^{i}y_{j}=0$ if $Y_{0}=(y_{0},\ldots, y_{N-1})$. Since $0\le j<N$ is arbitrary, this gives the result at step $2k+1$ and also at step $2k+2$.
\end{proof}

\section*{Appendix }
\section{A Lemma of complex analysis}
\begin{lemma}\label{lemmasets}
    Consider 
\[
{\mathcal B}_{R,C} :=\left\{  z:[-1,1]\to \C ;  \ \exists f\in H^\infty_R,\ \Vert f\Vert_{L^\infty(B(0,R))} \le C, \ f\Big|_{  [-1,1]}=z\right\}.
\]
Then, for any $1<r<R$ and  $C>0$

$${\mathcal{B}}_{R,C} \subset {\mathcal R}_{R,C}\subset {\mathcal B}_{r,C  ( 1-\frac{r}{R} )^{-1} }.$$
\end{lemma}

\begin{proof}
    For given $z\in {\mathcal B}_{R,C}$, if $f$ denotes its analytic extension to $B(0,R)$, writing 
 $f(\xi)=\sum_{n=0}^\infty \alpha _n \frac{\xi^n}{n!}$ for $|\xi|<R$, we have by Cauchy's formula that for any $n\in \N$ and any $r<R$:
\[
| \alpha _n | = |f^{(n)} (0)| =  \left\vert \frac{n!}{2\pi i} \int_{|\xi|=r} \frac{f (\xi)}{\xi^{n+1}}d\xi\right\vert \le \frac{n!}{r^n} \Vert f\Vert _{L^\infty ( B(0,R ))},  
\] 
and hence $| \alpha  _n |   \le \Vert f\Vert _{L^\infty (B(0,R))} \frac{n!}{R^n}$ by letting $r\to R^-$. On the other hand, if $z\in {\mathcal R}_{R,C}$ is given
by $z(\xi)=f(\xi):=\sum_{n=0}^\infty \alpha _n \frac{\xi^n}{n!}$ for $\xi\in [-1,1]$ and $1<r<R$, then for $|\xi|<r$ we have that 
$ | f (\xi) | \le C\sum_{n=0}^\infty ( \frac{r}{R} )^n =C(1-\frac{r}{R})^{-1}<\infty$.\\   

\end{proof}

ƒ

\subsection{Gevrey regularity of the solution of \eqref{C1} provided in Theorem \ref{thm2}}\label{s:A1} 

Assume that $f$ satisfies \eqref{AB1}-\eqref{AB4}, still under the assumption $|\zeta _M|=1$.
Let us show that $y\in G^{1,\lambda }([-1,1]\times [t_1,t_2])$. Let $L_0=L(s_0)=e^{\frac{1-s_0}{N}}L_1<\frac{e^{\frac{1}{N}}}{(4N)^{\lambda}}  $ where $s_0\in[0,1]$ and $L_1$ are given in the proof of Theorem \ref{thm2}. Then, we can pick some numbers $R_1,R_2$ such that $\frac{4M}{e^{\frac{1}{M}}}<R_1<R_2<\frac{\lambda }{L_0^{\frac{1}{\lambda}}}$.
Let us prove that there exists some constant $ Q>0$ such that \eqref{AAAAA} holds. 
To this end, picking any $\mu >M+1$, we prove by induction on $k\in \N$ that 
\ba
\vert \partial _x^k \partial _t ^n y(x,t) \vert &\le& C_k \frac{ (\lambda n+k)!}{R_1^kR_2^{\lambda n} (\lambda n+k+1)^\mu} 
\quad \forall (x,t)\in [-1,1]\times [t_1,t_2], \ \forall n \in \N , 
\label{TT1}
\ea
 with $\sup_{k\in \N} C_k <\infty$, the sequence $C_k$ being  nondecreasing.  
Let us start with $k\in \{ 0, ..., M-1 \}$. 

We already know that $y\in C^\infty ([-1,1]\times [t_1,t_2] ) $ and that $U=(y,\partial _x y , ..., \partial_x^{M-1}) \in C ( [-1,1] ,X_{s_0} )$  for some $s_0\in (0,1)$, 
the space $X_{s_0}$ being  defined in \eqref{PP0};
that is,   $U\in C( [-1,1],{\mathcal X}_{L_0} )$ with 
$L_0=L(s_0)=e^{r(1-s_0)}L_1 = e^{\frac{1-s_0}{N}} L_1 \le e^\frac{1}{N} L_1$. 
Thus, we have for some constant $C>0$ and for all $n\in \N$ and all $(x,t)\in [-1,1]\times [t_1,t_2]$
\begin{eqnarray*}
\vert \partial _x^k \partial _t ^{n+1} y(x,t) \vert &\le & C\, L_0^{| n -\frac{M-1-k}{\lambda }|}
 \Gamma (n+1- \frac{M-1-k}{\lambda })^{\lambda} (1+n)^{-2},
 \\
\vert \partial _x^k \partial _t ^{n+1} y(x,t) \vert &\le & C, 
 \\
\vert \partial _x^{M-1} \partial _t ^{n+1} y(x,t) \vert  &\le & C\,   L_0^{n} (n!)^\lambda (1+n)^{-2},
\end{eqnarray*}  
for  $0\le k<M-1, n\leq \left|\frac{M-k-1}{\lambda}\right|+1$.\\

We readily infer from Stirling's formula $\Gamma (x+1)\sim (\frac{x}{e})^x \sqrt{2\pi x}$ that 
$\Gamma (x+a) \sim \Gamma (x) x^a$ as $x\to \infty$, for any $a\in \R$, and that 
$(n!)^\lambda \sim (2\pi n)^\frac{\lambda -1}{2}  \lambda ^{-\frac{1}{2}} (\lambda n)! /
 \lambda^{\lambda n}$. It follows that for some constant $C>0$  
 \[
 | \partial _x^k\partial _t ^{n+1} y(x,t) | \le C L_0^n L_0^{-{\frac{M-1-k}{\lambda}}}  [n! (n+1) ^{-\frac{M-1-k}{\lambda}} ]^\lambda (n+1)^{-2}
 \le C L_0^n \frac{(\lambda n)!}{\lambda ^{\lambda n}}  (n+1)^{-({ M-1-k})} n ^\frac{\lambda -1}{2}\]
 for $0\le k\le M-1$.  
 Thus there are some positive constants $C_k$, $0\le k\le M-1$, such that 
 \eqref{TT1} holds, provided that 
$R_2< \lambda / L_0^\frac{1}{\lambda}$.\\

Assume now that \eqref{TT1} is true for $k\in \{ 0,..., l+M -1 \} $ for some $l\in \N $. Let us show that \eqref{TT1} is true for $k=l+M$; that is, for all $n\ge 0$ and all $(x,t)\in [-1,1]\times
[t_1,t_2]$
\[
\left\vert 
\partial _x ^{l+M} \partial _t ^n y (x,t)
\right\vert 
\le C_{l+M} \frac{(\lambda n + l+M)! }{R_1^{l+M}R_2^{\lambda n}  (\lambda n + l  + M + 1)^\mu},
\] 
for some constant $C_{l+M}>0$. Since $|\zeta _M|=1$, using \eqref{C1}, we have that 
\begin{eqnarray*}
\vert \partial _x ^{l+M} \partial _t ^n y\vert 
&=& \vert \partial _x^l \partial _t^n (\zeta _M\partial _x^M y)\vert \\
&=& \vert \partial _x^l \partial _t^n (\partial _t^N y - \sum_{j=0}^{M-1} \zeta _j\partial _x ^jy - 
f(x,y , \partial _x y , ..., \partial _x^{M-1}y )\vert \\
&\le& \vert \partial _x^l \partial _t ^{n+N} y \vert + \vert \partial _x^l\partial _t^n 
(\sum_{j=0}^{M-1} \zeta _j \partial _x^j y )\vert + 
\vert \partial _x^l \partial _t^n f(x,y, \partial _x y , ..., \partial _x ^{M-1} y )\vert \\
&=:& I_1 + I_2 + I_3. 
\end{eqnarray*}  
Then using directly the iteration assumption and $\lambda N=M$, we have
\[
I_1 \le C_l \frac{(\lambda n+\lambda N +l)!}{R_1^{l}R_2^{\lambda n+\lambda  N } (\lambda n+\lambda N+ l + 1)^\mu } 
=C_{l} \left( \frac{R_1}{R_2}\right) ^M
 \frac{(\lambda n + l+M)! }{R_1^{l+M}R_2^{\lambda n}  (\lambda n + l  + M + 1)^\mu} \cdot
 \]
On the other hand, we have that 

\begin{eqnarray*}
I_2 &\le & \sum_{j=0}^{M-1} 
\vert \zeta _j\vert \, \vert \partial _x ^{l+j} \partial _t^n y\vert \\
&\le& \sum_{j=0}^{M-1} |\zeta _j| C_{l+j} \, \frac{(\lambda  n + l + j ) !}{R_1^{l+j} R_2^{\lambda n} 
(\lambda n + l + j+ 1)^\mu} \\
&\le& \left( 
\sum_{j=0}^{M-1} |\zeta _j | C_{l+j} \frac{R_1^{M-j}}{(\lambda n + l+ j+ 1) \cdots (\lambda n +  l + M )}  
\left( \frac{\lambda n + l + M + 1}{ \lambda n  + l + j + 1} \right) ^\mu    
\right) 
\, \\
&&  \times \frac{(\lambda n + l+M)! }{R_1^{l+M}R_2^{\lambda n}  (\lambda n + l  + M + 1)^\mu} \cdot
\end{eqnarray*}
Finally, as in the proof of Proposition \ref{prop10} (see estimate \eqref{E3} iterating Lemma \ref{lem3}), we have that for some positive constant
 $\overline{C}$
\begin{eqnarray*}
I_3
&\le&
\left\vert  
\partial _x^l\partial _t^n \sum_{\vec{p}\ne 0} A_{\vec{p}} y^{p_0}(\partial _x y)^{p_1} 
\cdots (\partial _x^{M-1}y)^{p_{M-1}} 
\right\vert\\
&\le& \sum_{\vec{p}\ne 0} K^{|\vec{p}|} \frac{\overline{C}}{b^{ | \vec{p} |}}
\frac{(\lambda n+l+M-1)!}{ R_1^{{ l}} R_2^{\lambda n} (\lambda n+l+{1} )^\mu}
C_{l+M-1}^{|\vec{p}|}\prod_{j=0}^{M-1}\frac{1}{(R_1^j)^{p_j}}
\end{eqnarray*}
Note that $R_1>1$. If, for some constant $\delta \in (0,1)$, we have  
\be
\label{YY1bis}
\frac{C_{l+M-1}K}{b} \le \delta , 
\ee
this yields 
\begin{eqnarray*}
I_3 &\le& { \delta} \overline{C}  \frac{(\lambda n+l+M-1)!}{R_1^lR_2^{\lambda n} (\lambda n+l+{ 1} )^\mu}
\left( \frac{1}{1-\delta} \right) ^{M} \\
&\le&  \frac{ { \delta} \overline{C} R_1^M}{ (\lambda n + l +  M)(1-\delta )^M}
\left( \frac{\lambda n + l+ M +1}{\lambda n + l + { 1} }\right) ^\mu
 \frac{(\lambda n + l+M)! }{R_1^{l+M}R_2^{\lambda n}  (\lambda n + l  + M + 1)^\mu} 
\end{eqnarray*}
It follows that 
\be
\vert \partial _x ^{l+M}\partial _t ^n y (x,t)\vert \le C_{l+M} \frac{ (\lambda n+l+M) ! }{R_1^{l+M} R_2^{\lambda n} (\lambda n + l + M +1 )^\mu},  
\label{TT8}
\ee
with 
\begin{multline*}
C_{l+M} :=\max \left(C_{l+M-1},C_{l} \left( \frac{R_1}{R_2}\right) ^M \right.\\
 + \left( 
\sum_{j=0}^{M-1} |\zeta_j | C_{l+j} \frac{R_1^{M-j}}{(\lambda n + l+ j+ 1) \cdots (\lambda n +  l + M )} 
\left( \frac{\lambda n + l + M + 1}{ \lambda n  + l + j + 1} \right) ^\mu    \right.
 \\
\left.+ 
 \frac{ { \delta} \overline{C} R_1^M}{ (\lambda n + l +  M)(1-\delta )^M}
\left( \frac{\lambda n + l+ M +1}{\lambda n + l + { 1}}\right) ^\mu \right)
\cdot \label{TT10}
\end{multline*}

Then, using the fact that $R_1<R_2$, { if $\frac{C_j K}{b} \le \delta$ for $j=0,1, ... ,l+M-1$}, then 
 $\frac{C_{l+M}K}{b} \le \delta$ provided that $l$ is large enough, say $l\ge l_0$. 
 It is then sufficient to impose that 
 \[
 \max(C_0, ..., C_{l_0+M-1})\le \frac{\delta b}{K},
 \]
 and this is the case provided that the constant $C$ in \eqref{AA} is small enough. 
 \section{On the complex Ginzburg-Landau equation}
 \label{s:GLproof}
\begin{theorem}
\label{t:exGL}
    Theorem \ref{thm2} holds true for the complex Ginzburg-Landau equation.
\end{theorem}
 \bnp
 The fact that the equation is complex-valued does not change the proof. The only slight difference is for Lemma \ref{lemmeF} where the nonlinearity contains some conjugate. The proof is even simpler since the sum is finite. We give a simpler proof for the convenience of the reader. In that case, $M=2$, $N=1$ and $\lambda=2$. If $U=(u_0,u_1)\in L^\infty (K)^2$, and $F(x,U)=\left( 
\begin{array}{c}
0\\
-e^{i\varphi }|u_0|^2u_0
\end{array}
\right) $ then 
 \begin{eqnarray*}
\Vert F(x,U)-F(x,V) \Vert _{X_{s'}}
&=&
\left\| 
\left( 
\begin{array}{c}
0\\
|u_0|^2u_0-|v_0|^2v_0
\end{array}
\right) 
\right\| _{X_{s'}}\\
&=& e^{-\tau (1-s') } \Vert |u_0|^2u_0-|v_0|^2v_0\Vert _{L(s')}\\
&=&e^{-\tau (1-s') } \Vert (u_0-v_0)(u_0+v_0)\overline{u_0}+v_0^2(\overline{u_0}-\overline{v_0})\Vert _{L(s')}\\
&\leq &\frac{3}{2} e^{-\tau (1-s') } \Vert u_0-v_0\Vert _{L(s')}  \left(\Vert u_0\Vert _{L(s')}^2 +\Vert v_0\Vert _{L(s')}^2\right)  .\end{eqnarray*} 
We used the algebra property of Lemma \ref{algebre} and the fact that the norm is invariant by conjugation. Using \eqref{Gacroissance+} and 
\eqref{gainXs}, we get, for a constant $C$ depending on $L_1$, $M$ and $N$,  
$$\Vert u_0\Vert _{L(s')}\leq C\Vert u_0\Vert _{L(s'),\frac{M-1}{\lambda}}\leq C\nor{U}{{\mathcal X}_{L(s')}}=Ce^{\tau (1-s')} \nor{U}{X_{s'}} 
\leq Ce^\tau e^{-\tau (s-s')} \Vert U\Vert _{X_s}.$$
The same estimate is true for $u_0-v_0$, and therefore we obtain
 \begin{multline*}
\Vert F(x,U)-F(x,V) \Vert _{X_{s'}}
\leq\\
C^3 e^{-\tau (s-s') }e^{3\tau } \nor{U-V}{X_s}\left(\nor{U}{X_s}^2+ \nor{V}{X_s}^2\right)\leq \frac{C^3e^{-1}e^{3\tau }}{\tau (s-s')}D^2\nor{U-V}{X_s},
\end{multline*} 
where we have used \eqref{estimexp}. For fixed $\tau$, it can be made arbitrarily small when $D$ is chosen small enough. The proof finishes the same way for the existence of the solution. Concerning the estimates given in Section \ref{s:A1}, the only difference concerns the term $I_3$ that becomes $I_3=|e^{i\varphi}\partial _x^l\partial _t^n (y^2\overline{y}) |$.  In this part of the proof the  induction argument  \eqref{TT1} is valid for $k\in \{ 0,..., l+M -1 \} $ for some $l\in \N $. The derivatives of $\overline{y}$ have the same bounds as those of $y$ in \eqref{TT1}, namely 
\ba
\vert \partial _x^k \partial _t ^n \overline{y}(x,t) \vert &\le& C_k \frac{ (2 n+k)!}{R_1^kR_2^{2 n} (2 n+k+1)^\mu} 
\quad \forall (x,t)\in [-1,1]\times [t_1,t_2], \ \forall n \in \N . 
\ea

We can apply similarly Lemma \ref{lem3} twice to get
\begin{eqnarray*}
I_3
&=&
\left\vert  
\partial _x^l\partial _t^n (y^2\overline{y})
\right\vert\le K^2 C_l^3\frac{(2 n+l)!}{ R_1^{l} R_2^{2 n} (2 n+l+ 1 )^\mu}\le\widetilde{\beta}_{l+2}C_l\frac{(2 n+l+2)!}{ R_1^{l+2} R_2^{2 n} (2 n+l+2+ 1 )^\mu}
\end{eqnarray*}
with $\widetilde{\beta}_{l+2}= \sup_{n\in\N}\frac{ K^2 C_l^2R_1^2}{(2 n+l+1)(2 n+l+2)}\frac{(2 n+l+2 +1 )^\mu}{(2 n+l+1 )^\mu}\leq \frac{ K^2 C_l^2R_1^2}{(l+1)(l+2)}3^\mu$. The rest of the estimate being the same, we can make the $\widetilde{\beta}_{l+2}$ arbitrarily small in a similar way. This completes the inductive step. 
 \enp
 \begin{proposition}
       Proposition \ref{prop10} holds true for the complex Ginzburg-Landau equation.
 \end{proposition}
 \bnp
The reconstruction is exactly the same working in $\C$ instead of $\R$. The modifications of the estimates of the nonlinear term are done in the same way as in Theorem \ref{t:exGL}, noticing that $\overline{y}$ satisfies the same estimates as $y$.
 \enp
 \begin{proof}[Proof of Theorem \ref{thmGLDir} and \ref{thmGLN}]
This is the same as before with $\lambda=2/1=2$. It only remains to check the condition about the non-linearity. We have $f(x,y_0,y_1)=e^{i\varphi} |y_0|^2y_0$. It satisfies $f(-x,-y_{0}, y_{1})=-e^{i\varphi} |y_0|^2y_0= -f(x , y_{0} , y_{1})$ which is  condition \eqref{WWimpair} for system  \eqref{GLDir}, 
 and $f(-x,y_{0},- y_{1})=e^{i\varphi} |y_0|^2y_0= f(x , y_{0} , y_{1})$ which is condition \eqref{WWpair} for system \eqref{GLN}.
\end{proof}
\section*{Acknowledgements}
Part of this work was done while the first author was visiting the Universidad del Valle at Cali. He would like to thank that institution for 
its hospitality. The second author 
would like to express her gratitude to Universidad del Valle, Department of Mathematics for all the facilities used during the realization of this work. The first author was supported by ANR project ISDEEC (ANR-16-CE40-0013). The third author was supported by the ANR project CHAT (ANR-24-CE40-5470). The first and third authors were supported by the COFECUB project COCECUB-20232505780P.   


\end{document}